\documentclass[a4paper,10pt]{article}
\usepackage{mathtools}
\usepackage{geometry}
\usepackage[pdftex]{graphicx}
\usepackage{amsmath}
\usepackage{amsfonts}
\usepackage{amssymb}
\usepackage{amsthm}
\usepackage{epstopdf}
\usepackage{csquotes}
\usepackage{color}
\usepackage{authblk}
\usepackage[percent]{overpic}
\usepackage{url}
\numberwithin{equation}{section}
\newtheorem{theorem}{Theorem}
\newtheorem{remark}{Remark}

\newtheorem{lemma}[theorem]{Lemma}

\newtheorem{proposition}[theorem]{Proposition}
\newtheorem{definition}[theorem]{Definition}

\numberwithin{theorem}{section}
\usepackage{algorithm}
\usepackage{algorithmic}
\usepackage{tcolorbox}
\usepackage{enumerate}
\usepackage[title]{appendix}
\usepackage{cite}
\usepackage{pdfpages}
\usepackage{authblk}
\usepackage{titling}
\usepackage[colorlinks]{hyperref}
\usepackage[capitalise]{cleveref}
\usepackage{cite}

\newcommand{\ket}[1]{|{#1}\rangle}
\newcommand{\bra}[1]{\langle{#1}|}

\DeclareMathOperator{\Tr}{Tr}
\DeclareMathOperator{\ran}{ran}

\DeclarePairedDelimiter\ceil{\lceil}{\rceil}

\newcommand{\de}{\ensuremath{\partial}}
\newcommand{\dee}{\ensuremath{\textrm{d}}}

\newcommand{\inty}[4]{\ensuremath{ \int_{#1}^{#2} \! #3 \, \dee#4 }}

\newcommand{\field}[1]{\mathbb{#1}}
\newcommand{\ip}[2]{\ensuremath{ \left< \left. #1 \right| #2 \right> } }

\renewcommand{\vec}[1]{\boldsymbol{#1}}
\let\oldhat\hat
\renewcommand{\hat}[1]{\oldhat{\boldsymbol{#1}}}

\begin{document}
\title{Computing spectral properties of topological insulators without artificial truncation or supercell approximation}
\author{Matthew J. Colbrook\thanks{Department of Applied Mathematics and Theoretical Physics, University of Cambridge and Centre Sciences des Données, Ecole Normale Supérieure, Paris.\\\hspace*{1.8em}Corresponding author: m.colbrook@damtp.cam.ac.uk}, Andrew Horning\thanks{Center for Applied Mathematics, Cornell University}, Kyle Thicke\thanks{Department of Mathematics, Technical University of Munich}, Alexander B. Watson\thanks{Mathematics Department, University of Minnesota Twin Cities, MN, USA.}}

\date{}

\maketitle

\begin{abstract}
Topological insulators (TIs) are renowned for their remarkable electronic properties: quantised bulk Hall and edge conductivities, and robust edge wave-packet propagation, even in the presence of material defects and disorder. Computations of these physical properties generally rely on artificial periodicity (the supercell approximation), or unphysical boundary conditions (artificial truncation). In this work, we build on recently developed methods for computing spectral properties of infinite-dimensional operators. We apply these techniques to develop efficient and accurate computational tools for computing the physical properties of TIs. These tools completely avoid such artificial restrictions and allow one to probe the spectral properties of the infinite-dimensional operator directly, even in the presence of material defects and disorder. Our methods permit computation of spectra, approximate eigenstates, spectral measures, spectral projections, transport properties, and conductances. Numerical examples are given for the Haldane model, and the techniques can be extended similarly to other TIs in two and three dimensions.
\end{abstract}


\section{Introduction}

Topological insulators (TIs) are materials with remarkable electronic properties\footnote{For simplicity of presentation, in this work we treat topological insulators as synonymous with two-dimensional Chern insulators: two-dimensional topological insulators whose bulk topology is measured by the Chern number (the classification of topological insulators by dimension and symmetry class was given by \cite{doi:10.1063/1.3149495,doi:10.1063/1.3149481}, for reviews, see \cite{2010HasanKane,Qi2011,Moore2010,BernevigHughes,2013FruchartCarpentier}). Our methods are not fundamentally restricted to two dimensions or Chern insulators.}. The bulk Hall and edge conductances of a TI are quantised precisely, even in the presence of defects and disorder (see \cite{1993Hatsugai,Qi2011,2018GrafShapiro,2010HasanKane,Moore2010,BernevigHughes,2005ElgartGrafSchenker,2013GrafPorta,2002KellendonkRichterSchulz-Baldes,1999Schulz-BaldesKellendonkRichter,2004KellendonkSchulz-Baldes_2,2016ProdanSchulz-Baldes,Bellissard1994,Avron1994,Bestwick2015,Chang2013,Chang2015}
 and the references therein). The remarkable robustness of these physical quantities has generated huge interest in TIs for potential industrial applications such as spintronics, quantum computing, and the ``topological transistor'' \cite{2010HasanKane,Moore2010,BernevigHughes,Qi2011,Gilbert2021}. The importance of topological insulators was confirmed by the award of the 2016 Nobel prize to Thouless, Haldane, and Kosterlitz for foundational work on topological phases of matter.

The edge currents of TIs are mediated by electronic states localised at edges known as edge states. The robustness of the edge conductance of a TI can be seen at the level of localised wave-packets formed from these edge states, which snake around corners and defects of the edge, even in the presence of disorder \cite{1982Halperin,2010HasanKane,Qi2011,Moore2010,BernevigHughes,Michala2021,Buchendorfer2006,Bols2021,1999Schulz-BaldesKellendonkRichter,Frohlich2000,2002KellendonkRichterSchulz-Baldes,Drouot2021,Ludewig2020,bal2021edge}. The behaviour of these wave-packets has spurred interest in building photonic and acoustic devices which mimic topological insulators for wave-guiding applications \cite{2009WangChongJoannopoulosSoljacic,2013Rechtsmanetal,2015SusstrunkHuber,Tauber2021,2019Lee-ThorpWeinsteinZhu,2017DelplaceMarstonVenaille}.

The significant interest in TIs for industrial applications makes finding good numerical methods for computing their electronic properties essential. Non-interacting electrons in TIs are typically modelled by spatially discrete Schr\"odinger equations known as tight-binding models. In order to account for the vast number of ionic cores in typical materials, the bulk of the material is generally treated as extending infinitely in all directions. In contrast, the edge of the material is modelled as a truncation of the infinite bulk model in one direction. The infinite extent of these models, coupled with non-periodic defects and disorder in realistic models, makes the numerical computation of the electronic properties of TIs very challenging. 

In many scenarios, existing methods for computing electronic properties of TIs are unsatisfactory because they rely on imposing either artificial periodicity (supercell approximation), or unphysical boundary conditions (artificial truncations), or some combination of the two (see, for example, the package \verb|PythTB| \cite{vanderbilt_2018}). As far as we are aware, supercell approximation has not been rigorously justified in this context. In any case, it cannot be used in the direction transverse to an edge. At the same time, artificial truncation leads immediately to spectral pollution and spurious edge states \cite{doi:10.1137/19M1282696}.

This paper applies recently developed methods for rigorously and efficiently computing spectral properties of infinite-dimensional operators \cite{Colbrook2019,colbrook_IMA_LT,colbrook2019b,colbrook2019computing,colbrook2020,colbrook2021semigroups,colbrook2020foundations} to the problem of computing the electronic properties of TIs. We do this \emph{without assuming model periodicity} and \emph{without introducing artificial truncations.} From a numerical analysis perspective, the majority of methods that deal with spectra in infinite dimensions are of a ``truncate-then-solve'' flavour. A truncation/discretisation of the operator is adopted, possibly taking advantage of periodicity in suitable directions (e.g., supercell method), and methods for computing the eigenvalues of a finite matrix are used. In contrast, we adopt a ``solve-then-discretise'' approach. The ``solve-then-discretise'' paradigm has also recently been applied to other spectral problems \cite{horning2020feast,webb2021spectra,doi:10.1137/19M1282696}, extensions of classical methods such as the QL and QR algorithms \cite{webb_thesis,colbrook2019infinite}, Krylov methods \cite{gilles2019continuous}, and the computation of resonances \cite{benartzi2020computing,Jonathan_res}. We are concerned with the following four types of spectral computations for infinite-dimensional operators:
\begin{itemize}
    \item[(P1)] Computing \textbf{spectra} with error control and computing \textbf{approximate eigenstates}\footnote{Where the spectrum is discrete, these approximate eigenstates are guaranteed to converge to exact eigenstates. This is sufficient to compute exact edge states of periodic TIs; see \cref{sec:edge_conductance}.} discussed in \cref{Num_spectra}, where we use the method of \cite{Colbrook2019}.
	\item[(P2)] Computing \textbf{spectral measures}, discussed in \cref{Num_spec_meas}, where we use the method of \cite{colbrook2019computing,colbrook2020}.
	\item[(P3)] Computing \textbf{spectral projections}, discussed in \cref{Num_spec_proj}, for which we propose a new efficient method building on the ideas in \cite{colbrook2019computing,colbrook2020}.
	\item[(P4)] Computing \textbf{transport properties} (and more generally the \textbf{functional calculus}), discussed in \cref{Num_FC}, where we use the method of \cite{colbrook2021semigroups}.
\end{itemize}

With this set of computational tools in hand, we focus on computing the following physical properties of TIs:
\begin{itemize}
    \item[(PA)] \textbf{Bulk} and \textbf{edge conductances} of non-periodic TIs, discussed in~\cref{sec:bulk_conductance,sec:edge_conductance}, with results shown in \cref{sec:res:conductance}.
    \item[(PB)] \textbf{Edge states} of periodic TIs and their \textbf{dispersion relations}, discussed in~\cref{sec:edge_conductance}, with results shown in \cref{sec:res:11}.
    \item[(PC)] \textbf{Approximate edge states} and \textbf{edge wave-packets} of non-periodic TIs, and their \textbf{spectral measures}, discussed in~\cref{sec:edge_conductance}, with results shown in \cref{sec:res:non-periodic,sec:res:WPs}.
    \item[(PD)] \textbf{Dynamics} of edge wave-packets of non-periodic TIs, discussed in~\cref{sec:edge_conductance}, with results shown in \cref{sec:res:time_prop}.
\end{itemize}
For brevity, we have restricted ourselves to reporting results for the Haldane model \cite{1988Haldane}. However, our methods are not fundamentally restricted to this model and allow the computation of the electronic properties of more general TIs in two and three dimensions.

The paper is organised as follows. First, we discuss the motivation and idea of our methods in \cref{sec:mtvhh}. We then give a summary of the algorithms in \cref{num_analysis_user_manual}. The Haldane model and its physics are discussed in \cref{sec:Haldane}. Results are reported in \cref{sec:resultshjhj}. We also provide appendices of pseudocode, proofs and further details on the physical model. 

Finally, code for our paper can be found at \texttt{https://github.com/SpecSolve/SpecTB}~\cite{SpecTB_code}. We hope this paper can also act as a user manual for those who wish to apply these techniques to their problems of interest.

\section{Motivation for the numerical methods}
\label{sec:mtvhh}

This work presents algorithms for discrete models commonly used to model electrons in materials known as tight-binding models. For such models, the Hilbert space is always isomorphic to the space of square summable sequences $l^2(\field{N})$. Basis vectors correspond to atomic sites, perhaps with additional internal degrees of freedom such as sublattice label and spin. In this basis, the Hamiltonian is described by an infinite Hermitian matrix $\widetilde{H}=\{\widetilde{H}_{ij}\}_{i,j\in\mathbb{N}}$, which we assume to be sparse, or finite range, i.e. finitely many non-zero entries in each column.\footnote{We use the notational $\widetilde{\cdot}$ to distinguish between the abstract Hamiltonian $H$ and its representation as an infinite matrix $\widetilde{H}$ on $l^2(\mathbb{N})$. This notation is to avoid confusion later on, where we write $H$ in terms of infinite matrices acting on Hilbert spaces different to $l^2(\mathbb{N})$.} After a suitable ordering of the sites (e.g., by positional radius from an origin), there exists a function $f:\mathbb{N}\rightarrow\mathbb{N}$ such that $\widetilde{H}_{ij}=0$ if $i>f(j)$. Thus $f$ describes the sparsity of $\widetilde{H}$. We, therefore, describe the algorithms below for infinite sparse matrices representing Hermitian Hamiltonians. The restriction to sparse tight-binding models is not fundamental: for non-sparse matrices and even non-Hermitian operators, see~\cite{Colbrook2019}. Extensions of the algorithms to unbounded operators and partial differential operators can be found in the relevant papers~\cite{colbrook2019b,colbrook2019computing,colbrook2020,colbrook2021semigroups}. In this paper, \textbf{our aim is to compute the spectral quantities of interest from the infinite matrix $\widetilde{H}$.}

\subsection{Rectangular, as opposed to square, truncations}
\label{sec:bijbibav}

In the context of this paper, the algorithms we use rely on rectangular, as opposed to square, truncations of $\widetilde{H}$. Since this may be an unfamiliar approach to the reader, we first explain the general idea before discussing the algorithms. See also \cite{colbrook_IMA_LT} for a pedagogical description. In what follows, we use $\mathrm{Sp}(\cdot)$ to denote the spectrum. 

\begin{figure}
\centering
\begin{tabular}{c|c|c}
(a) Infinite Lattice & (b) Finite Truncation & (c) Our Method \\
\hline
\begin{overpic}[width=0.30\textwidth,clip,trim={0mm 0mm 0mm -20mm}]{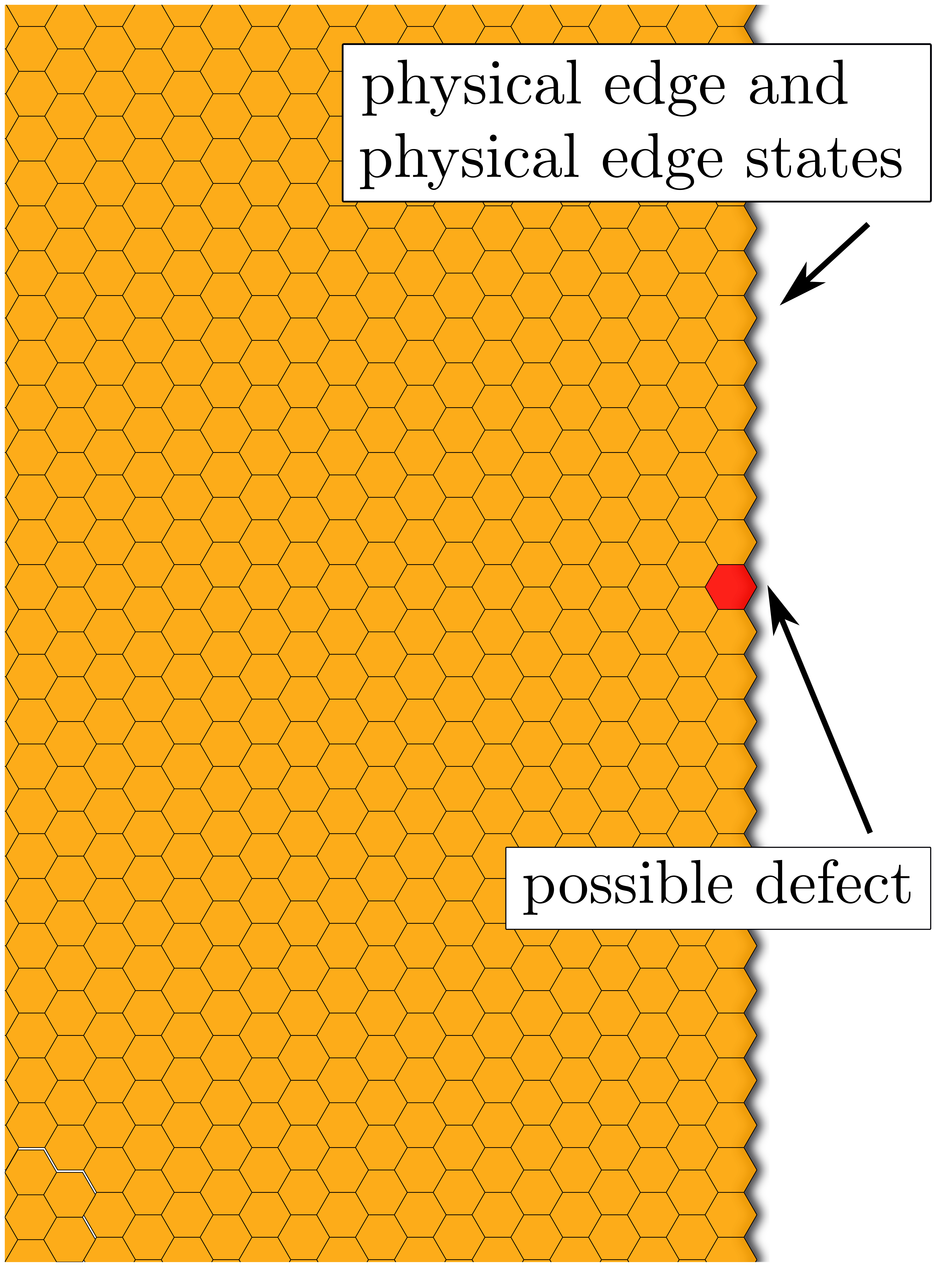}
\end{overpic}&
\begin{overpic}[width=0.30\textwidth,clip,trim={0mm 0mm 0mm -20mm}]{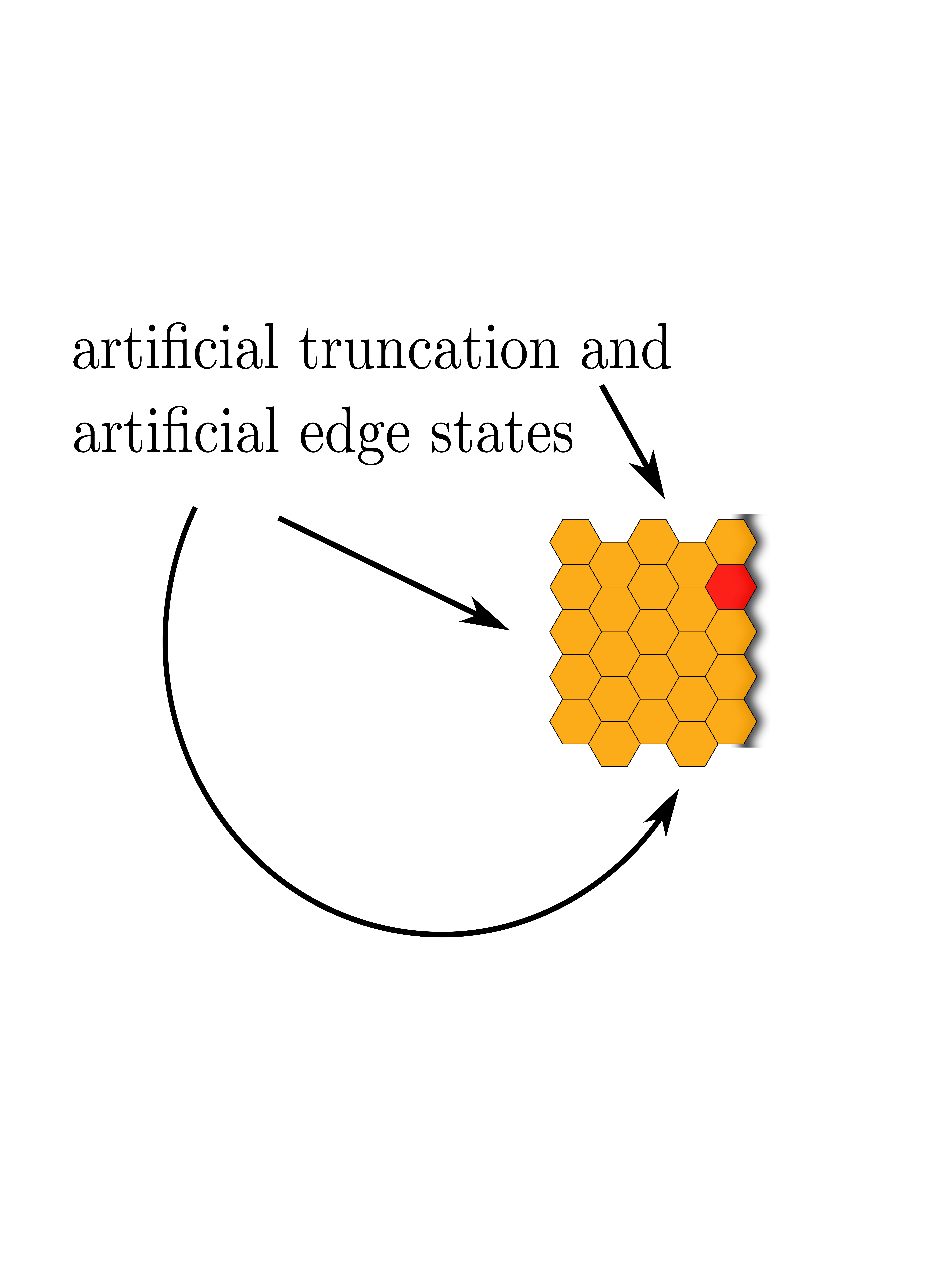}
\end{overpic}&
\begin{overpic}[width=0.30\textwidth,clip,trim={0mm 0mm 0mm -20mm}]{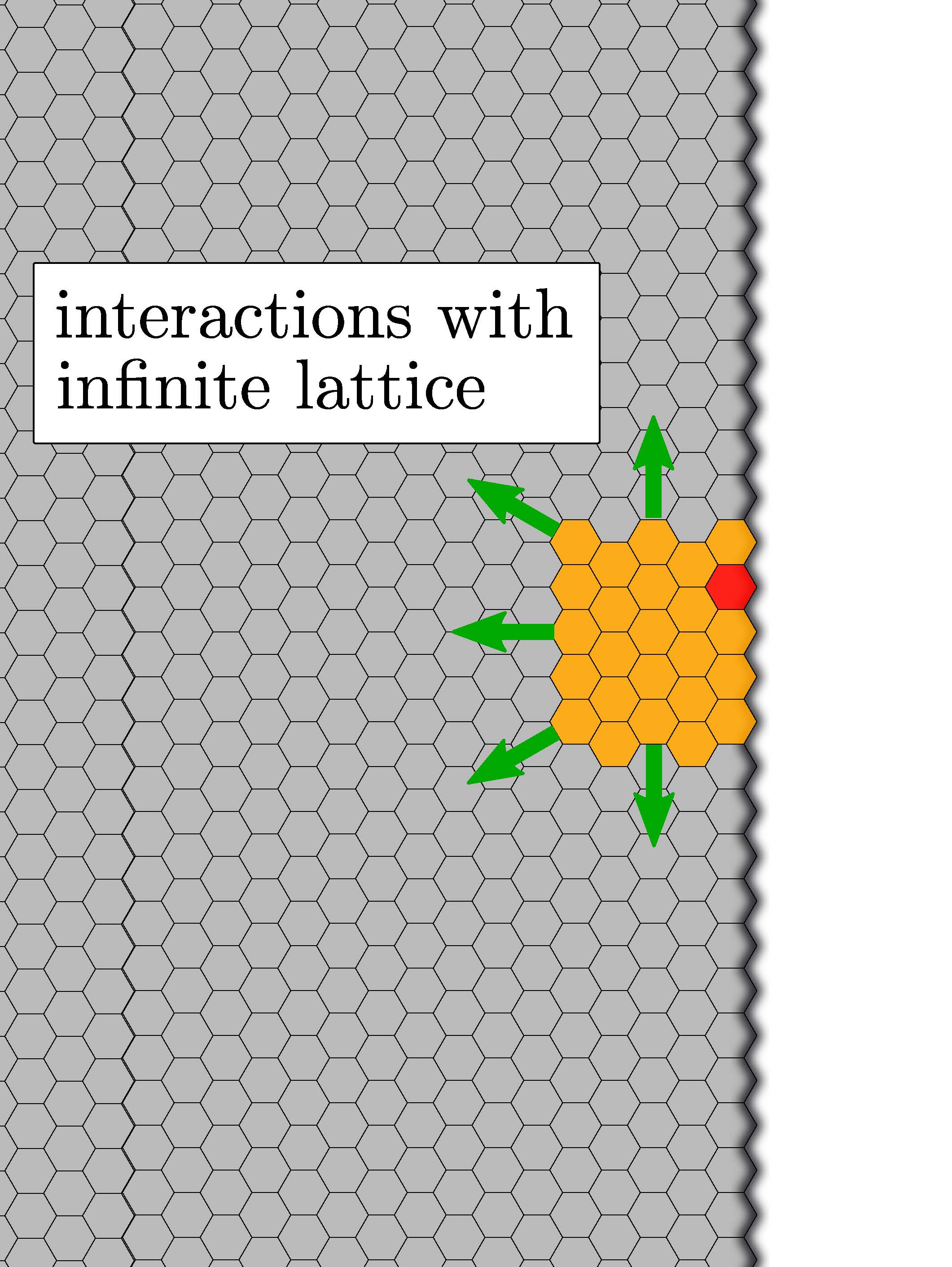}
\end{overpic}\\
\hline
\rule{0pt}{4ex}  
$\widetilde{H}$ & $P_n\widetilde{H}P_n$ & $P_{f(n)}\widetilde{H}P_n$\\
\begin{overpic}[width=0.30\textwidth,trim={0mm 0mm 0mm -3mm},clip]{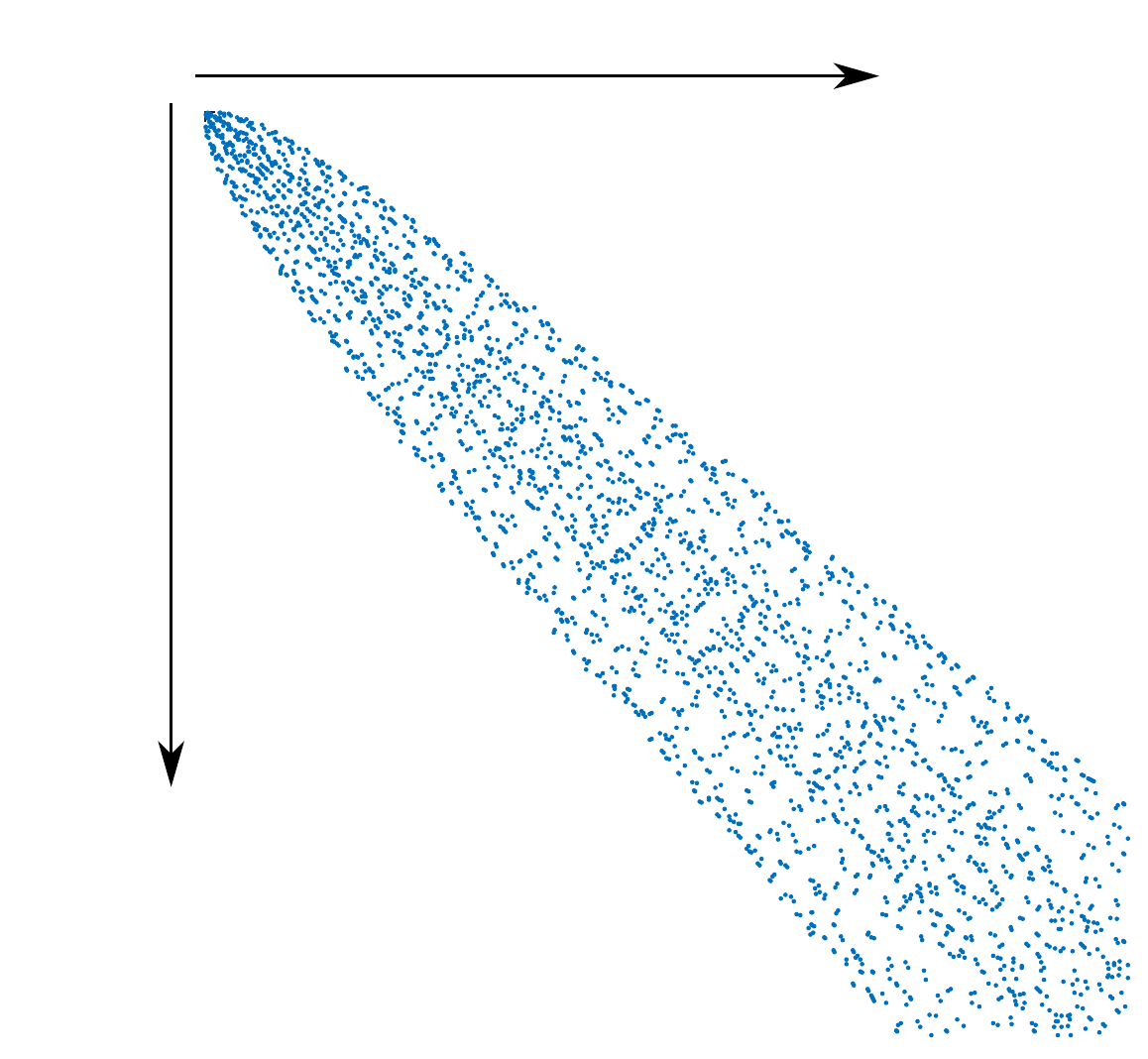}
\end{overpic}&
\begin{overpic}[width=0.30\textwidth,trim={0mm 0mm 0mm -3mm},clip]{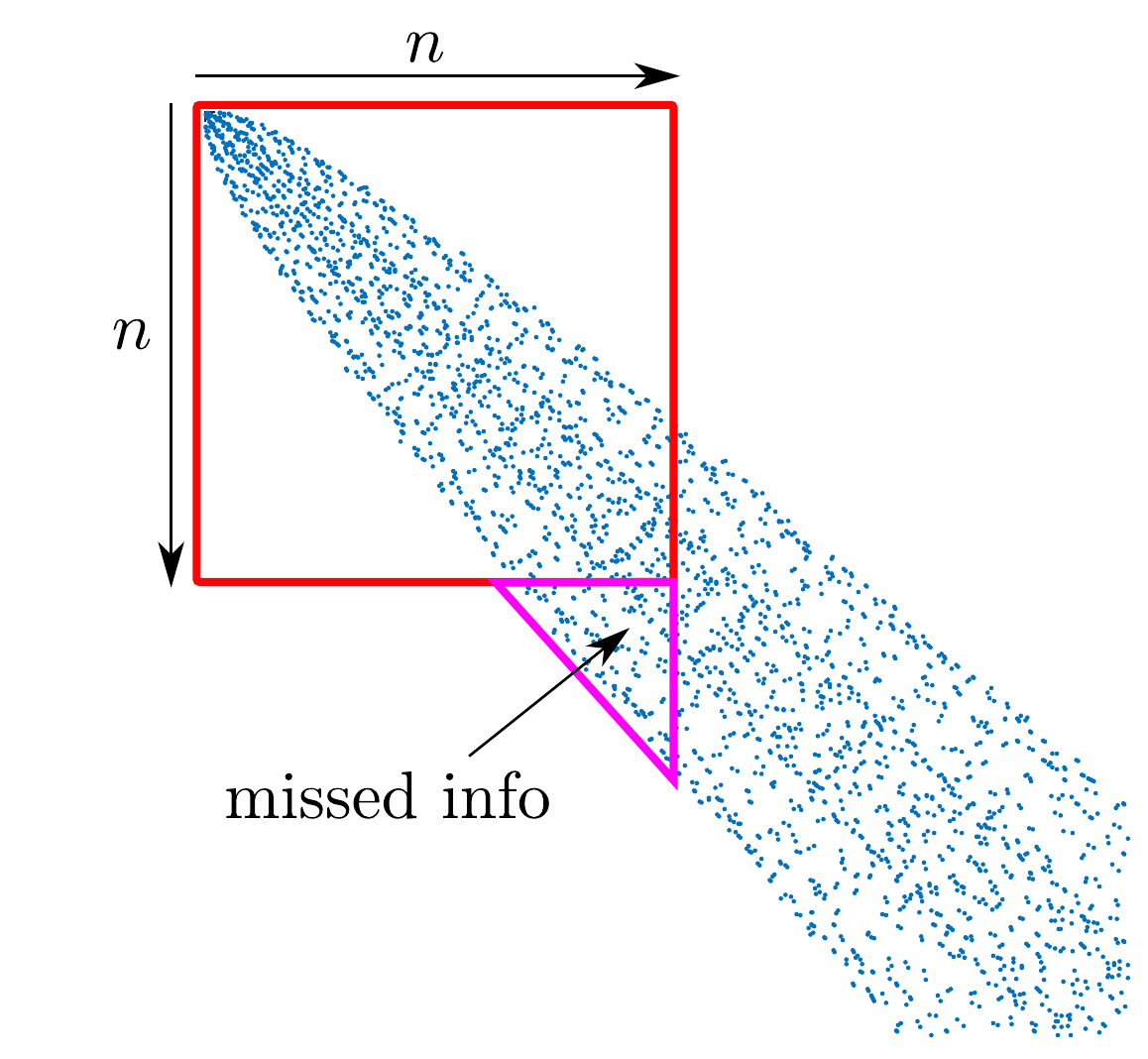}
\end{overpic}&
\begin{overpic}[width=0.30\textwidth,trim={0mm 0mm 0mm -3mm},clip]{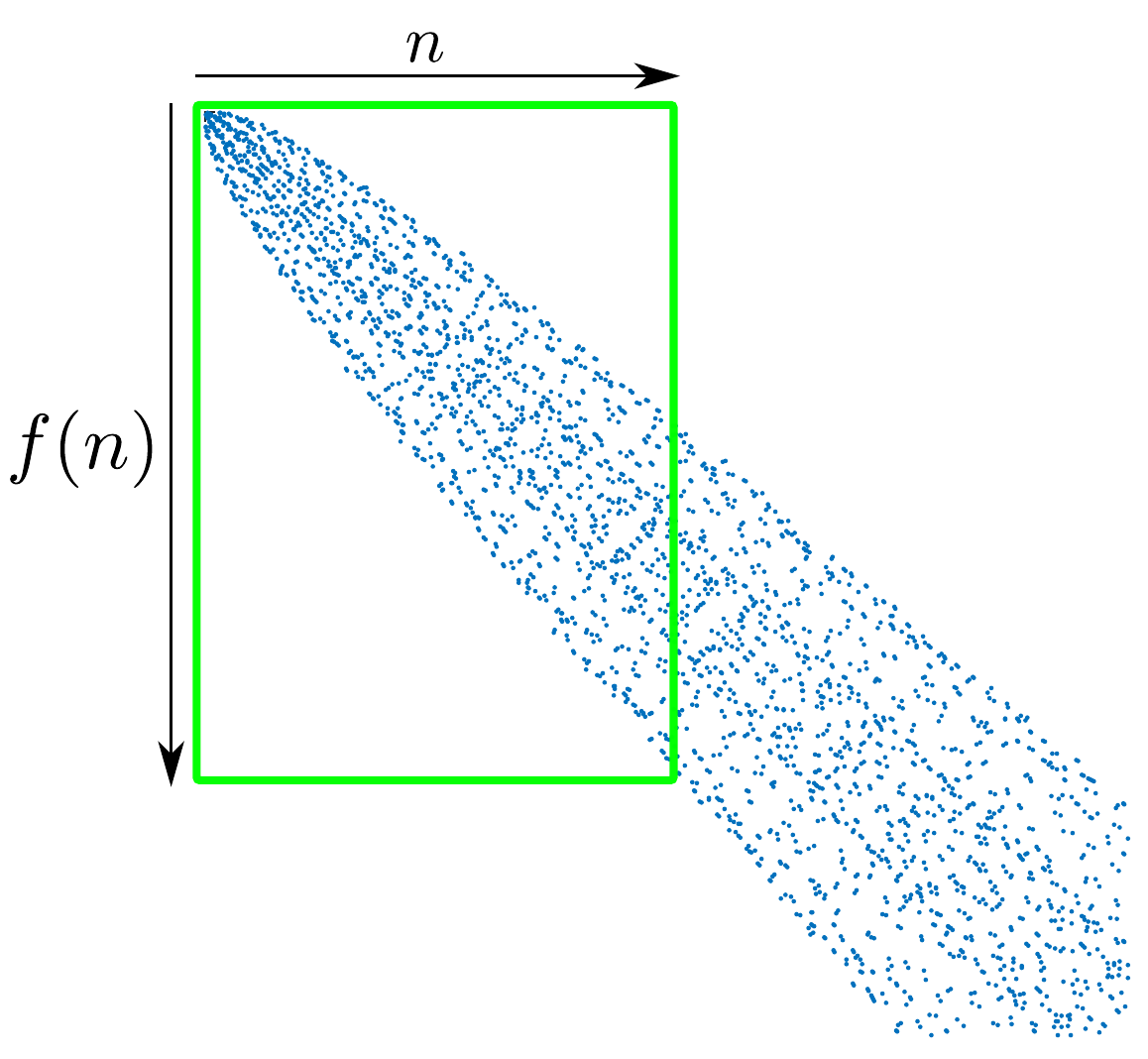}
\end{overpic}
\end{tabular}
\caption{Top: (a) Infinite hexagonal lattice with an infinite edge and possible defect. (b) Finite truncation of tile to $n$ sites. (c) Finite truncation with interactions shown as green arrows (our method). Bottom: The corresponding sparsity patterns (non-zero entries of the infinite matrix $\widetilde{H}$). The boxes show the different types of truncations of the operator. In (c), $f(n)$ is chosen to include all of the interactions of the first $n$ sites (basis vectors).}
\label{fig1}
\end{figure}

\begin{figure}
\centering
\begin{overpic}[width=0.5\textwidth,clip,trim={0mm 0mm 0mm 0mm}]{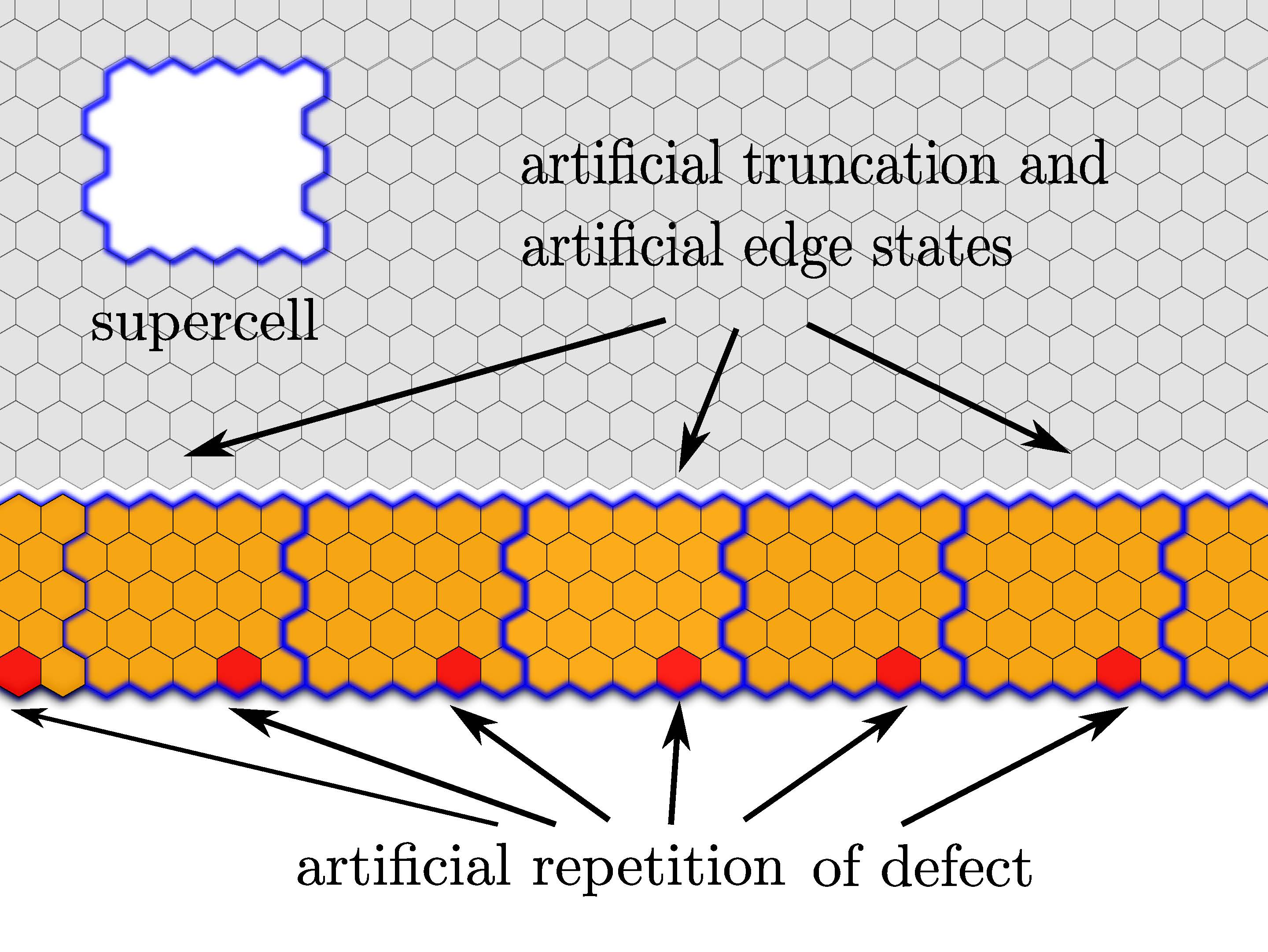}
\end{overpic}
    \caption{The supercell method on ribbon geometry combines periodic approximation along the edge with finite truncation away from the edge (note that the orientation of \cref{fig1} has been rotated by $\pi/2$). The artificial edge typically leads to numerical artifacts in the computed spectrum. When the lattice contains defects or disorder, the artificial periodicity may also contribute to numerical artifacts or degrade accuracy.}
\label{fig2}
\end{figure}

\vspace{2mm}

\underline{\textbf{Model:}} As a concrete example, consider a finite range Hamiltonian on a hexagonal lattice such as the Haldane model. The situation is shown in \cref{fig1} (panel (a)), where, for the sake of illustration, we have also added a physical edge (so that the model is a half lattice) as well as the position of a potential defect. The presence of the defect is significant because it breaks the translation symmetry of the system parallel to the edge, meaning that spectral properties cannot be computed (at least, not \emph{exactly}, see the `supercell method' below) using Bloch's theorem\footnote{Bloch's theorem (see, for example, \cite{ashcroft_mermin,reed_simon_4}) states that eigenfunctions of periodic Hamiltonians can be decomposed as $\psi(\mathbf{r})=e^{i\mathbf{k}\cdot\mathbf{r}}p(\mathbf{r})$ where $p$ shares the periodicity of the Hamiltonian. Bloch's theorem reduces spectral computations on infinite domains to computations (parametrised by $\mathbf{k}$) on domains that are bounded in each direction of periodicity, with periodic boundary conditions {(up to a phase)}.}
parallel to the edge.

\vspace{2mm}

\underline{\textbf{Previous approaches:}} Let $P_n$ denote the orthogonal projection onto the linear span of the first $n$ basis vectors. The most straightforward approach to computing spectral properties of the infinite operator ${H}$ is to compute the spectral properties of large square truncations of the matrix $\widetilde{H}$ and hope that the computations converge in the limit of large truncations. Mathematically, {this amounts to computing spectral properties of} the finite-dimensional matrices $P_n \widetilde{H} P_n$, where $n$ is a large positive integer (shown as a red box in \cref{fig1}). Physically, this corresponds to studying the interactions of a finite number of sites within the truncation (\cref{fig1} (b)). Although this method is straightforward, it is easy to see that the matrices $P_n \widetilde{H} P_n$ can have eigenvalues that never approach the spectrum of ${H}$, even as $n \rightarrow \infty$. This is an example of the general phenomenon known as `spectral pollution', where eigenvalues of finite discretisations/truncations can cluster in gaps between the essential spectrum of infinite self-adjoint operators \cite{Pokrzywa_79,2004DaviesPlum,2010LewinSere} as the truncation size increases. In the context of TIs, which necessarily have eigenstates localised at edges, spectral pollution arising from the new edges created by the truncation is inevitable, see \cite{doi:10.1137/19M1282696}.

In the TI literature, a common approach (taken, e.g., in the package \verb|PythTB| \cite{vanderbilt_2018}) to computing spectral properties of ${H}$ is the `supercell method on ribbon geometry'. The method involves two approximations. The first (the `supercell method') is to approximate the edge of the material by a periodic edge with a large fundamental cell (the `supercell') by repeating the defect along the edge. The spectral properties of the periodic edge are then related, via Bloch's theorem, to those of the Hamiltonian restricted to a semi-infinite strip extending perpendicular to the edge with periodic (up to a phase) boundary conditions along the boundaries of the strip. The second approximation is then to truncate these semi-infinite Hamiltonians far from the original edge so that the computational domain forms a `ribbon'. \cref{fig2} illustrates this approach. 

Supercell approximations have been proved to converge as the supercell width increases in other contexts, see, e.g., \cite{2005Soussi,Ehrlacher2016,doi:10.1137/18M122830X}, but we are not aware of any work justifying them in the present context. In practice, supercell approximations can be computationally inefficient when a large supercell is required, e.g., for materials with disorder (for a comprehensive discussion in the context of photonic quasicrystals, see~\cite{rodriguez2008computation}). As for the second approximation, spectral pollution is inevitable, leading to results that can be misleading and difficult to interpret (see~\cref{rem:ribbon}). It is worth remarking that the second approximation must be dealt with even when the supercell approximation is exact, e.g., when there is no defect at all! In certain circumstances, the second approximation can be removed (see~\cref{rem:GFM}).

\begin{remark} \label{rem:ribbon}
It is common in the physics literature to accept the additional spectrum arising from the truncation away from the edge and simply treat the system as having a second edge. However, the spectrum of the truncated Hamiltonian is often clearly different from that of the semi-infinite operator: see, for example, Lee--Thorp \cite{2016Lee-Thorp}, in particular Figure 26.7, and compare with Figure 5.3 of \cite{doi:10.1137/19M1282696}.
\end{remark}

\begin{remark} \label{rem:GFM}
    Two of the authors have introduced a method (the `Green's function method') which computes the discrete spectrum and associated eigenfunctions of the semi-infinite Hamiltonians obtained via supercell approximation which eliminates spectral pollution, see \cite{doi:10.1137/19M1282696}. Similarly to the present work, the fundamental idea of the `Green's function method' is to work with the resolvent. The main contrast between \cite{doi:10.1137/19M1282696} and the present work is that here we make fewer assumptions on the form of the edge Hamiltonian: we do not assume periodicity either parallel to the edge or into the bulk. We also compute other spectral properties not considered in \cite{doi:10.1137/19M1282696}: spectral measure, spectral projection, and time propagation. 
\end{remark}

\vspace{2mm}

\underline{\textbf{Rectangular truncations:}} In this work, we compute spectral properties using rectangular truncations of the form (shown as a green box in \cref{fig1}):
\begin{equation}
\label{rectangular_key}
P_{f(n)}\widetilde{H}P_n\in\mathbb{C}^{f(n)\times n}.
\end{equation}
Recall that $f$ describes the sparsity pattern of the matrix. In our case, the rectangular truncation $P_{f(n)}\widetilde{H}P_n$ corresponds to including all of the interactions of the first $n$ sites (the first $n$ columns of $\widetilde{H}$) without needing to apply boundary conditions. \cref{fig1} (panel (c)) shows the general idea. This truncation is in sharp contrast to {na\"ive} methods that typically take a square truncation of the matrix $\widetilde{H}$, such as $P_n\widetilde{H}P_n\in\mathbb{C}^{n\times n}$, with a boundary condition. This difference allows us to rigorously compute properties via computation of the resolvent operator
$$
R(z,\widetilde{H})=(\widetilde{H}-z)^{-1},\quad z\notin\mathrm{Sp}(H).
$$
Once we have computed the resolvent, we can then compute the spectral properties of the operator $H$. This approach lends itself to adaptive computations of the full infinite-dimensional operator directly, eliminating \emph{both} of the approximations involved in the `supercell method on ribbon geometry'. For our computational problems, this allows:

\begin{itemize}
	\item[(P1)] Computation of spectra and approximate eigenstates with guaranteed error control. In the case of discrete spectrum, the approximate eigenstates correspond to bona fide eigenstates (e.g., edge states when the edge is periodic).	
	\item[(P2)] \& (P3) Computation of spectral measures/projections with guaranteed convergence and given rates of approximation.
	\item[(P4)] Computation of the functional calculus (and hence transport properties) with guaranteed error control.
\end{itemize}
With this technique in hand, we can reliably probe the spectral properties of systems in infinite dimensions. Indeed, this technique is already allowing for the discovery and investigation of new physics in quasicrystalline systems, including their transport and topological properties \cite{johnstone2021bulk}.

\begin{remark}[Foundations of computation]
Our methods form part of a wider programme on the foundations of computations. One can classify computational spectral problems (and other types of computational problems) into a hierarchy (the SCI hierarchy) \cite{colbrook2020foundations,SCI_big,hansen2011solvability,colbrook2019b,colbrook2019computing,colbrook2019infinite,colbrook2019computation,ben2021universal,rosler2021computing,colbrook2020pseudoergodic}. This measures the intrinsic difficulty of computational problems and provides proofs of the optimality of algorithms, realising the limits of what computers can achieve. Beyond spectral theory, this framework is now being applied to optimisation, machine learning and artificial intelligence, solving partial differential equations, and computer-assisted proofs\footnote{See for example the work by C. Fefferman \& L. Seco (Dirac-Schwinger conjecture) \cite{fefferman1996interval,fefferman1990} and T. Hales et. al (Kepler's conjecture/Hilbert's 18th problem) \cite{Hales_Annals, hales_Pi} that implicitly prove results in the SCI hierarchy.} \cite{colbrook2021can,opt_big,webb2021spectra,colbrook2021semigroups,becker2020computing,Jonathan_res,benartzi2020computing}. The SCI hierarchy generalises S. Smale's seminal work \cite{Smale2, Smale_Acta_Numerica}  with L. Blum, F. Cucker, M. Shub \cite{ BSS_Machine, BCSS} and his program on the foundations of scientific computing  and existence of algorithms pioneered by C. McMullen \cite{McMullen1,mcmullen1988braiding}  and P. Doyle \& C. McMullen \cite{Doyle_McMullen}. As science and society become increasingly reliant on computations, it is essential to understand what is computationally possible and design algorithms that are optimal and achieve these bounds. For the sake of brevity, we have omitted this wider framework, but invite the interested reader to consult the above references.
\end{remark}

\section{The infinite-dimensional numerical methods}
\label{num_analysis_user_manual}

Here we briefly describe the algorithms for infinite-dimensional spectral computations. Pseudocode is provided in \cref{append_psueod}.

\subsection{Computation of spectra and approximate eigenstates}
\label{Num_spectra}

We utilise an algorithm, developed in \cite{Colbrook2019}, that computes the spectrum of an infinite-dimensional operator with error control. Recall that in our setting, the Hamiltonian $H$ can be represented by an infinite Hermitian matrix, $\widetilde{H}=\{\widetilde{H}_{ij}\}_{i,j\in\mathbb{N}}$ and we are given a function $f:\mathbb{N}\rightarrow\mathbb{N}$ such that $\widetilde{H}_{ij}=0$ if $i>f(j)$. Thus $f$ describes the sparsity of $\widetilde{H}$. Our starting point is the function
\begin{equation}
F_n(z):=\sigma_{\mathrm{inf}}(P_{f(n)}(\widetilde{H}-z)P_n),
\end{equation}
where we remind the reader that $P_m$ denotes the orthogonal projection onto the linear span of the first $m$ basis vectors. We also use $\sigma_{\mathrm{inf}}$ to denote the smallest singular value of the corresponding rectangular matrix. Since our operator is normal (commutes with its adjoint), the function $F$ is an upper bound for the distance of $z$ to the spectrum $\mathrm{Sp}(H)$, and converges down to this distance uniformly on compact sets as $n\rightarrow\infty$ \cite{Colbrook2019}. Physically, $F_n(z)$ is the square-root of the ground state energy of the folded Hamiltonian $P_n(\widetilde{H}-z)^*(\widetilde{H}-z)P_n$. There are numerous ways to compute $F_n$, such as standard iterative algorithms or incomplete Cholesky decomposition of the shifts $P_n(\widetilde{H}-z)^*P_{f(n)}(\widetilde{H}-z)P_n$ (see the supplementary material of \cite{Colbrook2019} for a discussion). The other ingredient we need is a grid of points $G_n=\{z_1^{(n)},...,z_{j(n)}^{(n)}\}\subset\mathbb{R}$ providing the wanted resolution $r_n$ over the spectral region of interest.

The method is sketched in \cref{alg:spec_comp} and produces three quantities: $\Gamma_n$, $E_n$ and $V_n$. The simple idea of the method is a local search routine. If $ F_n(z)\leq 1/2$, we search within a radius $ F_n(z)$ around $z$ to minimise the approximated distance to the spectrum. This gives our best estimate of points in the spectrum near $z$ (the set $M_z$). The output $\Gamma_n(H)$ is then the collection of these local minimisers. $\Gamma_n(H)$ converges to the spectrum $\mathrm{Sp}(H)$ of the full infinite-dimensional operator as $n\rightarrow\infty$ (for suitable $r_n\rightarrow\infty$). This convergence is free from the edge states/spectral pollution that are associated with any artificial or numerical truncation. In other words, we compute $\mathrm{Sp}(H)$, and only $\mathrm{Sp}(H)$. Note that in the examples of this paper, $\mathrm{Sp}(H)$ does include {spectrum associated to} edge states of the full Hamiltonian $H$. The algorithm also outputs an error bound $E_n$ that satisfies
\begin{equation}
\label{spec_err_controlll}
\sup_{z\in\Gamma_n(H)} \mathrm{dist}(z,\mathrm{Sp}(H))\leq E_n\quad \text{with}\quad \lim_{n\rightarrow\infty}E_n=0.
\end{equation}
For an accuracy $\delta>0$, we simply increase $n$ until $E_n\leq\delta$. The final quantity $V_n$ consists of the approximate states corresponding to the output $\Gamma_n$. The approximate eigenstate $v_n(z)$ satisfies
$$
\|(\widetilde{H}-z)v_n(z)\|=\|P_{f(n)}(\widetilde{H}-z)P_nv_n(z)\|=F_n(z)\leq E_n,
$$
up to numerical errors. For an interval arithmetic implementation of this algorithm (allowing verified error bounds) and extensions to partial differential operators, see \cite{colbrook2019b}. We can also verify the spectral content of these approximate eigenstates by computing their spectral measure, see \cref{Num_spec_meas}.

\subsection{Computation of scalar spectral measures}
\label{Num_spec_meas}

Associated with the Hamiltonian $H$ is a projection-valued measure, $\mathcal{E}$, whose existence is guaranteed by the spectral theorem \cite[Theorem VIII.6]{reed1980} and whose support is the spectrum $\mathrm{Sp}(H)$. This diagonalises $H$, even when there does not exist a basis of normalisable eigenstates (recall that we are working in an infinite-dimensional Hilbert space):
\begin{equation}\label{eqn:PVMdiag}
 H=\inty{\mathrm{Sp}(H)}{}{\lambda}{\mathcal{E}(\lambda)}.
\end{equation}
In finite dimensions, {or when $H$ is compact or has compact resolvent,} $\mathcal{E}$ consists of a sum of Dirac measures, located at the eigenvalues of $H$, whose values are the corresponding projections onto eigenspaces. More generally, however, there may be a continuous component of the spectrum and spectral measure.

The key ingredient that allows approximations of $\mathcal{E}$ to be computed is the formula for the resolvent
\begin{equation}\label{eqn:PVMres}
( H-z)^{-1}=\inty{\mathrm{Sp}(H)}{}{\frac{1}{\lambda-z}}{\mathcal{E}(\lambda)}.
\end{equation}
In \cite{colbrook2019computing}, it is shown how to compute the action of the resolvent with error control via the rectangular truncations $\smash{P_{f(n)}(\widetilde{H}-z)P_n}$ and solving the resulting overdetermined linear system in the least squares sense. The residual converges to zero as $n\rightarrow\infty$ and can be used to provide the needed error bounds through an adaptive selection of $n$ \cite[Theorem 2.1]{colbrook2019computing}. Using this, we compute a smoothed approximation of $\mathcal{E}$ via convolution with a rational kernel $K_\epsilon$ for smoothing parameter $\epsilon>0$.

We explain the method for the important case of scalar-valued measures, before discussing the case of spectral projections in \cref{Num_spec_proj}. The spectral measure of $H$ with respect to $\psi\in\mathcal{H}$ is a scalar measure defined as $\mu_\psi(\Omega):=\langle\mathcal{E}(\Omega)\psi,\psi\rangle$. Lebesgue's decomposition of $\mu_\psi$~\cite{stein2009real} gives
\begin{equation}\label{eqn:spec_meas}
\dee\mu_\psi(y)= \underbrace{\sum_{\lambda\in\mathrm{Sp}_{{\rm p}}(H)}\langle\mathcal{P}_\lambda \psi,\psi\rangle\,\delta({y-\lambda})\dee y}_{\text{discrete part}}+\underbrace{\rho_\psi(y)\,\dee y +\dee\mu_\psi^{(\mathrm{sc})}(y)}_{\text{continuous part}}.
\end{equation}
The discrete part of $\mu_\psi$ is a sum of Dirac delta distributions on the set of eigenvalues of $H$, which we denote by $\mathrm{Sp}_{{\rm p}}(H)$. The coefficient of each $\delta$ in the sum is $\langle\mathcal{P}_\lambda \psi,\psi\rangle=\|\mathcal{P}_\lambda \psi\|^2$, where $\mathcal{P}_\lambda$ is the orthogonal spectral projector associated with the eigenvalue $\lambda$. The continuous part of $\mu_\psi$ consists of an absolutely continuous part with Radon--Nikodym derivative $\rho_\psi\in L^1(\mathbb{R})$ and a singular continuous component $\smash{\mu_\psi^{(\mathrm{sc})}}$.

We evaluate smoothed approximations of $\mu_\psi$ via a function $g_\epsilon$, with smoothing parameter $\epsilon>0$, that converges weakly to $\mu_\psi$~\cite[Ch.~1]{billingsley2013convergence}. That is, 
\[
\inty{\mathbb{R}}{}{\phi(y)g_\epsilon(y)}{y}\rightarrow \inty{\mathbb{R}}{}{\phi(y)}{\mu_\psi(y)}, \qquad\text{as}\qquad\epsilon\downarrow 0,
\]
for any bounded, continuous function $\phi$. The classical example of this is Stone's formula which corresponds to convolution with the Poisson kernel
\begin{equation}\label{jbjibfvbq}
g_\epsilon(x)=\frac{1}{2\pi i}\left\langle\left[(H-(x+i\epsilon))^{-1}-( H-(x-i\epsilon))^{-1}\right]\psi,\psi\right\rangle=\inty{\mathbb{R}}{}{\frac{\epsilon\pi^{-1}}{(x-\lambda)^2+\epsilon^2}}{\mu_\psi(\lambda)}.
\end{equation}
As $\epsilon\downarrow 0$, this approximation converges weakly to $\mu_\psi$. However, for a given truncation size, if $\epsilon$ is too small the approximation of \eqref{jbjibfvbq} via $\smash{P_{f(n)}(\widetilde{H}-z)P_n}$ (described above) becomes unstable due to the truncation of $\widetilde{H}$. There is an increased computational cost for smaller $\epsilon$, which typically requires larger truncation parameters. Since we want to approximate spectral properties without finite-size effects, it is advantageous to replace the Poisson kernel with higher-order rational kernels developed in \cite{colbrook2020}. These kernels have better convergence rates as $\epsilon\downarrow 0$, allowing a larger $\epsilon$ to be used for a given accuracy, thus leading to a lower computational burden. We use the high-order kernel machinery developed in \cite{colbrook2020}, where the following definition is made.

\begin{definition}[$m$th order kernel]
\label{def:mth_order_kernel}
Let $m\in\mathbb{N}$ and $K\in L^1(\mathbb{R})$. We say $K$ is an $m$th order kernel if:
\begin{itemize}
	\item[(i)] Normalised: $\inty{\mathbb{R}}{}{K(x)}{x}=1$.
	\item[(ii)]  Zero moments: $K(x)x^j$ is integrable and $\inty{\mathbb{R}}{}{K(x)x^j}{x}=0$ for $0<j<m$.
	\item[(iii)] Decay at $\pm\infty$: There is a constant $C_K$, such that $
	\left|K(x)\right|\leq {C_K}{(1+\left|x\right|)^{-(m+1)}}$, $\forall x\in \mathbb{R}.$
\end{itemize}
\end{definition}

We set $K_{\epsilon}(\cdot)=\epsilon^{-1}K(\cdot /\epsilon)$ to obtain an approximate identity. High-order kernels can be constructed using rational functions as follows. Let $\{a_j\}_{j=1}^m$ be distinct points in the upper half plane and suppose that the constants $\{\alpha_j\}_{j=1}^m$ satisfy the following (transposed) Vandermonde system:
\begin{equation}\label{eqn:vandermonde_condition}
\begin{pmatrix}
1 & \dots & 1 \\
a_1 & \dots & a_{m} \\
\vdots & \ddots & \vdots \\
a_1^{m-1} &  \dots & a_{m}^{m-1}
\end{pmatrix}
\begin{pmatrix}
\alpha_1 \\ \alpha_2\\ \vdots \\ \alpha_m
\end{pmatrix}
=\begin{pmatrix}
1 \\ 0 \\ \vdots \\0
\end{pmatrix}.
\end{equation}
Then the kernel
\begin{equation}\label{eqn:high_order_kernel}
K(x)=\frac{1}{2\pi i}\sum_{j=1}^m\frac{\alpha_j}{x-a_j}-\frac{1}{2\pi i}\sum_{j=1}^m\frac{\overline{\alpha_j}}{x-\overline{a_j}},
\end{equation}
is an $m$th order kernel, and we have the following generalisation of Stone's formula
\begin{align}
\label{gen_stone_comp}
[K_{\epsilon}*\mu_\psi](x)=&\frac{-1}{2\pi i}\sum_{j=1}^{m}\left\langle\left[\alpha_j (H-(x-\epsilon a_j))^{-1}-\bar\alpha_j (H-(x-\epsilon \bar a_j))^{-1}\right]\psi,\psi\right\rangle\notag\\
=&\frac{-1}{\pi}\sum_{j=1}^{m}{\rm Im}\left(\alpha_j\,\langle (H-(x-\epsilon a_j))^{-1}\psi,\psi \rangle\right).
\end{align}
This convolution converges with $m$th order of convergence in $\epsilon$ (up to a logarithmic factor and for sufficiently smooth $\mu_\psi$) \cite{colbrook2020}. The second line of \eqref{gen_stone_comp} follows from the conjugate symmetry of the resolvent. Here, $\bar z$ denotes the complex conjugate of $z$ and $*$ represents convolution. As a natural extension of the Poisson kernel, whose two poles are at $\pm i$, we consider the choice $a_j={2j}/({m+1})-1+i$. We then determine the residues by solving the Vandermonde system in \eqref{eqn:vandermonde_condition}. The first six kernels are explicitly written down in \cref{tab_kernel} (taken from \cite{colbrook2020}).

Given an $m$th order rational kernel, defined by distinct poles $a_1,\dots,a_m$ in the upper half-plane, the resolvent-based framework for evaluating an approximation of the spectral measure $\mu_{\psi}$ is summarised in \cref{alg:spec_meas}. This algorithm, which can be performed in parallel for several $x_0$, forms the foundation of \texttt{SpecSolve} \cite{SpecSolve_code}. In practice, the resolvent in~\cref{alg:spec_meas} is discretised before being applied. We compute an accurate value of $\mu_{\psi}^\epsilon$ provided that the resolvent is applied with sufficient accuracy, which can be done {\em{}adaptively} with {\em{}a posteriori} error bounds~\cite{colbrook2019computing}. For an efficient adaptive implementation, \texttt{SpecSolve} constructs a fixed discretisation, solves linear systems at each required complex shift, and checks the approximation error at each shift. If further accuracy is needed at a subset of the shifts, then the discretisation is refined geometrically, applied at these shifts, and the error is recomputed. This process is repeated until the resolvent is computed accurately at all shifts.

\subsection{Computation of spectral projections}
\label{Num_spec_proj}

Given an interval $[a,b]\subset\mathrm{Sp}(H)$, a vector $\psi\in\mathcal{H}$, and an $m$th order kernel $K$, the identity
\begin{equation}
\label{gen_stone_comp2}
[K_{\epsilon}*\mathcal{E}](x)=\frac{-1}{2\pi i}\sum_{j=1}^{m}\left[\alpha_j (H-(x-\epsilon a_j))^{-1}-\bar\alpha_j (H-(x-\epsilon \bar a_j))^{-1}\right]
\end{equation}
allows us to approximate $\mathcal{E}([a,b])\psi$ by solving shifted linear system of the form $(H-z)u=\psi$. Employing a quadrature rule with weights $w_1,\ldots,w_\ell$ and nodes $x_1,\ldots,x_\ell$, we form the approximation
\begin{equation}\label{eqn:pvm1}
\inty{a}{b}{[K_{\epsilon}*\mathcal{E}](x)}{x}\approx \frac{-1}{2\pi i}\sum_{\ell=1}^Nw_\ell\sum_{j=1}^{m}\left[\alpha_j (H-(x_\ell-\epsilon a_j))^{-1}-\bar\alpha_j (H-(x_\ell-\epsilon \bar a_j))^{-1}\right].
\end{equation}
The following generalisation of Stone's formula establishes that the approximation converges in the limit $\epsilon\rightarrow 0$, up to contributions from atoms of $\mathcal{E}$ at the endpoints. It is convenient to distinguish between the real and imaginary parts of the residues explicitly, so we denote $\alpha_j=\beta_j+i\gamma_j$.

\begin{theorem}\label{thm:Stones_formula_generalized}
Given a projection-valued measure $\mathcal{E}$ (see~\eqref{eqn:PVMdiag}) and $m$th order kernel $K$ with conjugate pole pairs (see~\eqref{eqn:high_order_kernel}), for any $[a,b]\subset\mathbb{R}$ we have that
$$
\lim_{\epsilon\rightarrow 0^+}\inty{a}{b}{[K_{\epsilon}*\mathcal{E}](x)}{x} = \mathcal{E}((a,b))+c_l\mathcal{E}(\{a\})+c_r\mathcal{E}(\{b\}),
$$
where $c_l=\smash{\pi^{-1}\sum_{j=1}^m\beta_j(\pi-\arg(a_j))+i\gamma_j\log|a_j|}$ and $c_r=\smash{\pi^{-1}\sum_{j=1}^m\beta_j\arg(a_j)-i\gamma_j\log|a_j|}$. Moreover, if the poles are symmetric about the imaginary axis so that $a_{m+1-j}=-\bar a_j$, then $c_l=c_r=1/2$.
\end{theorem}
\begin{proof}
See~\cref{sec:PVM_converge}.
\end{proof}

\begin{remark}[Contribution of singleton sets]
One can easily show via the dominated convergence theorem that
$$
\lim_{\epsilon\rightarrow 0^+}\frac{\epsilon}{2i}\left[(H-x-i\epsilon)^{-1}-( H-x+i\epsilon)^{-1}\right]=\mathcal{E}(\{x\}).
$$
Together with \cref{thm:Stones_formula_generalized}, this allows computation of $\mathcal{E}((a,b))$ and $\mathcal{E}([a,b])$.
\end{remark}

By analogy with scalar spectral measures, approximating projection-valued spectral measures with higher-order rational convolution kernels is computationally advantageous because they achieve comparable accuracy with larger $\epsilon$. By increasing the kernel order rather than decreasing $\epsilon$, the resolvent in~\eqref{eqn:pvm1} remains well-conditioned and is usually significantly cheaper to apply. In addition, we can reduce the computational cost further by leveraging the resolvent's analyticity in the upper and lower half-plane and deforming the contour of integration in~\eqref{eqn:pvm1} away from the spectrum (see~\cref{fig:compare_contours1} (left)).

\begin{figure}
\centering
\begin{minipage}[b]{0.48\textwidth}
\begin{overpic}[width=0.99\textwidth]{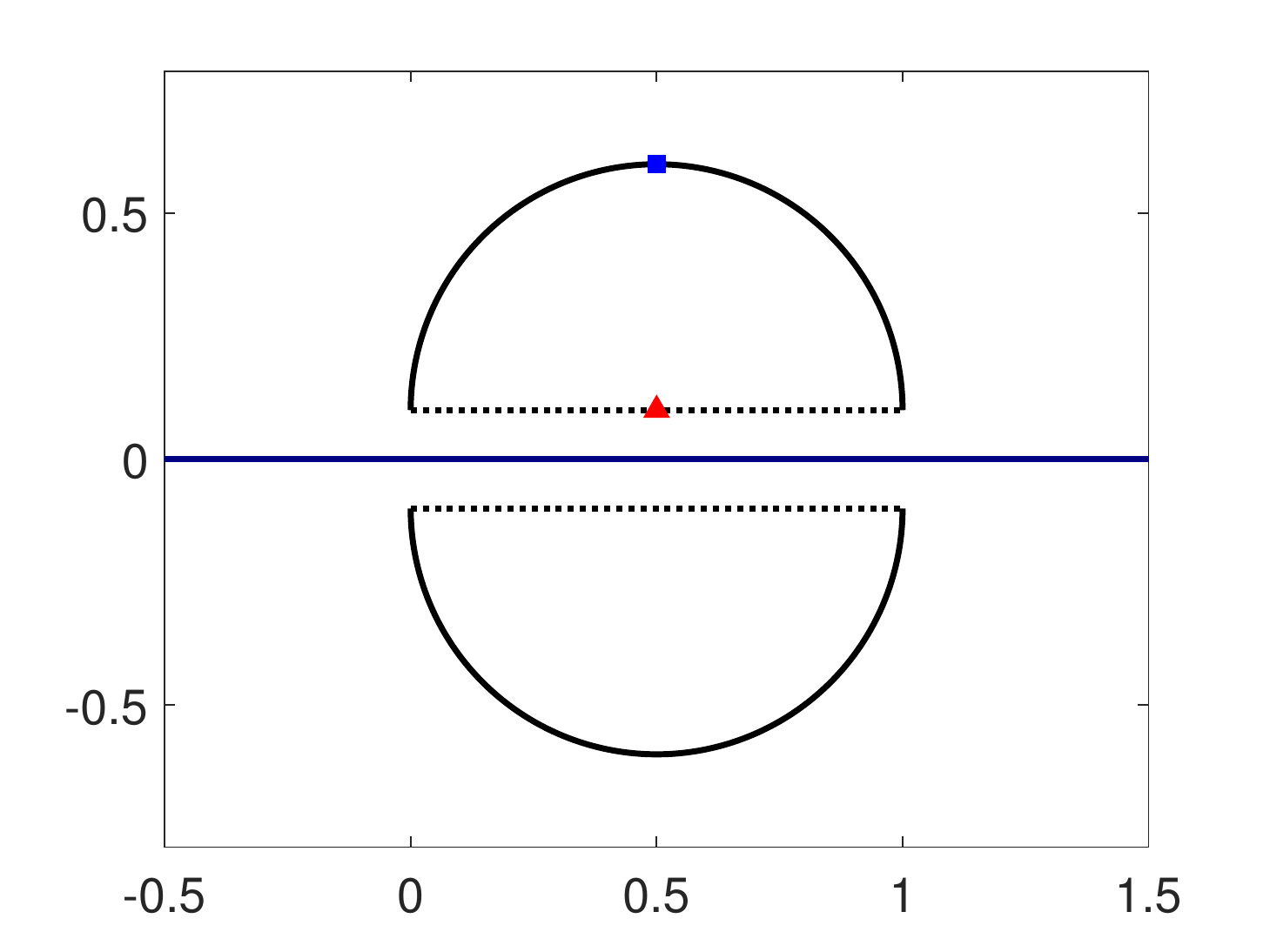}
\put (46,-2) {$\displaystyle {\rm Re}(z)$}
\put (2,33) {\rotatebox{90}{$\displaystyle {\rm Im}(z)$}}
\end{overpic}
\end{minipage}
\begin{minipage}[b]{0.48\textwidth}
\begin{overpic}[width=0.99\textwidth]{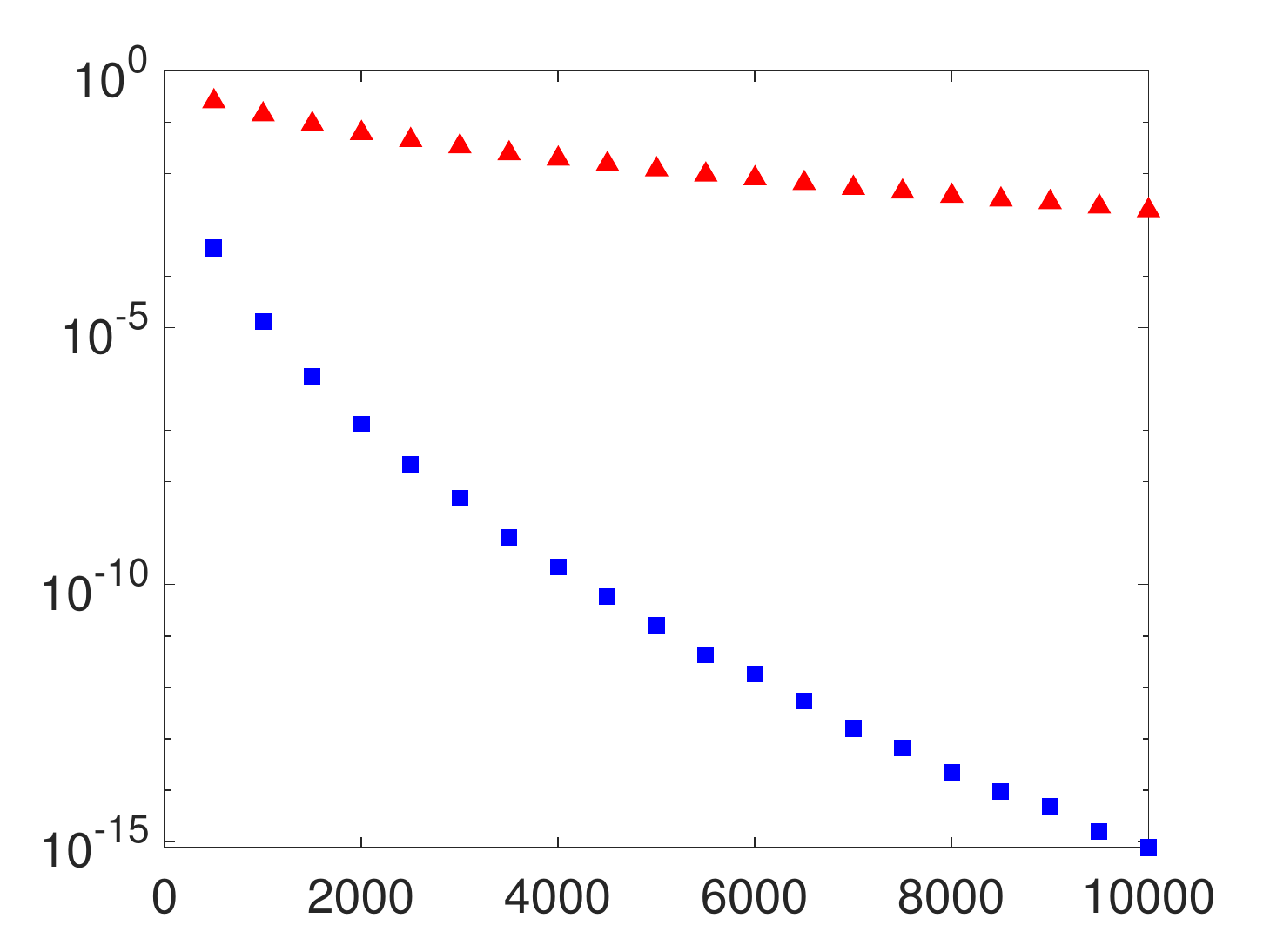}
\put (35,-2) {$\displaystyle n$ (truncation size)}
\put (37,72) {$\displaystyle \|v_n-v\|/\|v\|$}
\end{overpic}
\end{minipage}
\caption{Deforming the contour of integration away from the spectrum of $H$ in~\eqref{eqn:pvm1} alleviates the computational cost of computing the resolvent at the interior quadrature nodes. The left panel depicts two integration contours (solid and dashed lines) for the Poisson kernel ($m=1$) with smoothing parameter $\epsilon=0.1$. The right panel displays the relative approximation errors in the solutions of the truncated system $(P_{f(n)}\widetilde{H}P_n-z)v_n=e_1$ (see~\cref{sec:bijbibav}), where $e_1$ is the first canonical basis vector and $\widetilde{H}$ is the Hamiltonian of the bulk Haldane model (see~\cref{sec:bulk_Haldane}), corresponding to the two values of $z$ marked along the contours in the left panel (blue square and red triangle).}
\label{fig:compare_contours1}
\end{figure}

Consider the semi-circle contour connecting the points $a$ and $b$ and oriented in the clockwise direction, parametrised explicitly by $0\leq\theta\leq 1$ as
$$
z(\theta)=b+\frac{a-b}{2}(1+\exp(i\pi\theta)), \qquad\text{with}\qquad z'(\theta)=i\pi\frac{a-b}{2}\exp(i\pi\theta).
$$
Since the resolvent is analytic in $\mathbb{C}\setminus\mathrm{Sp}(H)$ and $a_1,\ldots,a_m$ lie in the upper half-plane, we may write
\begin{equation}\label{eqn:pvm2}
\begin{aligned}
\inty{a}{b}{[K_{\epsilon}*\mathcal{E}](x)}{x} &= \frac{-1}{2\pi i}\inty{0}{1}{\sum_{j=1}^{m}\left[\alpha_j(H-(z(\theta)-\epsilon a_j))^{-1}z'(\theta) - \bar\alpha_j(H-(\bar z(\theta)-\epsilon \bar a_j))^{-1}\bar z'(\theta)\right]}{z} \\
&\approx \frac{-1}{2\pi i}\sum_{\ell=1}^N\tilde w_\ell\sum_{j=1}^{m}\left[\alpha_j(H-(z(\theta_\ell)-\epsilon a_j))^{-1}z'(\theta_\ell) - \bar\alpha_j(H-(\bar z(\theta_\ell)-\epsilon \bar a_j))^{-1}\bar z'(\theta_\ell)\right].
\end{aligned}
\end{equation}
with quadrature weights $\tilde w_1,\ldots,\tilde w_N$ and nodes $\theta_1,\ldots,\theta_N$. From a computational standpoint,~\eqref{eqn:pvm2} improves two-fold on the formulation in~\eqref{eqn:pvm1}:
\begin{itemize}
	\item First, the resolvent is evaluated further from the spectrum and is typically well-approximated by smaller discretisations at many interior quadrature nodes (see~\cref{fig:compare_contours1} (right)).
	\item Second, the convergence rate of quadrature rules are improved because the integrand's region of analyticity is effectively enlarged  when the contour is deformed away from the spectrum (see~\cref{fig:compare_contours2})~\cite{hale2008new}. Consequently, fewer quadrature nodes are required to approximate the integral to a fixed tolerance.
\end{itemize}
Therefore, comparable accuracy is achieved while solving both fewer and smaller linear systems.
 
To compare the computational efficiency of the two contours with respect to the second point, we estimate the number of quadrature nodes required to achieve approximation error $0<\delta_*<1$. We consider spectral projection onto the interval $[0,1]$ (without loss of generality), Clenshaw--Curtis quadrature (CCQ), and an $m$th order rational kernel with equispaced poles $a_j={2j}/({m+1})-1+i$ (see~\cref{Num_spec_meas}). In this setting, Clenshaw--Curtis converges exponentially so that the quadrature approximation error is bounded by $\|E_N\|\leq C\rho^{-N}$, where $N$ is the number of quadrature nodes, $\rho>1$ is half the sum of the major and minor axes of any Bernstein ellipse $B_\rho$ with focii at $0$ and $1$ in which $[K_\epsilon*\mathcal{E}](x)$ is analytic, and $C>0$ is a constant proportional to $\sup_{z\in B_\rho}\|[K_\epsilon*\mathcal{E}](z)\|$. The minimal number of nodes required to achieve $\|E_N\|\leq\delta_*$ error is therefore $N\approx\log\left(C/\delta_*\right)/\log{\rho}$.

To estimate the convergence rate $\rho$ for each contour, suppose the singularities of $[K_\epsilon*\mathcal{E}](x)$ are determined precisely by the spectrum of $H$.\footnote{In fact, $[K_\epsilon*\mathcal{E}](x)$ may sometimes be analytically continued across the spectrum of $H$, in which case CCQ may converge faster than our analysis indicates for both contours.} For the flat contour (see~\eqref{eqn:pvm1}), the integrand is analytic between parallel lines displaced from the contour of integration by $\pm i\epsilon$ in the complex $x$-plane (see~\cref{fig:compare_contours2}, left). We consider the elliptic region of analyticity with minor axis $(\rho_1-\rho_1^{-1})/2=\epsilon$, so that
$$
\rho_1=\epsilon+\sqrt{\epsilon^2+1}=1+\mathcal{O}(\epsilon)\qquad\text{as}\qquad\epsilon\rightarrow 0.
$$
For the deformed contour (see~\eqref{eqn:pvm2}), the integrand's region of analyticity is bounded by the curves in the complex $\theta$-plane defined by $z(\theta)-\epsilon a_j\in{\rm Sp}(H)$ for $j=1,\ldots, m$ (see~\cref{fig:compare_contours2}, right). Here, we may take the Bernstein ellipse with major axis $(\rho_2+\rho_2^{-1})/2=1+\epsilon/2$, so that
$$
\rho_2=1+{\epsilon}/{2}+\sqrt{\left(1+{\epsilon}/{2}\right)^2-1}=1+\mathcal{O}(\sqrt{\epsilon}),\qquad\text{as}\qquad\epsilon\rightarrow 0.
$$
Since $H$ is self-adjoint, $\sup_{z\in B_{\rho}}\|[K_\epsilon*\mathcal{E}](z)\|$ grows in inverse proportion to ${\rm dist}(B_\rho,{\rm Sp}(H))$ as does the constant $C$. For both contours, our choice of $B_\rho$ yields $C=\mathcal{O}(\epsilon^{-1})$. Therefore, we conclude that the number of quadrature nodes required on each contour is, as $\epsilon\rightarrow 0$,
$$
N_1\approx\frac{\log\left(C/\delta_*\right)}{\log{\rho_1}}=\mathcal{O}\left(\epsilon^{-1}\log(\epsilon^{-1}\delta_*^{-1})\right), \qquad\text{and}\qquad N_2\approx\frac{\log\left(C/\delta_*\right)}{\log{\rho_2}}=\mathcal{O}\left(\epsilon^{-\frac{1}{2}}\log(\epsilon^{-1}\delta_*^{-1})\right).
$$
The deformed contour improves on~\eqref{eqn:pvm1} by requiring a factor of up to $\mathcal{O}(\sqrt{\epsilon})$ fewer CCQ nodes. This analysis also reveals the further benefit of a reduced number of quadrature nodes when increasing $\epsilon$ using high-order kernels.\footnote{Under certain smoothness conditions we can take $\epsilon=\mathcal{O}(\delta_*^{1/m})$ up to logarithmic factors.}

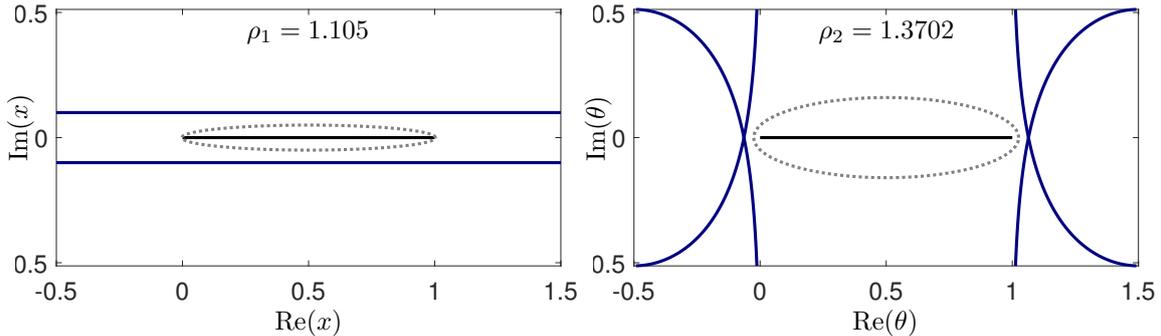
\begin{figure}
\centering
\begin{minipage}[b]{0.48\textwidth}
\begin{overpic}[width=0.99\textwidth,clip,trim={10mm 20mm 10mm 23mm}]{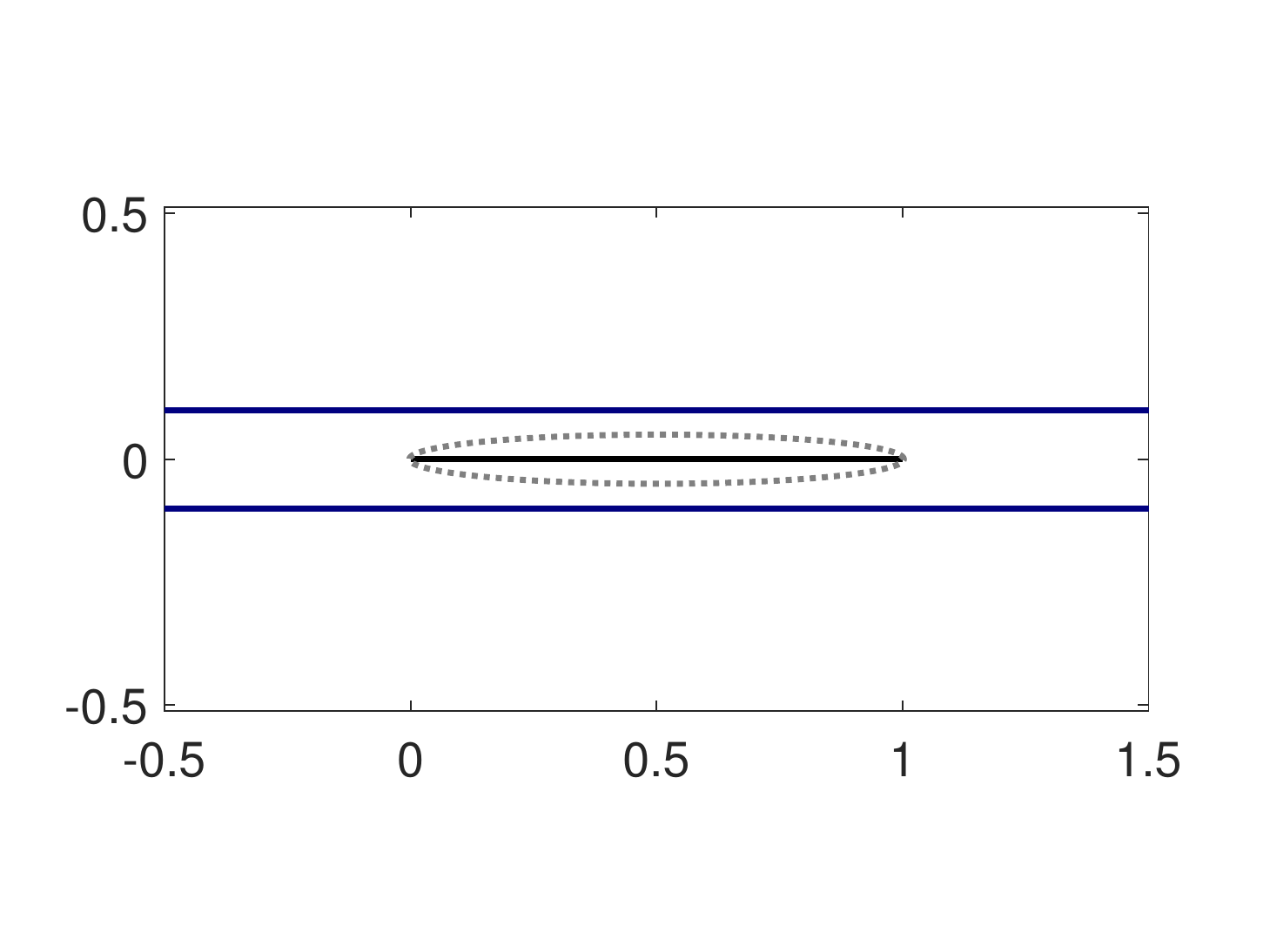}
\put (41,47) {$\displaystyle \rho_1=1.105$}
\put (46,-5) {$\displaystyle {\rm Re}(x)$}
\put (-1.5,25) {\rotatebox{90}{$\displaystyle {\rm Im}(x)$}}
\end{overpic}
\end{minipage}
\begin{minipage}[b]{0.48\textwidth}
\begin{overpic}[width=0.99\textwidth,clip,trim={10mm 20mm 10mm 23mm}]{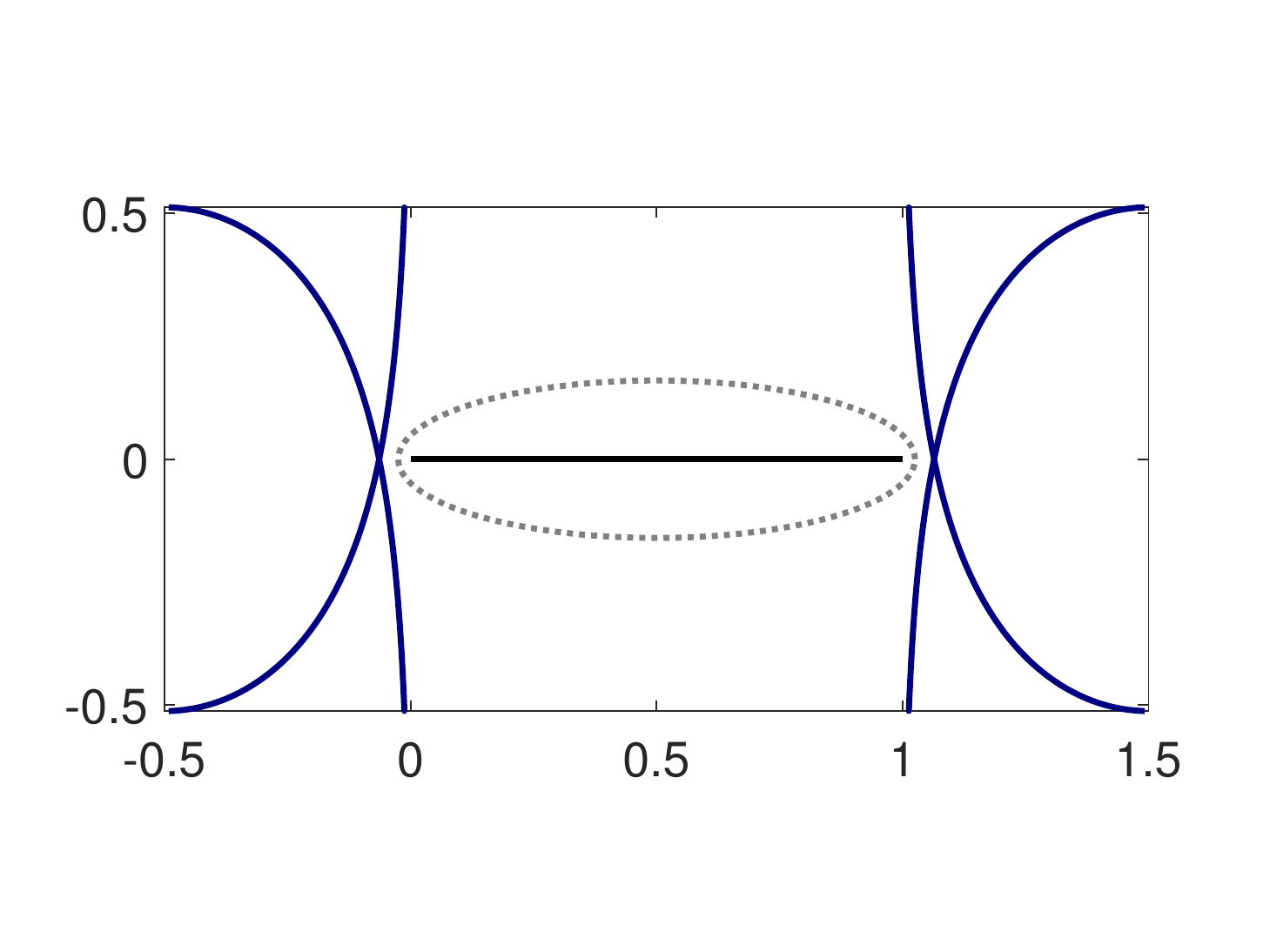}
\put (40,47) {$\displaystyle \rho_2=1.3702$}
\put (46,-5) {$\displaystyle {\rm Re}(\theta)$}
\put (-1.5,25) {\rotatebox{90}{$\displaystyle {\rm Im}(\theta)$}}
\end{overpic}
\end{minipage}
\caption{Bernstein ellipses (dashed grey contours) for the integrands in~\eqref{eqn:pvm1} (left) and in~\eqref{eqn:pvm2} (right) with $m=1$ and $\epsilon=0.1$. The integrand in~\eqref{eqn:pvm2} is analytic in a larger ellipse because the spectrum of $H$ is effectively deformed away (blue lines) from the integration contour (black line). The ellipse parameters $\rho_1$ and $\rho_2$ govern convergence rates for the Clenshaw--Curtis quadrature approximations in~\eqref{eqn:pvm1} and~\eqref{eqn:pvm2}, respectively.}
\label{fig:compare_contours2}
\end{figure}

\subsection{Computing transport properties and the functional calculus}
\label{Num_FC}

For the computation of general semigroups with error control using rectangular truncations, we refer the reader to \cite{colbrook2021semigroups,colbrook2021semigroupsfrac}. Related to the method we adopt here, many works use contour methods to invert the Laplace transform and solve time evolution problems, with a focus on parabolic PDEs \cite{gavrilyuk2001exponentially,weideman2010improved,weideman2007parabolic,lopez2006spectral,dingfelder2015improved,mclean2004time,sheen2003parallel,gavrilyuk2011exponentially}. An excellent survey of contour methods is provided in \cite{trefethen2014exponentially}.

In our case, the relevant Hamiltonians are bounded and the procedure is considerably simplified. For a holomorphic function $g$, Cauchy's integral formula yields
\begin{equation}
g(H)=\frac{1}{2\pi i}\inty{\gamma}{}{g(z)(H-z)^{-1}}{z},
\end{equation}
where $\gamma$ is a closed contour looping once around the spectrum. Transport properties are computed via the choice $g(z)=\exp(-i zt)$. Namely, given an initial wavefunction $\psi_0$, we wish to compute
\begin{equation}\label{func_calc_desc}
\psi(t)=\exp(-i Ht)\psi_0=\frac{1}{2\pi i}\inty{\gamma}{}{\exp(-i zt)\left[(H-z)^{-1}\psi_0\right]}{z}.
\end{equation}
The contour integral is computed using quadrature and approximations of the resolvent $(H-z)^{-1}$ via rectangular truncations as above. In particular, the rectangular truncation of the Hamiltonian is chosen adaptively through \textit{a posteriori} error bounds. This allows us to perform rigorous computations with error control that are guaranteed to be free from finite-size or truncation/discretisation effects, directly probing the transport properties of the infinite lattice. It is difficult to achieve error control or computations free from truncation effects via other methods since it can be challenging to predict how large the truncation needs to be \textit{a priori}.

Suppose that the spectrum is located in an interval $[a,b]\subset\mathbb{R}$. We take $\gamma$ to be a rectangular contour split into four line segments: two parallel to the imaginary axis with real parts $a-1$ and $b+1$ and two parallel to the real axis with imaginary parts $\pm \eta$ ($\eta>0$). Along these line segments we apply Gaussian quadrature with enough quadrature nodes for the desired accuracy (the number of nodes can be found by bounding the analytic integrand). Suppose that the weights and nodes for the quadrature rule applied to the whole of $\gamma$ are $\{w_j\}_{j=1}^N$ and $\{z_j\}_{j=1}^N$. Then the approximation of \eqref{func_calc_desc} is given by
\begin{equation}
\psi(t)\approx \sum_{j=1}^N \frac{w_j}{2\pi i} \exp(-i z_jt)\left[(H-z_j)^{-1}\psi_0\right].
\end{equation}
The vectors $(H-z_j)^{-1}\psi_0$ are computed using the adaptive method, which can be performed in parallel across the quadrature nodes. We also reuse these computed vectors for different times $t$. Numerically, this requires $\eta$ to not be too large due to the growth of the complex exponential in the complex plane. Suitable $N$ can be selected for a finite interval of desired times $t$.

\section{The Haldane model} \label{sec:Haldane}

In this work, we apply the methods just described to the Haldane model \cite{1988Haldane}. The Haldane model describes electrons hopping on a two-dimensional honeycomb lattice (\cref{fig:honeycomb}) in the presence of a periodic magnetic field with zero net flux. In this section, we first present basic features of the periodic bulk and edge Haldane models and their Bloch reductions in~\cref{sec:bulk_Haldane,sec:edge_Haldane}. We do this only for the reader's convenience since excellent reviews already exist in the literature \cite{2013FruchartCarpentier,2019MarcelliMonacoMoscolariPanati}. We then describe how we model defects and disorder in~\cref{sec:defects_and_disorder}. We then discuss edge states, and bulk and edge conductances, in~\cref{sec:bulk_conductance,sec:edge_conductance}, before making our computational goals more precise in~\cref{sec:comp_goals}. To improve readability, we postpone some long formulas to~\cref{sec:Haldane_details}.

\begin{figure}
\centering
\begin{tabular}{c}
\begin{overpic}[width=0.49\textwidth,clip,trim={0mm 0mm 0mm 0mm}]{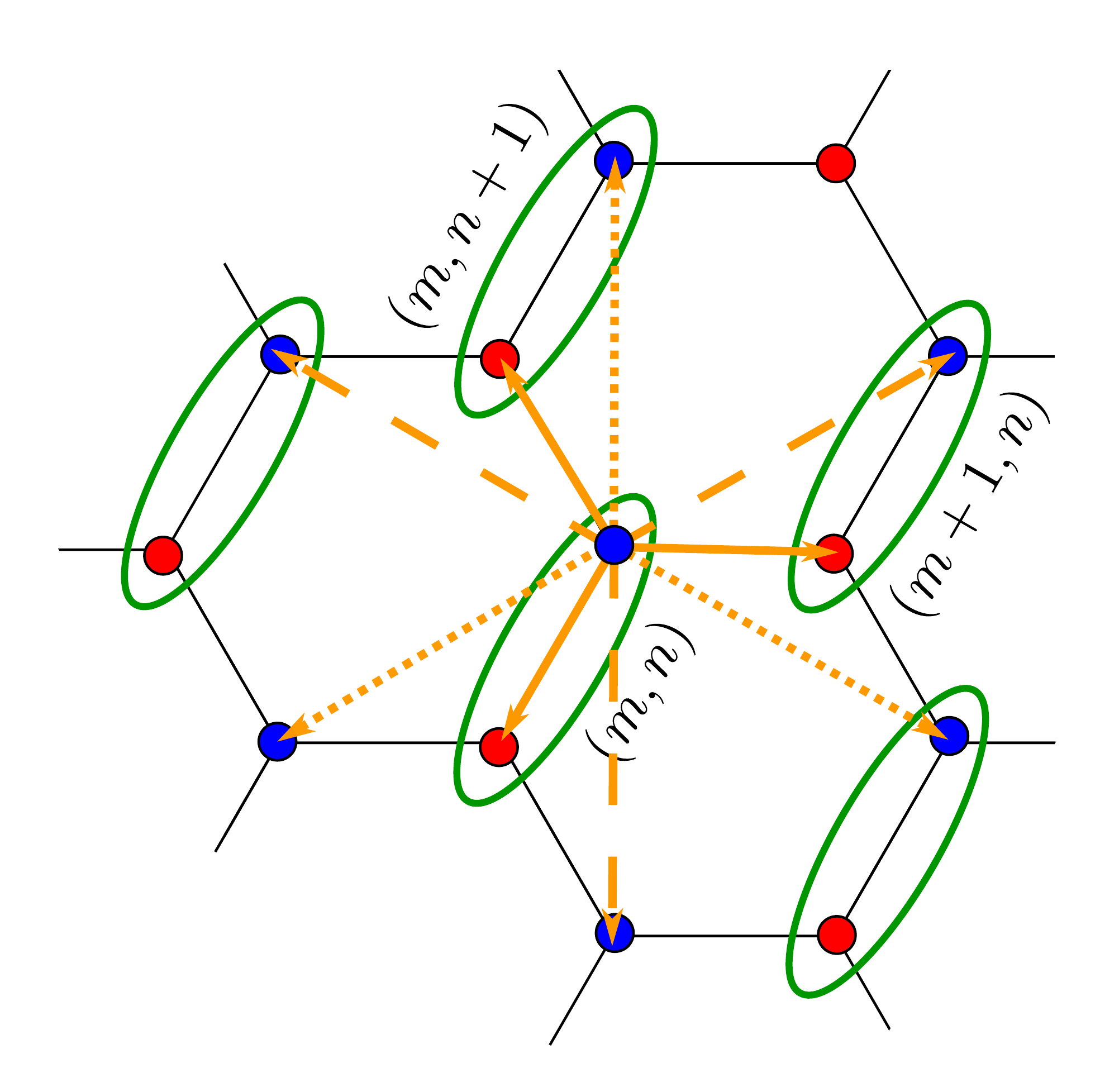}
\end{overpic}
\end{tabular}
    \caption{Illustration of hopping terms of the Haldane model \eqref{eq:inf_H_0}. Red and blue circles denote the $A$ and $B$ sublattices, respectively, and orange lines denote hopping starting at the $B$ site in the $(m,n)$th cell. Solid orange lines show nearest-neighbour hoppings, while dashed orange lines show next-nearest-neighbour hoppings. Short orange dashes correspond to next-nearest-neighbour hoppings with phase $e^{- i \phi}$, while long orange dashes correspond to next-nearest-neighbour hoppings with phase $e^{i \phi}$.} 
\label{fig:honeycomb}
\end{figure}

\subsection{The periodic bulk Haldane model}\label{sec:bulk_Haldane}

We model electrons in the bulk of the material as elements of the Hilbert space $\ell^2(\mathbb{Z}^2;\mathbb{C}^2)$, written $\psi : \vec{m} \mapsto \psi_{\vec{m}}$, where $\vec{m} = (m,n)$ and $\psi_{\vec{m}} = \left( \psi_{\vec{m}}^A, \psi_{\vec{m}}^B \right)^\top$. The quantity $| \psi_{\vec{m}}^\upsilon |^2$ is then the electron probability density on site $\upsilon\in \{A,B\}$ in the $\vec{m}$th cell. The bulk Haldane Hamiltonian $H_{B}$ is then \cite{1988Haldane}
\begin{equation} \label{eq:inf_H_0}
\begin{split}
    (H_{B} \psi)_{\vec{m}} := &\; t \begin{pmatrix} \psi_{m,n}^B + \psi_{m-1,n}^B + \psi_{m,n-1}^B \\ \psi_{m,n}^A + \psi_{m+1,n}^A + \psi_{m,n+1}^A \end{pmatrix} + V \begin{pmatrix} \psi^A_{\vec{m}} \\ - \psi^B_{\vec{m}} \end{pmatrix} \\
    &+ t' \begin{pmatrix} e^{i \phi} \left( \psi_{m,n+1}^A + \psi_{m-1,n}^A + \psi_{m+1,n-1}^A \right) + e^{- i \phi} \left( \psi^A_{m,n-1} + \psi^A_{m+1,n} + \psi^A_{m-1,n+1} \right) \\ e^{i \phi} \left( \psi^B_{m,n-1} + \psi^B_{m+1,n} + \psi^B_{m-1,n+1} \right) + e^{- i \phi} \left( \psi^B_{m,n+1} + \psi^B_{m-1,n} + \psi^B_{m+1,n-1} \right) \end{pmatrix}.
\end{split}
\end{equation}
Here, $t$ and $t' \in \field{R}$ are hopping amplitudes between nearest-neighbours and next-nearest-neighbours in the lattice, while $V \in \field{R}$ is a potential difference between sublattices, and $\phi \in [0,2\pi)$ is the complex phase of the next-nearest-neighbour hopping. Nearest-neighbour and next-nearest-neighbour hoppings are shown in \cref{fig:honeycomb}. When $t' \neq 0$ and $\phi \notin \{0,\pi\}$, the next-nearest-neighbour hoppings model a periodic non-zero magnetic flux through the material, whose average over any unit cell is zero.

The bulk Hamiltonian $H_{B}$ \eqref{eq:inf_H_0} is invariant under translations with respect to both components of $\vec{m}$, and can therefore be diagonalised via the Fourier transform \cite{ashcroft_mermin,kuchment,reed_simon_4}. Let $\Gamma^* := [0,2 \pi)^2$, then we write elements of $L^2(\Gamma^*;\mathbb{C}^2)$ as $\oldhat{\psi}: \vec{k} \mapsto \oldhat{\psi}(\vec{k})$, where $\vec{k} := (k_1,k_2)$ and $\oldhat{\psi}(\vec{k}) = \left( \oldhat{\psi}^A(\vec{k}), \oldhat{\psi}^B(\vec{k}) \right)^\top$. We introduce the Fourier transform $\mathcal{F} : \ell^2( \mathbb{Z}^2 ; \mathbb{C}^2 ) \rightarrow L^2( \Gamma^*;\mathbb{C}^2)$ and its inverse
\begin{equation} \label{eq:Fourier}
    \left( \mathcal{F} \psi \right)(\vec{k}) := \sum_{\vec{m} \in \mathbb{Z}^2} e^{- i \vec{k} \cdot \vec{m}} \psi_{\vec{m}}, \quad \left( \mathcal{F}^{-1} \oldhat{\psi} \right)_{\vec{m}} := \frac{1}{|\Gamma^*|} \inty{\Gamma^*}{}{ e^{i \vec{k} \cdot \vec{m}} \oldhat{\psi}(\vec{k}) }{\vec{k}}.
\end{equation}
Under the transformation \eqref{eq:Fourier}, the operator \eqref{eq:inf_H_0} takes the form \cite{1988Haldane}
\begin{equation} \label{eq:Bloch_bulk}
\begin{split}
    &\left( \left( \mathcal{F} H_{B} \mathcal{F}^{-1} \right) \oldhat{\psi} \right)(\vec{k}) = \oldhat{H}_{B}(\vec{k}) \oldhat{\psi}(\vec{k}), \\
    &\oldhat{H}_{B}(\vec{k}) := \begin{pmatrix} V + t' e^{i \phi} \left( e^{i k_2} + e^{- i k_1} + e^{i (k_1 - k_2)} \right) + c.c. & t \left( 1 + e^{- i k_1} + e^{- i k_2} \right) \\ t \left( 1 + e^{i k_1} + e^{i k_2} \right) & - V + t' e^{i \phi} \left( e^{- i k_2} + e^{i k_1} + e^{i (k_2 - k_1)} \right) + c.c. \end{pmatrix},
\end{split}
\end{equation} 
where $+ c.c.$ means add the complex conjugate of the term immediately before. Let $E_{\pm}(\vec{k})$ denote the ordered eigenvalues of $\oldhat{H}_{B}(\vec{k})$, known as the Bloch band functions. The spectrum of $H_{B}$ is then  
\begin{equation} \label{eq:H_bulk_spec}
    \text{Sp}( H_{B} ) = \text{Sp}_- \cup \text{Sp}_+, \quad \text{Sp}_\pm := \bigcup_{\vec{k} \in \Gamma^*} E_\pm(\vec{k}),
\end{equation}
where $E_{\pm}(\vec{k})$ is given by an explicit formula \eqref{eq:full_band_functions}. The associated (non-normalisable) eigenfunctions of $H_{B}$ are plane wave-like, given explicitly by $\Phi_{\pm}(\vec{k}) : \vec{m} \mapsto \Phi_{\pm,\vec{m}}(\vec{k})$, where
\begin{equation} \label{eq:bulk_modes}
    \Phi_{\pm,\vec{m}}(\vec{k}) := e^{i \vec{k} \cdot \vec{m}} \oldhat{\Phi}_\pm(\vec{k}),
\end{equation}
and $\oldhat{\Phi}_{\pm}(\vec{k})$ denotes an associated eigenvector of $\oldhat{H}_{B}(\vec{k})$ with eigenvalue $E_{\pm}(\vec{k})$.

\subsection{The periodic edge Haldane model} \label{sec:edge_Haldane}

We model electrons at an edge (specifically, a zig-zag edge) of the material as elements $\psi$ in the Hilbert space $\mathcal{H} := \ell^2(\mathbb{N}\times \mathbb{Z};\mathbb{C}^2)$. The Hamiltonian is again given by \eqref{eq:inf_H_0}, but we now impose a Dirichlet boundary condition
\begin{equation} \label{eq:edge_BC}
    \psi_{-1,n} = 0, \quad n \in \mathbb{Z}.
\end{equation}
We denote the Hamiltonian \eqref{eq:inf_H_0} subject to the boundary condition \eqref{eq:edge_BC} by $H_{E}$.

The edge Hamiltonian $H_{E}$ is invariant under translations with respect to $n$, so it is natural to take a partial Fourier transform. Let $L^2([0,2\pi);\ell^2(\mathbb{N};\mathbb{C}^2))$ denote the space of functions $\tilde{\psi} : k \mapsto \tilde{\psi}(k)$, where $\tilde{\psi}(k) : m \mapsto \tilde{\psi}_m(k)$ and $\tilde{\psi}_m(k) = \left( \tilde{\psi}_m^A(k) , \tilde{\psi}_m^B(k) \right)^\top$, such that $\inty{0}{2 \pi}{ \sum_{m = 0}^\infty | \tilde{\psi}_m(k) |^2 }{k} < \infty$. We introduce the partial Fourier transform $\mathcal{G} : \ell^2(\mathbb{N} \times \mathbb{Z} ; \mathbb{C}^2) \rightarrow L^2( [0,2 \pi) ; \ell^2(\mathbb{N};\mathbb{C}^2))$ and its inverse
\begin{equation} \label{eq:Fourier_edge}
    \left( \mathcal{G} \psi \right)_m(k) := \sum_{n \in \mathbb{Z}} e^{- i k n} \psi_{\vec{m}}, \quad \left( \mathcal{G}^{-1} \tilde{\psi} \right)_{\vec{m}} := \frac{1}{2 \pi} \inty{0}{2 \pi}{ e^{i k n} \tilde{\psi}_m(k) }{k}.
\end{equation}
The action of the operator $H_{E}$ under the transformation \eqref{eq:Fourier_edge} is then
\begin{equation}
    \left( \left( \mathcal{G} H_{E} \mathcal{G}^{-1} \right) \tilde{\psi} \right)_m(k) = \left( \oldhat{H}_{E}(k) \tilde{\psi}(k) \right)_{m},
\end{equation}
where $\oldhat{H}_{E}(k)$ is given by \eqref{eq:edge_H}, subject to the boundary condition $\tilde{\psi}_{-1}(k) = 0$. The spectrum of $H_{E}$ is then
\begin{equation} \label{eq:H_edge_spec}
    \text{Sp}( H_{E} ) = \bigcup_{k \in [0,2\pi)} \text{Sp}( \oldhat{H}_{E}(k) ),
\end{equation}
and the associated eigenfunctions are plane wave-like parallel to the edge, given explicitly by $\Phi({k}) : \vec{m} \mapsto \Phi_{\vec{m}}({k})$, where
\begin{equation} \label{eq:edge_modes}
    \Phi_{\vec{m}}(k) := e^{i k n} \oldhat{\Phi}_m(k),
\end{equation}
and $\oldhat{\Phi}(k) : m \mapsto \oldhat{\Phi}_m(k)$ denotes any eigenfunction of $\oldhat{H}_{E}(k)$. 

Note that, unlike $\oldhat{H}_{B}(\vec{k})$, the Bloch-reduced operator in this case, $\oldhat{H}_{E}(k)$, generally cannot be diagonalised explicitly. However, we can nonetheless make some general observations. By the Weyl criterion (see Theorem 5.10 of \cite{HislopSigal}), we clearly have that\footnote{Suppose $\lambda \in \text{Sp}( H_B )$. Then there exists a sequence $\{f_n\} \in \ell^2(\mathbb{Z}^2;\mathbb{C}^2)$ such that $\| f_n \| = 1$ and $\| ( H_B - \lambda ) f_n \| \rightarrow 0$ as $n \rightarrow \infty$. But since $H_E$ and $H_B$ act identically for $m > 0$, and since $H_B$ is periodic, we can always translate the $f_n$ in order to generate a sequence $\{ g_n \} \in \ell^2(\mathbb{N}\times\mathbb{Z};\mathbb{C}^2)$ such that $\| g_n \| = 1$ and $\| ( H_E - \lambda ) g_n \| \rightarrow 0$, and hence $\lambda \in \text{Sp}( H_E )$.}
\begin{equation}
    \text{Sp}( H_{B} ) \subset \text{Sp}( H_{E} ).
\end{equation}
However, equality does not hold in general, because $H_{E}$ may have additional spectrum due to edge states: (non-normalisable) eigenfunctions \eqref{eq:edge_modes} of $H_{E}$ arising from the truncation \eqref{eq:edge_BC} which decay rapidly away from the edge \cite{1982Halperin,1993Hatsugai,2013GrafPorta}. Edge states are closely tied to topological properties of the Haldane model and are discussed in more detail in \cref{sec:edge_conductance}.

\subsection{Modeling defects and disorder} \label{sec:defects_and_disorder}

We model onsite disorder by adding an additional potential term 
\begin{equation} \label{eq:V_dis}
    ( V_{d} \psi )_{\vec{m}} = \begin{pmatrix} V^A_{\vec{m}} \psi^A_{\vec{m}} \\ V^B_{\vec{m}} \psi^B_{\vec{m}} \end{pmatrix}
\end{equation}
to $H_{B}$ and $H_{E}$, where the $V^\upsilon_{\vec{m}}$ are independently drawn from a uniform distribution with mean $0$ and width $w$
\begin{equation}
    V^\upsilon_{\vec{m}} \sim \mathcal{U}(0,w), \quad \upsilon \in \{A,B\}, \quad \vec{m} \in \mathbb{Z}^2.
\end{equation}
Clearly, we have $\| V_d \| \leq w/2$, where $\| \cdot \|$ denotes the operator norm. Note that in this work, we only compute physical properties for individual realisations of disorder; we do not attempt to compute statistical properties over many realisations. We model missing atom defects by setting the wave-function $\psi$ equal to zero at the missing sites. We write the Hamiltonians $H_{B}$ and $H_{E}$ with defects and/or disorder as $H_{B,d}$ and $H_{E,d}$, respectively. Note that $H_{B,d}$ and $H_{E,d}$ cannot be Bloch reduced. To compute their spectral properties, we must work with the infinite-dimensional operators directly.

\subsection{Bulk Hall conductance} \label{sec:bulk_conductance}

Physically speaking, eigenfunctions of the operators $H_B, H_E, H_{B,d}$, and $H_{E,d}$ correspond to states that can be occupied by electrons, with energies given by the associated eigenvalues. At zero temperature, there exists a threshold such that every state with energy below the threshold, and no state with energy above the threshold, is occupied by an electron. This threshold is known as the Fermi level. In what follows, we assume that the bands $\text{Sp}_-$ and $\text{Sp}_+$ of $H_B$ are separated by a gap, and that the addition of defects and disorder does not close this gap, so that $H_B$ and $H_{B,d}$ have a common gap $\Delta$. We assume further that the Fermi level lies in $\Delta$. Under these assumptions, $H_B$ and $H_{B,d}$ describe (bulk) insulators.

To build intuition, we focus first on the case without defects or disorder. The bulk conductance measures the current excited in the bulk of a material by an applied electric field. The linear coefficient of the conductance can be calculated analytically through linear response theory and is known as the Kubo formula \cite{ashcroft_mermin,Schulz-Baldes1998}. The part of the conductance parallel to the field vanishes in insulators, but the conductance may have a non-zero transverse (Hall) component. In natural units\footnote{So that the electron charge and Planck's constant both equal $1$.}, and in the limit of zero frequency and dissipation, this component takes the form \cite{1982ThoulessKohmotoNightingaledenNijs},
\begin{equation}
    \sigma_{B} = \frac{i}{2 \pi} \inty{\Gamma^*}{}{ \frac{ \bra{ \oldhat{\Phi}_-(\vec{k}) } \de_{k_1} \oldhat{H}_{B} (\vec{k}) \ket{ \oldhat{\Phi}_+(\vec{k}) } \bra{ \oldhat{\Phi}_+(\vec{k}) } \de_{k_2} \oldhat{H}_{B} (\vec{k}) \ket{ \oldhat{\Phi}_-(\vec{k}) } - ( 1 \leftrightarrow 2 ) }{ (E_+(\vec{k}) - E_-(\vec{k}))^2 } }{\vec{k}},
\end{equation}
where $(1 \leftrightarrow 2)$ is shorthand for the term immediately before, with $1$ replaced everywhere by $2$, and vice versa. After a series of manipulations \cite{1982ThoulessKohmotoNightingaledenNijs}, we find
\begin{equation}
    \sigma_{B} = \frac{ i }{ 2 \pi } \inty{\Gamma^*}{}{ \de_{k_1} \ip{ \oldhat{\Phi}_-(\vec{k}) }{ \de_{k_2} \oldhat{\Phi}_-(\vec{k}) } - \de_{k_2} \ip{ \oldhat{\Phi}_-(\vec{k}) }{ \de_{k_1} \oldhat{\Phi}_-(\vec{k}) } }{\vec{k}}.
\end{equation}
The integrand on the right-hand side is the Berry curvature \cite{berry} of the $-$ band, and its integral over the Brillouin zone must be an integer multiple of $2 \pi$ \cite{1982ThoulessKohmotoNightingaledenNijs,1988Haldane,nakahara2018geometry}. Thus we have
\begin{equation} \label{eq:chern}
    \sigma_{B} = c_-,
\end{equation}
where $c_-$ is an integer known as the Chern number. The Chern number, being an integer, cannot change value continuously as model parameters are varied and hence remains fixed as long as the bulk gap does not close. The Haldane phase diagram, which describes the values the Chern number can take as the model parameters are varied, can be calculated analytically when the model is periodic \cite{1988Haldane}. Whenever the Hall conductance is non-zero, we say the model is in its topological phase. 

We now consider the case of defects and/or disorder, which prevent Bloch reduction. A convenient expression of the Kubo formula is \cite{Avron1994,2005ElgartGrafSchenker}
\begin{equation} \label{eq:non_periodic_Kubo_lambdas}
    \sigma_{B} = - 2 \pi i \Tr\left\{ P_{B} \left[ [ P_{B}, \Lambda_1 ], [ P_{B} , \Lambda_2 ] \right]\right\},
\end{equation}
where $\Tr$ denotes the trace in $\ell^2(\mathbb{Z}^2;\mathbb{C}^2)$, $P_{B}$ denotes the spectral projection for the part of the spectrum of $H_{B,d}$ below $\Delta$, $\Lambda_1$ and $\Lambda_2$ denote characteristic functions for the sets $\{ \vec{m} \in \mathbb{Z}^2 : m < 0 \}$ and $\{ \vec{m} \in \mathbb{Z}^2 : n < 0 \}$, respectively. Note that the operator on the right-hand side of \eqref{eq:non_periodic_Kubo_lambdas} is not obviously trace-class. To see that it is, note that Combes--Thomas estimates \cite{1973CombesThomas} imply that $[ P_{B} , \Lambda_1 ]$ acts trivially on sites away from the line $m = 0$, while $[ P_{B} , \Lambda_2 ]$ acts trivially on sites away from the line $n = 0$. It follows that the operator on the right-hand side of \eqref{eq:non_periodic_Kubo_lambdas} acts trivially on sites away from the origin and is hence trace-class.

Finally, although we have so far assumed a spectral gap for $H_{B,d}$ in this section, we expect that definition \eqref{eq:non_periodic_Kubo_lambdas} remains valid and is an integer, even when $H_{B,d}$ has no spectral gap, but does exhibit dynamical (Anderson) localisation in a spectral interval, following \cite{2005ElgartGrafSchenker} who proved this in the setting of the quantum Hall effect. Our computational methods do not rely on the existence of a spectral gap, and can be used in this setting as well.

\subsection{Edge states, edge conductance, and edge wave-packets} \label{sec:edge_conductance}

Throughout this section, we continue to assume that $H_B$ and $H_{B,d}$ have a common gap $\Delta$. Recall that it does not follow that $\Delta$ is a spectral gap of $H_{E}$ or $H_{E,d}$, because edge states with energies in the gap may occur. 

To build intuition, we again start by considering the periodic setting. Edge states of $H_{E}$ are extensions \eqref{eq:edge_modes} of bound (normalisable) states associated to discrete eigenvalues $E(k)$ of the operators $\oldhat{H}_{E}(k)$ acting on $\ell^2(\mathbb{N};\mathbb{C}^2)$. As $k$ varies through the interval $[0,2 \pi)$, these eigenvalues sweep out intervals $\ran_{k \in [0,2 \pi)} E(k)$ in the spectrum of $H_{E}$ \eqref{eq:H_edge_spec}. The maps $E : k \mapsto E(k)$ are known as the dispersion relations of the edge states. Superposing edge states with $k$ values near to some $k_0$ yields localised wave-packets which propagate along the edge with group velocity given by $E'(k_0)$ \cite{1993Hatsugai,1999Schulz-BaldesKellendonkRichter,2002KellendonkRichterSchulz-Baldes}. 

The current carried by edge states with energies in $\Delta$ is measured by the edge conductance. More precisely, let $\Delta'$ denote any subinterval of $\Delta$, and let $\chi_{\Delta'}$ denote the characteristic function for the interval $\Delta'$. The projection onto edge modes with energies in $\Delta'$ is then given by $P_E := \chi_{\Delta'}(H_E)$. In natural units, the edge conductance can then be defined by \cite{1999Schulz-BaldesKellendonkRichter,2002KellendonkRichterSchulz-Baldes}
\begin{equation}
    \sigma_E = \frac{1}{|\Delta'|} \inty{0}{2 \pi}{\tilde{\Tr}\left\{ P_E \de_k \oldhat{H}_{E}(k) \right\}}{k},
\end{equation}
where $\tilde{\Tr}$ denotes the trace in $\ell^2(\mathbb{N};\mathbb{C}^2)$. In the limit where $|\Delta'| \rightarrow 0$ so that $\frac{ \chi_{\Delta'} }{ |\Delta'| } \rightarrow \delta_{\lambda}$ for some $\lambda \in \Delta$, $\sigma_E$ can be computed analytically as follows. Let $\nu$ denote the number of edge state dispersion relations which cross $\lambda$, assigning $+1$ to those whose slopes are positive, and $-1$ to those whose slopes are negative. Then
\begin{equation}
    \sigma_E = \nu.
\end{equation}
The principle of bulk-edge correspondence \cite{1993Hatsugai,2013GrafPorta,1999Schulz-BaldesKellendonkRichter,2002KellendonkRichterSchulz-Baldes} states that $\sigma_B = \sigma_E$, and hence the integer $\nu$ equals the bulk Chern number $c_-$ \eqref{eq:chern}. A simple consequence of bulk-edge correspondence is that, since $\lambda \in \Delta$ was arbitrary, the edge Hamiltonian $H_{E}$ must have spectrum filling the whole bulk gap $\Delta$ whenever the bulk Chern number is non-zero\footnote{The spectrum filling the bulk gap is actually absolutely continuous, even in the presence of disorder; see \cite{Bols2021}.}.

In the presence of defects and/or disorder, a convenient expression for the edge conductance in an interval $\Delta' \subset \Delta$ (we assume $\Delta$ is again a gap for $H_{B,d}$) is
\begin{equation} \label{eq:non_periodic_edge_lambda}
    \sigma_E = \frac{- 2 \pi i}{|\Delta'|} \Tr \left\{P_E [H_{E,d},\Lambda_1]\right\},
\end{equation}
where $\Tr$ denotes the trace in $\ell^2(\mathbb{N} \times \mathbb{Z};\mathbb{C}^2)$, $P_E = \chi_{\Delta'}(H_{E,d})$, and $\Lambda_1$ again denotes the characteristic function for the set $\{ \vec{m} \in \mathbb{Z}^2 : m < 0 \}$; note that $m$ is the co-ordinate parallel to the edge. Just as with \eqref{eq:non_periodic_Kubo_lambdas}, it is not immediately obvious that the trace on the right-hand side of \eqref{eq:non_periodic_edge_lambda} is well-defined. That it is follows from decay of the matrix elements of $P_E$ away from the edge, and of $[H_{E,d},\Lambda_1]$ away from the line $m = 0$ \cite{1999Schulz-BaldesKellendonkRichter,2002KellendonkRichterSchulz-Baldes,2004KellendonkSchulz-Baldes,2005ElgartGrafSchenker}.

Note that, just like the bulk conductance, the edge conductance can be defined and is expected to be quantised, even when $H_{B,d}$ has no spectral gap, but does exhibit dynamical (Anderson) localisation. The definition of the edge conductance in this setting has to be slightly modified from \eqref{eq:non_periodic_edge_lambda}, however \cite{2005ElgartGrafSchenker}. Our computational methods can easily be generalised to this setting.

In the presence of defects and/or disorder, the Bloch decomposition \eqref{eq:H_edge_spec} is no longer valid. It follows that it no longer makes sense to form edge wave-packets by superposing edge states with nearby wavenumbers, or calculate their group velocities from the edge state dispersion relation. However, the persistence of the edge conductance in this regime suggests that localised initial data with spectral measure concentrated in the bulk gap will still propagate along the edge coherently. This remarkable behaviour has been observed numerically (see, for example, \cite{bal2021edge,Michala2021}), and even experimentally across various model systems \cite{2009WangChongJoannopoulosSoljacic,2013Rechtsmanetal,2015SusstrunkHuber,2017DelplaceMarstonVenaille}. Such initial data, which we again refer to as edge wave-packets, can be obtained by multiplying approximate edge states, i.e., approximate eigenfunctions of $H_{E,d}$ with energies in the bulk gap, by a smooth decaying function such as a Gaussian.

\subsection{Precise statement of TI physical properties to compute} \label{sec:comp_goals}

We can now re-state the computational problems (PA)--(PD) referred to in the introduction more precisely.
\begin{itemize}
    \item[(PA)] Numerically compute the \textbf{bulk} and \textbf{edge conductances}, defined for $H_{B,d}$ and $H_{E,d}$ through formulas \eqref{eq:non_periodic_Kubo_lambdas} and \eqref{eq:non_periodic_edge_lambda}, respectively.
    \item[(PB)] Numerically compute the \textbf{edge states} of $H_{E}$ and their associated \textbf{dispersion relations} $E : k \mapsto E(k)$ by computing the discrete eigenvalues and associated bound states of the Bloch-reduced operators $\oldhat{H}_{E}(k)$ \eqref{eq:edge_H}.
    \item[(PC)] Numerically compute \textbf{approximate edge states} and \textbf{edge wave-packets} (defined by multiplying approximate edge states by a smooth, decaying function) of $H_{E,d}$, and their \textbf{spectral measures}.
    \item[(PD)] Numerically compute the \textbf{dynamics} of edge wave-packets of $H_{E,d}$.
\end{itemize}

\section{Results}
\label{sec:resultshjhj}

In this section, we illustrate the utility of our methods by giving several numerical results.  

\subsection{Bulk and edge conductances} \label{sec:res:conductance}

We begin by examining the bulk and edge conductances defined in \cref{sec:bulk_conductance,sec:edge_conductance}, respectively. We compute spectral projections using the method of \cref{Num_spec_proj}.

\begin{figure}
\centering
\begin{minipage}[b]{0.48\textwidth}
\begin{overpic}[width=0.99\textwidth]{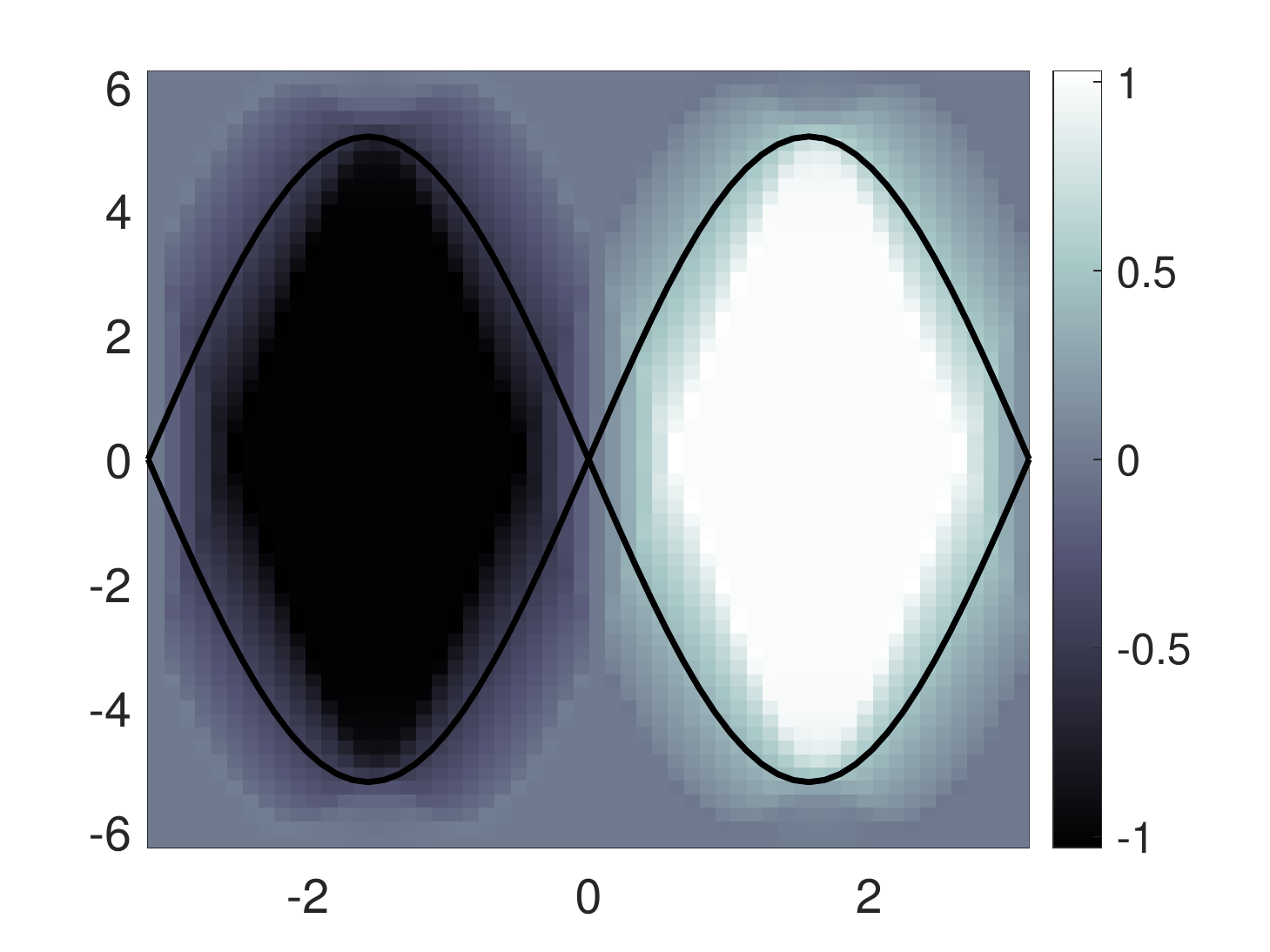}
\put (45,-2) {$\displaystyle \phi$}
\put (-2,39) {$\displaystyle V$}
\end{overpic}
\end{minipage}
\begin{minipage}[b]{0.48\textwidth}
\begin{overpic}[width=0.99\textwidth]{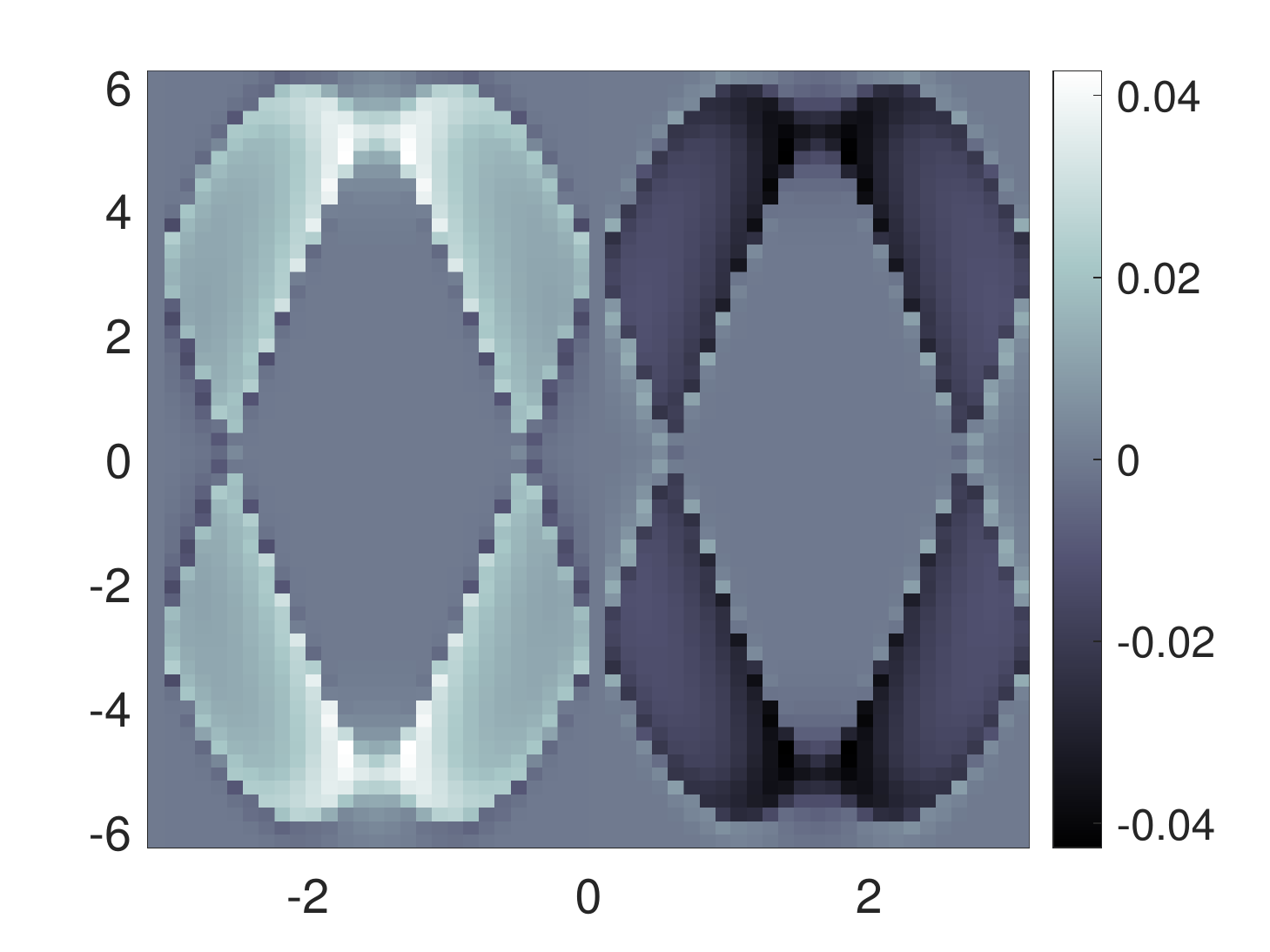}
\put (45,-2) {$\displaystyle \phi$}
\put (-2,39) {$\displaystyle V$}
\end{overpic}
\end{minipage}
\caption{Topological phase plot for $H_{B,d}$ (left panel) with uniform disorder of width $w=0.2$ (see~\cref{sec:defects_and_disorder}). The parameters $t$ and $t'$ are fixed at $1$ and $0.2$, respectively, while $V$ and $\phi$ are varied. The phase plot for $H_B$ is identical apart from small discrepancies caused by numerical under-resolution along the topological phase boundary; the difference is shown in the right panel.  On the left panel we overlay the curves $\pm 3 \sqrt{3} t' \sin \phi$ which mark the boundaries of the phase regions computed in the periodic case by Haldane \cite{1988Haldane}.}
\label{fig:conductivity_phase_diagrams}
\end{figure}

In the left panel of \cref{fig:conductivity_phase_diagrams}, we show results of numerically computing the bulk Hall conductance of the Haldane model with disorder $H_{B,d}$. In this experiment, the parameters $t = 1$, $t' = 0.2$, and the disorder width $w = 2$ are fixed, while $V$ and $\phi$ vary. The right panel of \cref{fig:conductivity_phase_diagrams} shows the difference between the phase diagram in the left panel and that of $H_B$, the bulk Haldane model without disorder. The difference is within the chosen numerical tolerance, confirming the stability of the Haldane model's bulk Hall conductance even in the face of disorder. The computational results are in excellent agreement with the phase diagram computed analytically in Haldane's original work \cite{1988Haldane}. 

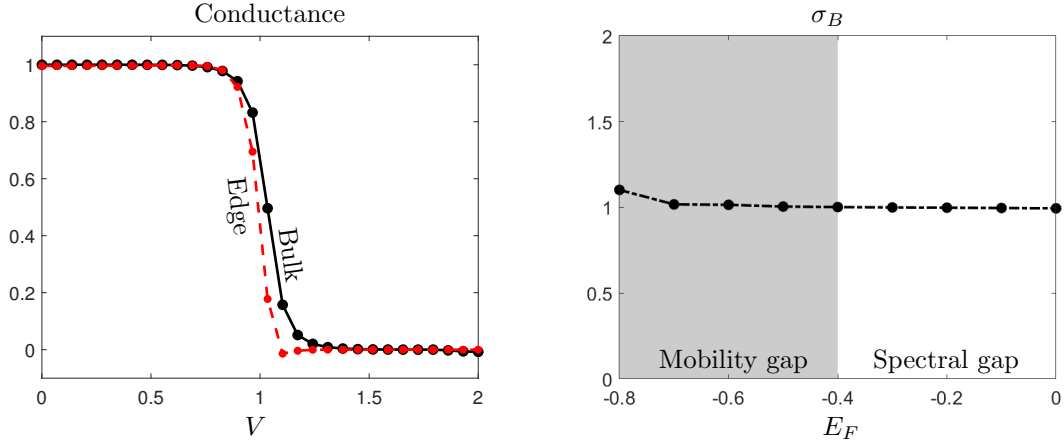
\begin{figure}
\centering
\begin{minipage}[b]{0.48\textwidth}
\begin{overpic}[width=0.99\textwidth]{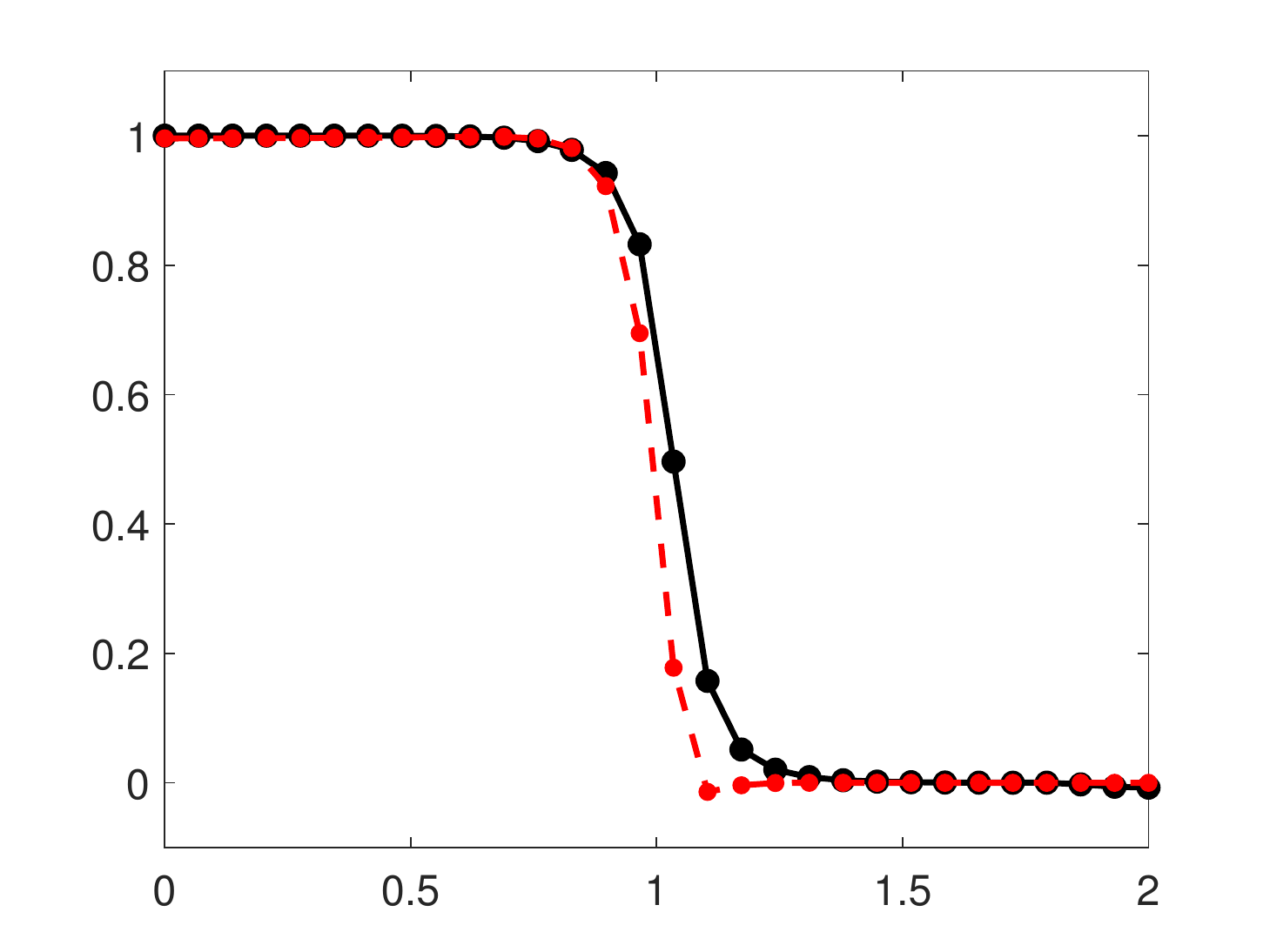}
\put (49,-2) {$\displaystyle V$}
\put (40,72) {Conductance}
\put (55,35) {\rotatebox{-82}{Bulk}}
\put (45,44) {\rotatebox{-83}{Edge}}
\end{overpic}
\end{minipage}
\begin{minipage}[b]{0.48\textwidth}
\begin{overpic}[width=0.99\textwidth]{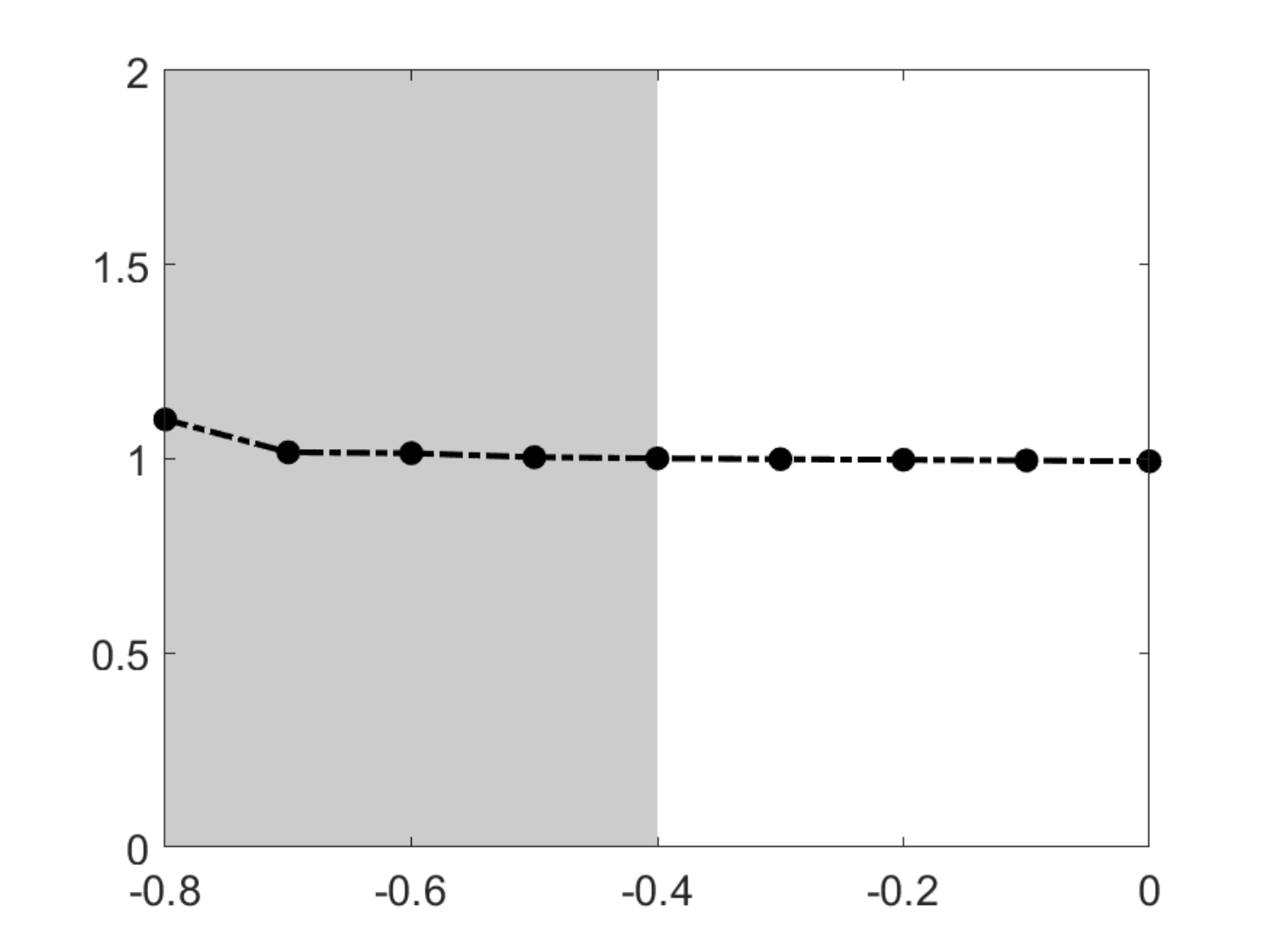}
\put (49,-2) {$\displaystyle E_F$}
\put (47,72) {$\displaystyle \sigma_B$}
\put (58,10) {Spectral gap}
\put (20,10) {Mobility gap}
\end{overpic}
\end{minipage}
\caption{The Bulk and Edge Hamiltonians with parameters $t=1$, $t'=0.2$, and $\phi=\pi/2$ exhibit a topological phase transition as $V$ is increased from $0$ to $2$, during which both the bulk and edge conductance (see~\eqref{eq:non_periodic_Kubo_lambdas}) switch from $1$ to $0$ (left panel). When a uniform random potential with $w=1.8$ (see~\cref{sec:defects_and_disorder}) is added to $H_B$, $\sigma_B$ remains stable as the Fermi level is varied through the spectral gap and regions of the spectrum exhibiting localised eigenstates, which do not support conduction (right panel)~\cite{2005ElgartGrafSchenker}.}
\label{fig:conductivity_phaseT}
\end{figure}

The bulk Hall conductance of $H_B$ is plotted as a function of $V$ in the left panel of~\cref{fig:conductivity_phaseT} and compared with the edge conductance $\sigma_E$ (see~\eqref{eq:non_periodic_edge_lambda}). The parameters $t=1$, $t'=0.2$, and $\phi=\pi/2$ are fixed in this computation. As $V$ is increased from $0$ to $2$, both the bulk and edge conductance exhibit a topological phase transition, and switch from $1$ to $0$. When a uniform random potential with $w=1.8$ (see~\cref{sec:defects_and_disorder}) is added to $H_B$, the spectral gap shrinks. However, the additional spectrum is typically highly localised and we expect that it does not contribute to the conductance. The right panel of~\cref{fig:conductivity_phaseT} demonstrates that $\sigma_B$ remains stable even as the Fermi level varies through a spectral gap and a so-called ``mobility gap",\footnote{Here we use the convenient terminology ``mobility gap'' to refer to an interval where the Hamiltonian has spectrum but the associated spectral projection does not contribute to conduction. We are not aware of any rigorous result proving existence of such a regime for the Haldane model, although such a regime is known to occur in models of the quantum Hall effect \cite{Germinet2007}.} where the Hamiltonian has spectrum but the associated spectral projection does not contribute to the Hall conductance \cite{2005ElgartGrafSchenker}. Approximate eigenstates of $H_{B,d}$ associated with two spectral regimes are plotted and compared with approximate eigenstates of $H_B$ in~\cref{fig:global_vs_local_states}. The states associated with the mobility gap are highly localised and do not contribute to conductance (see~\eqref{fig:global_vs_local_states}, bottom), while the states corresponding to points in the spectrum of $H_B$ are not localised (see~\cref{fig:global_vs_local_states}, top). 

\begin{figure}
\centering
\begin{minipage}[b]{0.48\textwidth}
\begin{overpic}[width=0.99\textwidth]{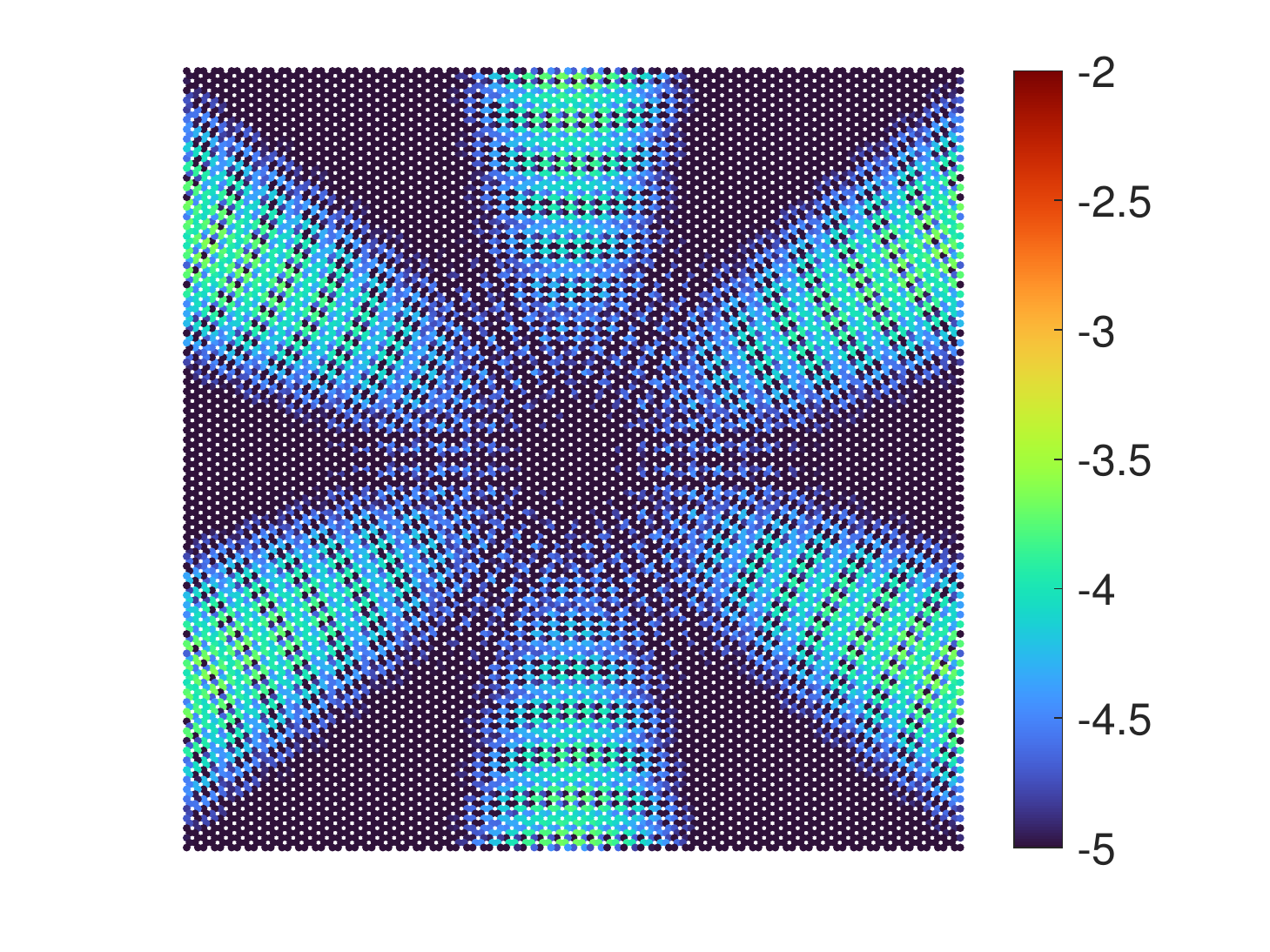}
\put (92,48) {\rotatebox{-90}{$\displaystyle \log(|\psi|^2)$}}
\end{overpic}
\end{minipage}
\begin{minipage}[b]{0.48\textwidth}
\begin{overpic}[width=0.99\textwidth]{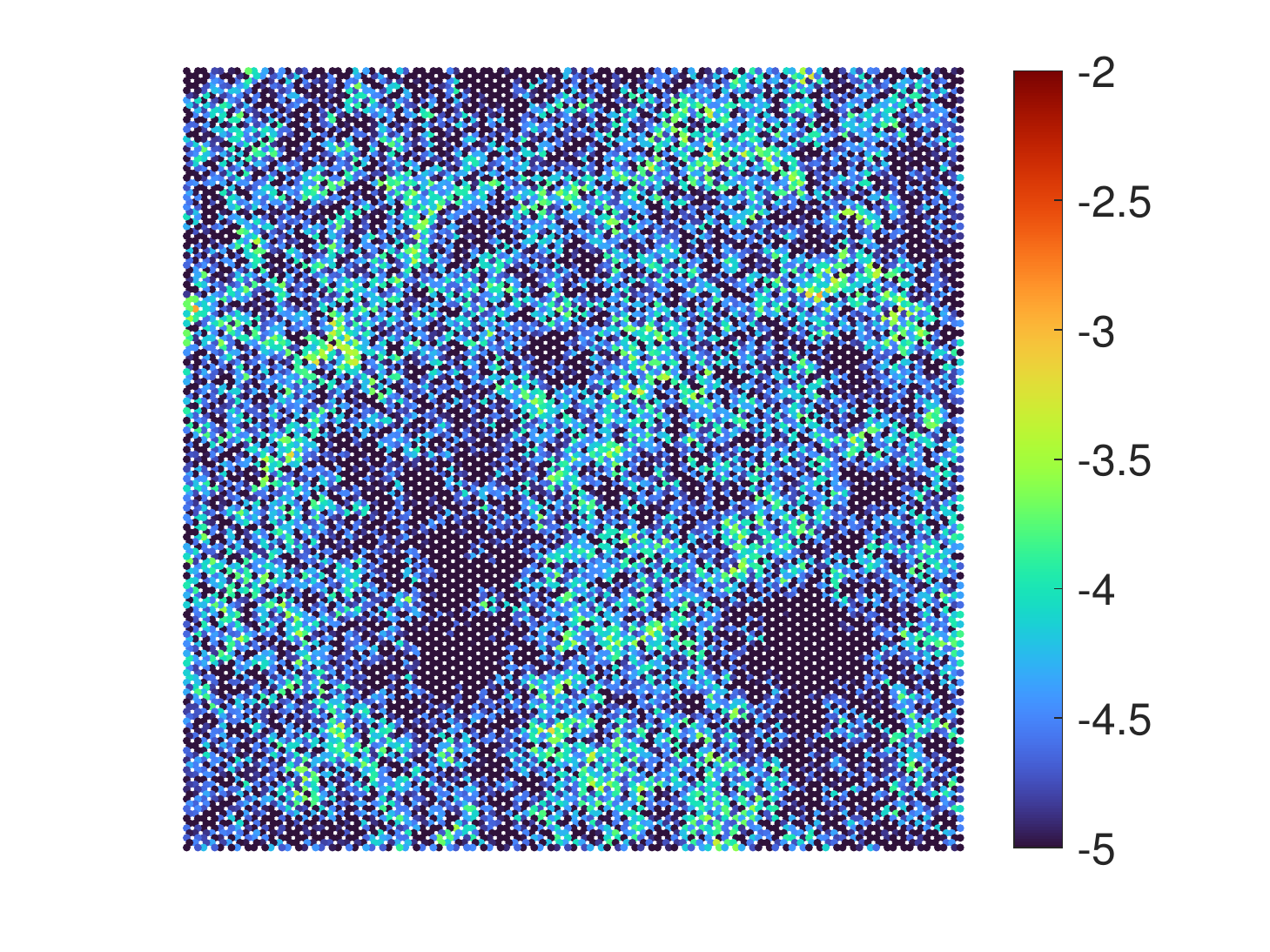}
\put (92,48) {\rotatebox{-90}{$\displaystyle \log(|\psi|^2)$}}
\end{overpic}
\end{minipage}
\begin{minipage}[b]{0.48\textwidth}
\begin{overpic}[width=0.99\textwidth]{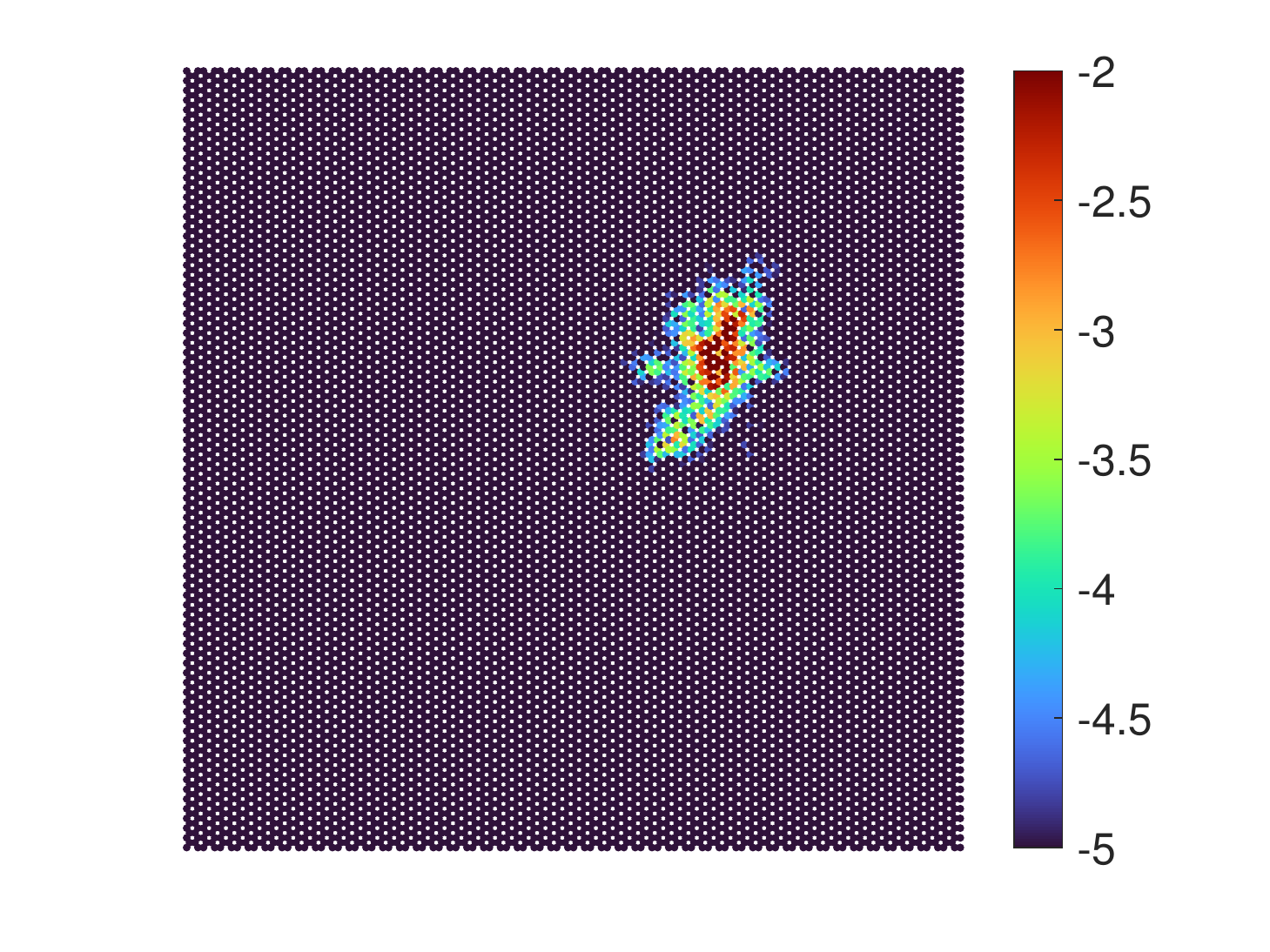}
\put (92,48) {\rotatebox{-90}{$\displaystyle \log(|\psi|^2)$}}
\end{overpic}
\end{minipage}
    \caption{Approximate eigenstate of $H_B$ with parameters $t=1$, $t'=0.2$, $\phi=\pi/2$ associated with point $-1.5$ in $\mathrm{Sp}(H_B)$ (top left). The approximate eigenstate has rotational symmetry because it is computed using rectangular truncations in the radial direction. Approximate eigenstates of $H_{B,d}$ with $w=1.8$ associated with points $-1.5$ (top right) and $-0.6$ (bottom) in $\mathrm{Sp}(H_{B,d})$. We show a truncated square portion of the computed approximate state in each case.}
\label{fig:global_vs_local_states}
\end{figure}

\subsection{Edge states at a periodic edge and their dispersion relations}
\label{sec:res:11}

We now show results of computing edge states and their dispersion relations, as discussed in \cref{sec:edge_conductance}. To do this, we compute the spectrum and associated eigenfunctions (when the spectrum is discrete) of the infinite-dimensional operators $\oldhat{H}_E(k)$ \eqref{eq:edge_H} using the methods of \cref{Num_spectra}.

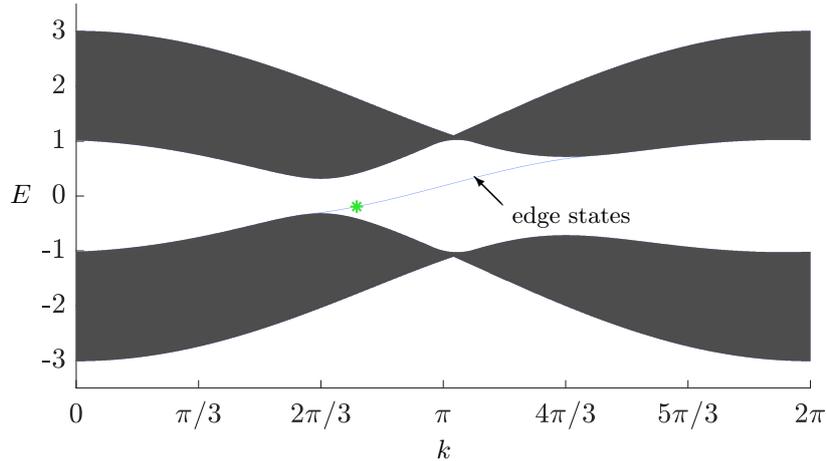
\begin{figure}
\centering
  \begin{minipage}[b]{0.8\textwidth}
    \begin{overpic}[width=1\textwidth]{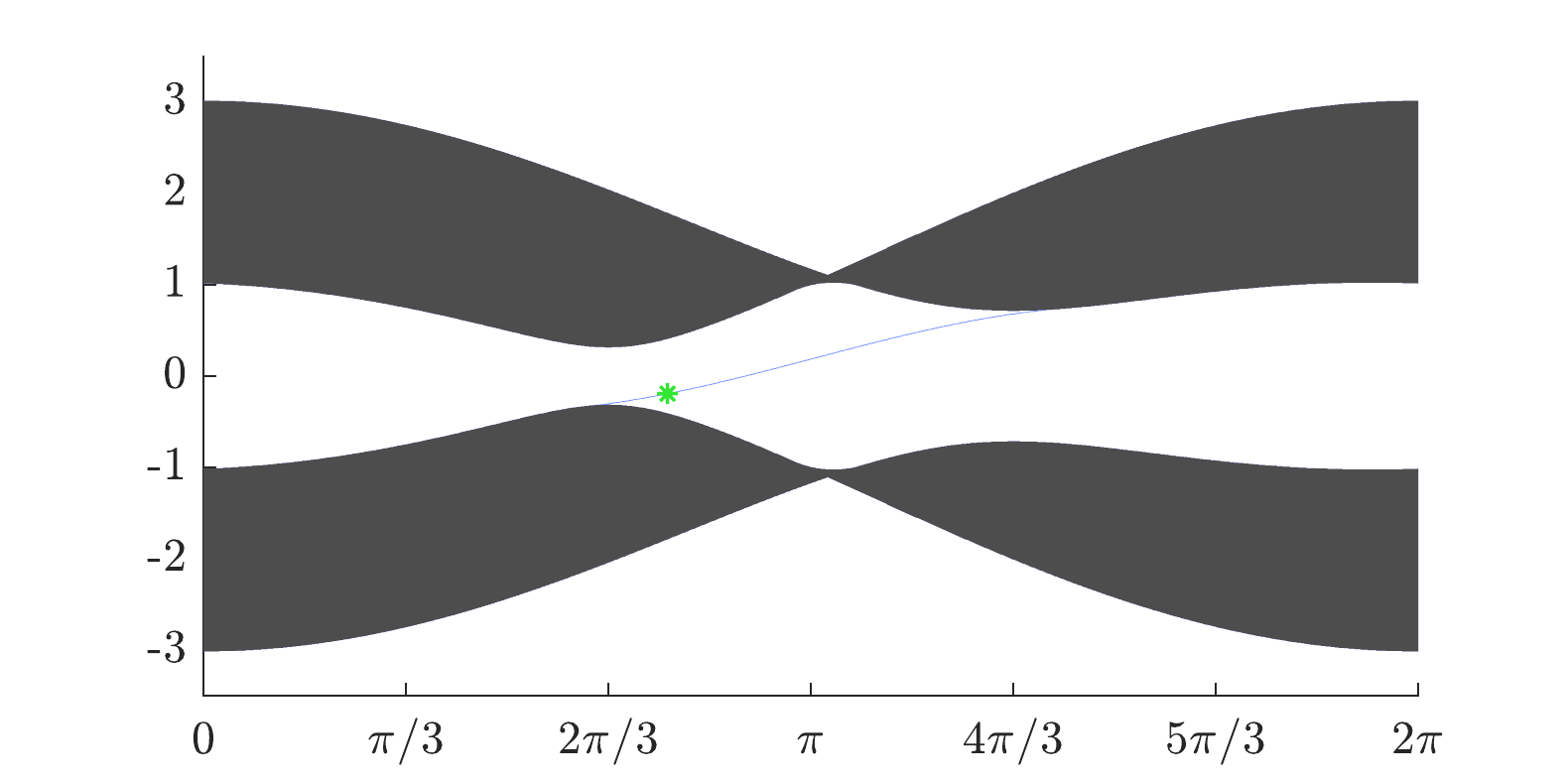}
		\put(58,25){\vector(-1,1){3}}
		\put(59,23){\small{}edge states}
		\put(51,-2){$k$}
		\put(6,25.2){$E$}
     \end{overpic}
  \end{minipage}
    \caption{Plot of the spectra of the Bloch-reduced edge Hamiltonian operators $\oldhat{H}_E(k)$ \eqref{eq:edge_H} as a function of $k \in [0,2\pi)$. The spectrum of $H_E$ \eqref{eq:edge_BC} is the union of these spectra. The black portion of the plot corresponds to the spectrum of the bulk Hamiltonian $H_B$ \eqref{eq:inf_H_0} after Bloch-reduction in the direction parallel to the edge. The blue portion corresponds to spectrum resulting from the Dirichlet boundary condition at the edge \eqref{eq:edge_BC}. The green star corresponds to the edge state in \cref{fig:haldane_remove_states1} (left). Note that we choose to consider an edge state whose energy is close to the bulk spectrum, which makes coupling of the edge state to bulk modes possible in \cref{fig:haldane_time_frames_remove_defect} and \cref{fig:haldane_time_frames_ext_pot}. For each value of $k$ we adaptively increased the truncation size until the error in \eqref{spec_err_controlll} was below the chosen tolerance $0.01$.
    }
\label{fig:haldane_band_structure}
\end{figure}

We start by computing the spectrum of the operators $\oldhat{H}_E(k)$ \eqref{eq:edge_H} for $k \in [0,2 \pi)$, showing the results in \cref{fig:haldane_band_structure}. We select the parameters $t=1$, $t'=0.1$, $\phi=\pi/2$ and $V=0.2$, for which the model is in its topological phase and hence edge states must occur. The spectrum consists of two parts. The first part is the spectrum of $H_{B}$ \eqref{eq:inf_H_0} after Bloch-reduction with respect to one of the quasi-momenta, and is marked in black on \cref{fig:haldane_band_structure}. The second part is the additional part of the spectrum arising from the Dirichlet boundary condition at the edge \eqref{eq:edge_BC}, and is marked in blue on \cref{fig:haldane_band_structure}. For any fixed $k$, this spectrum is discrete, with an associated eigenfunction that decays into the bulk. The associated eigenfunctions of these eigenvalues are known as edge states. When extended according to \eqref{eq:edge_modes} they become non-normalisable eigenfunctions of $H_E$ that extend parallel to the edge. We plot an extended edge state, corresponding to the green star marked in \cref{fig:haldane_band_structure}, in the left panel of \cref{fig:haldane_remove_states1}. Before extension to the whole lattice, the computed state has a residual of less than $4\times 10^{-16}$.

\begin{figure}
\centering
\begin{tabular}{cc}
Edge State (no defect) & Edge State (with defect)\\
\begin{overpic}[width=0.45\textwidth,clip,trim={0mm 15mm 70mm 15mm}]{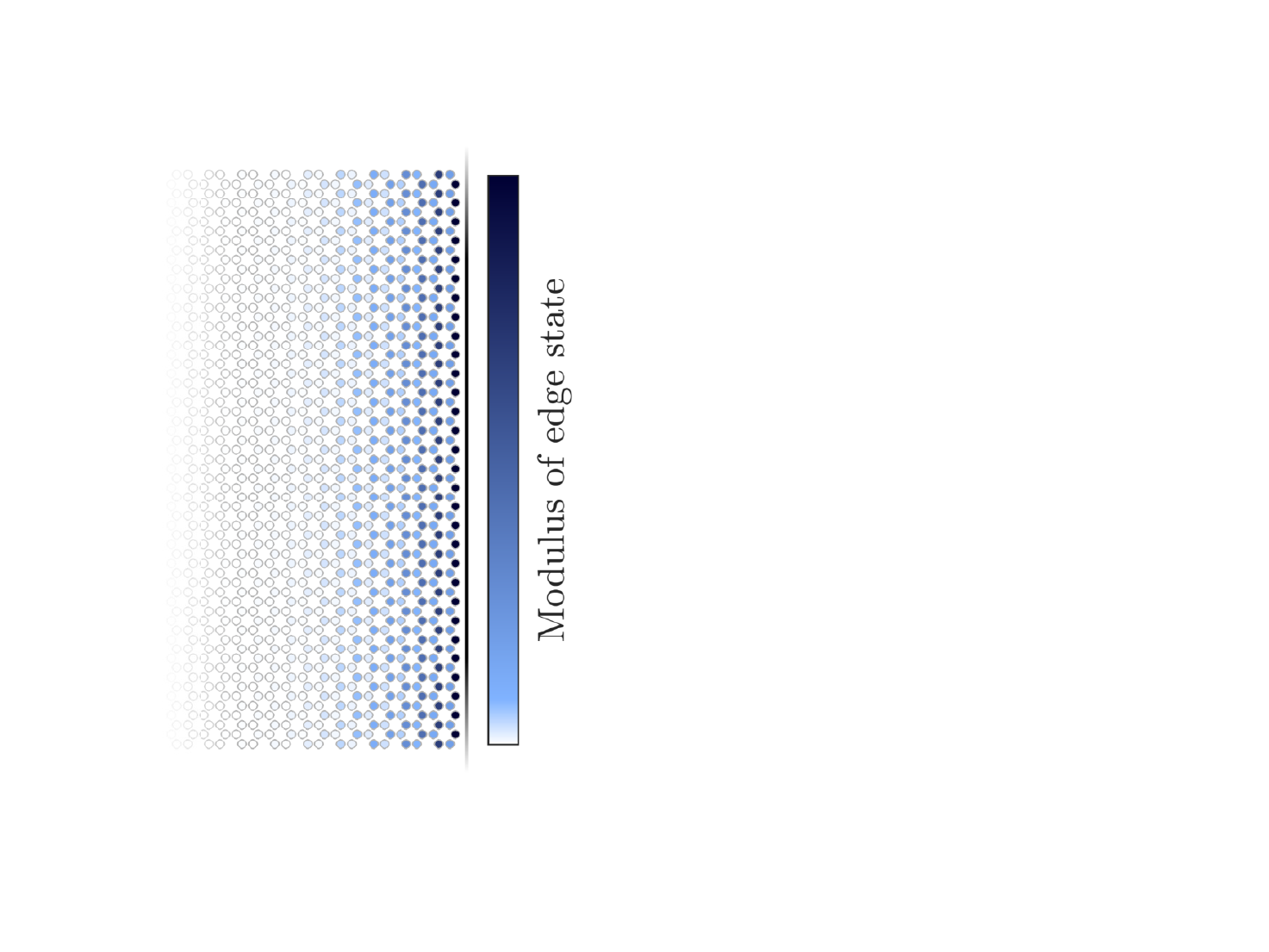}
\end{overpic}&
\begin{overpic}[width=0.45\textwidth,clip,trim={0mm 15mm 70mm 15mm}]{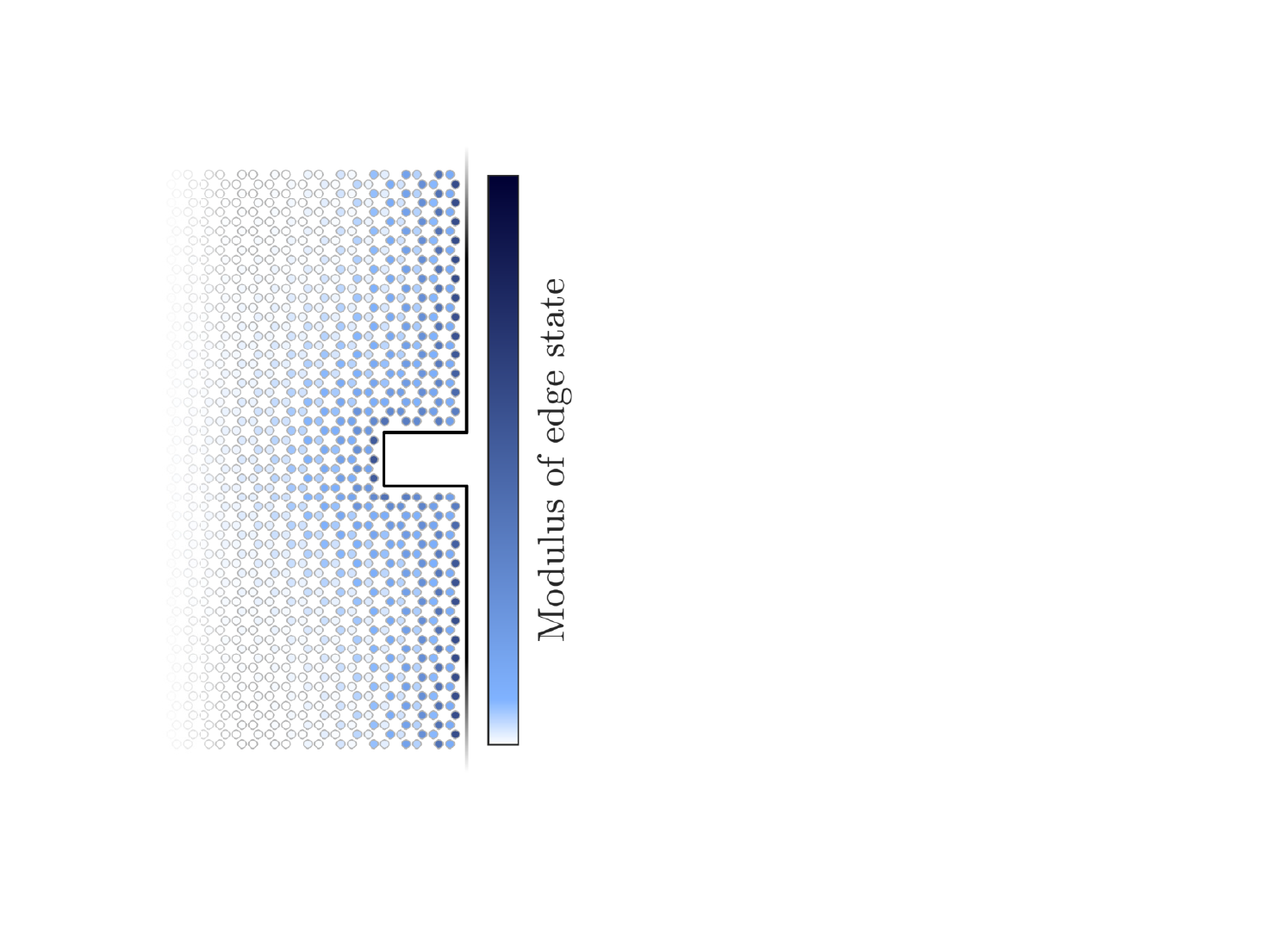}
\end{overpic}
\end{tabular}
\caption{Left: Absolute value of the approximate eigenstate (periodically continued to the whole lattice) of the edge Hamiltonian corresponding to the green star in \cref{fig:haldane_band_structure}. We have truncated the lattice to show only a finite portion. Right: An edge state at the energy of the green star when we have an additional defect of missing sites.}
\label{fig:haldane_remove_states1}
\end{figure}

\subsection{Approximate edge states at a non-periodic edge} \label{sec:res:non-periodic}

Next, we show results of computing approximate edge states non-periodic edges. Since the edge Hamiltonian cannot be Bloch reduced, we must compute edge states through the infinite-dimensional Hamiltonian directly. Since the spectrum associated to edge states is absolutely continuous \cite{Bols2021}, we restrict attention to computing \emph{approximate} edge states using the methods of \cref{Num_spectra}.

Our first experiment seeks to answer the question:  \textit{What happens to the edge state in \cref{fig:haldane_band_structure} if we remove a group of sites along the edge?} More specifically, what is the form of the approximate edge state of the edge Hamiltonian with the defect $H_{E,d}$, with the same energy as that exact edge state? This mode is computed in the right panel of \cref{fig:haldane_remove_states1} with residual bounded by $10^{-3}$ (this residual is larger than for the periodic case in \cref{sec:res:11} since we must compute an edge state that does not decay parallel to the edge through the infinite-dimensional Hamiltonian directly). The approximate edge state simply snakes around the defect. Away from the defect, the approximate edge state of $H_{E,d}$ is essentially identical to the exact edge state of $H_E$. This is quite a general phenomenon: in \cref{fig:haldane_P}, we compare the diagonal entries of the spectral projections onto an interval of edge states of the edge Hamiltonians with and without missing atoms. We find that the difference in the projections decays rapidly away from the defect.

\begin{figure}
\centerline{\includegraphics[width=0.49\textwidth,clip,trim={0mm 15mm 70mm 15mm}]{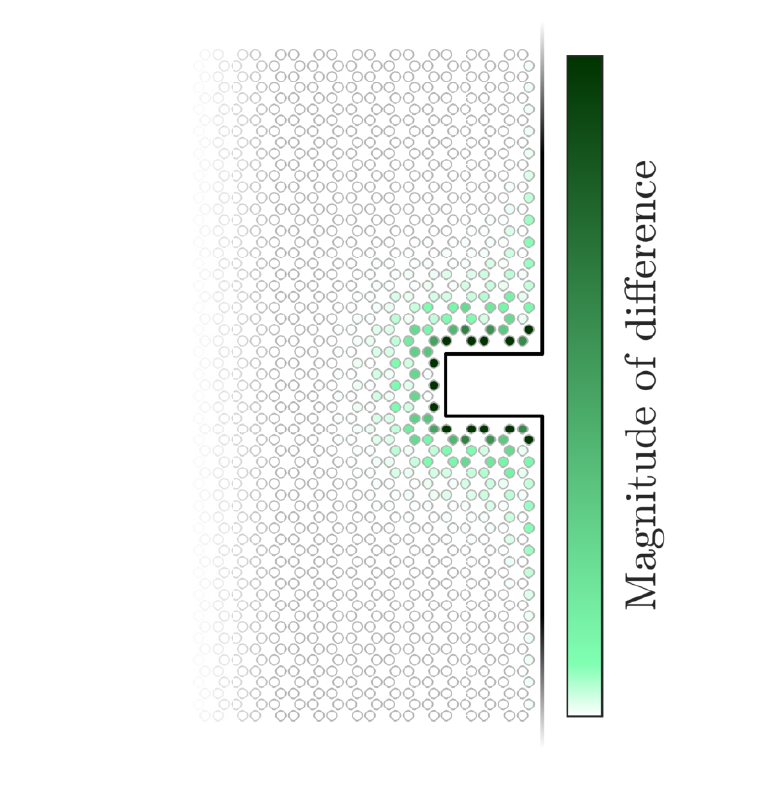}}
    \caption{Using the method of \cref{Num_spec_proj}, we can compute spectral projections onto an interval $\Delta$ of edge states of $H_E$ and $H_{E,d}$, where $H_{E,d}$ has sites removed from the edge, in order to visualise how the whole interval of edge states responds to the perturbation. The selected parameters are $t=1$, $t'=0.1$, $\phi=\pi/2$ and $V=0.2$. Here, we compute the spectral projections onto edge states associated with the interval $[-0.1,0.05]$ and plot the difference of the diagonal entries of these projections. We see that the defect only affects the projection locally, i.e., near the missing atoms.}
\label{fig:haldane_P}
\end{figure}

Our second experiment considers the same question, but in this case the perturbation is no longer a local defect of the edge, but a random onsite potential \eqref{eq:V_dis} added to every site with $w = 1$. This potential represents a non-compact perturbation of $H_{E}$. We plot the value of the specific realisation of the random potential used for our experiment in the left panel of \cref{fig:haldane_ext_pot}. In the right panel of \cref{fig:haldane_ext_pot} we plot the approximate edge state with a residual bounded by $10^{-3}$. We find that the approximate edge state persists again, despite the perturbation.

\begin{figure}
\centering
\begin{tabular}{cc}
Random Potential & Perturbed Edge State\\
\begin{overpic}[width=0.45\textwidth,clip,trim={0mm 15mm 70mm 15mm}]{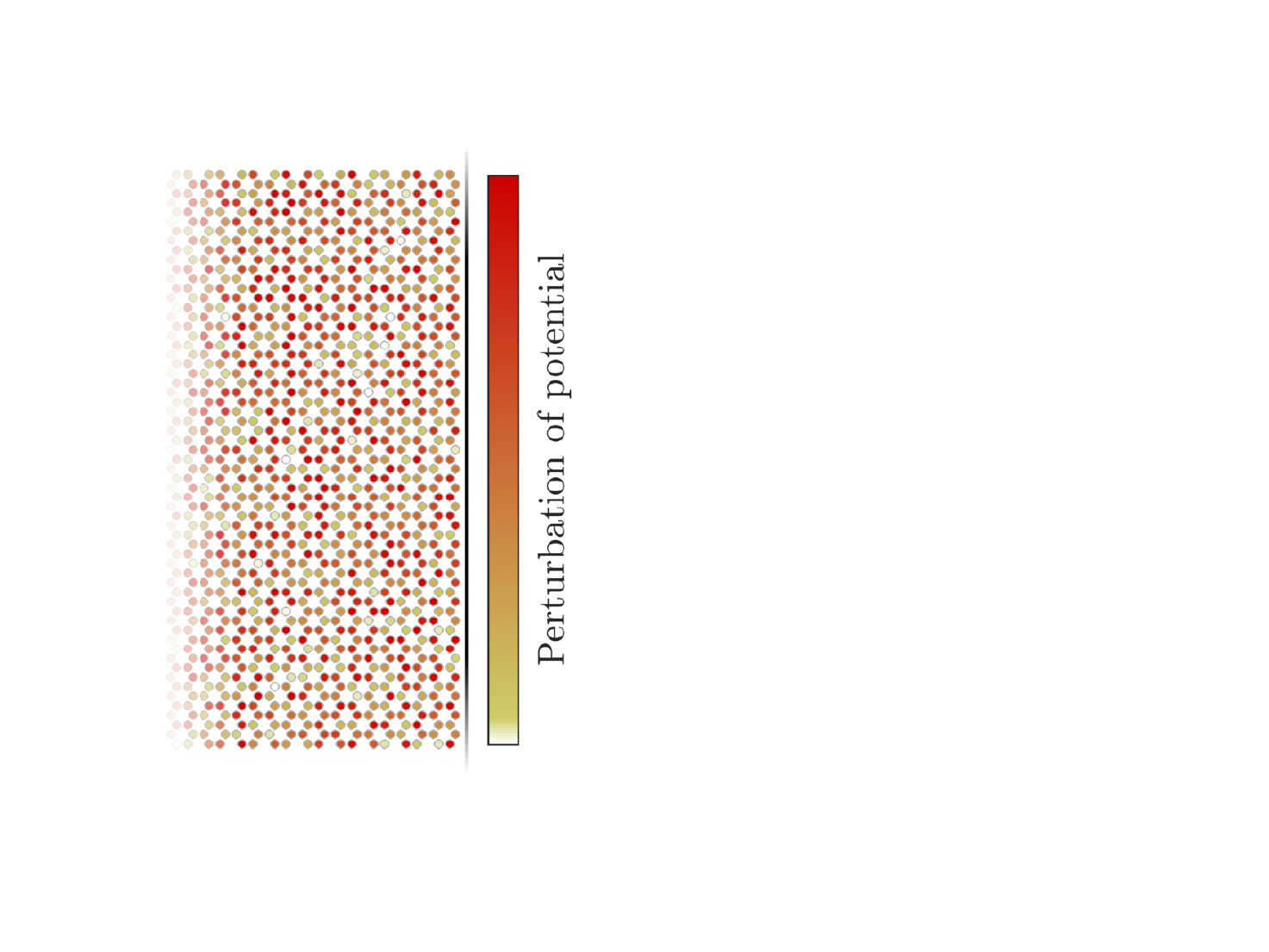}
\end{overpic}&
\begin{overpic}[width=0.45\textwidth,clip,trim={0mm 15mm 70mm 15mm}]{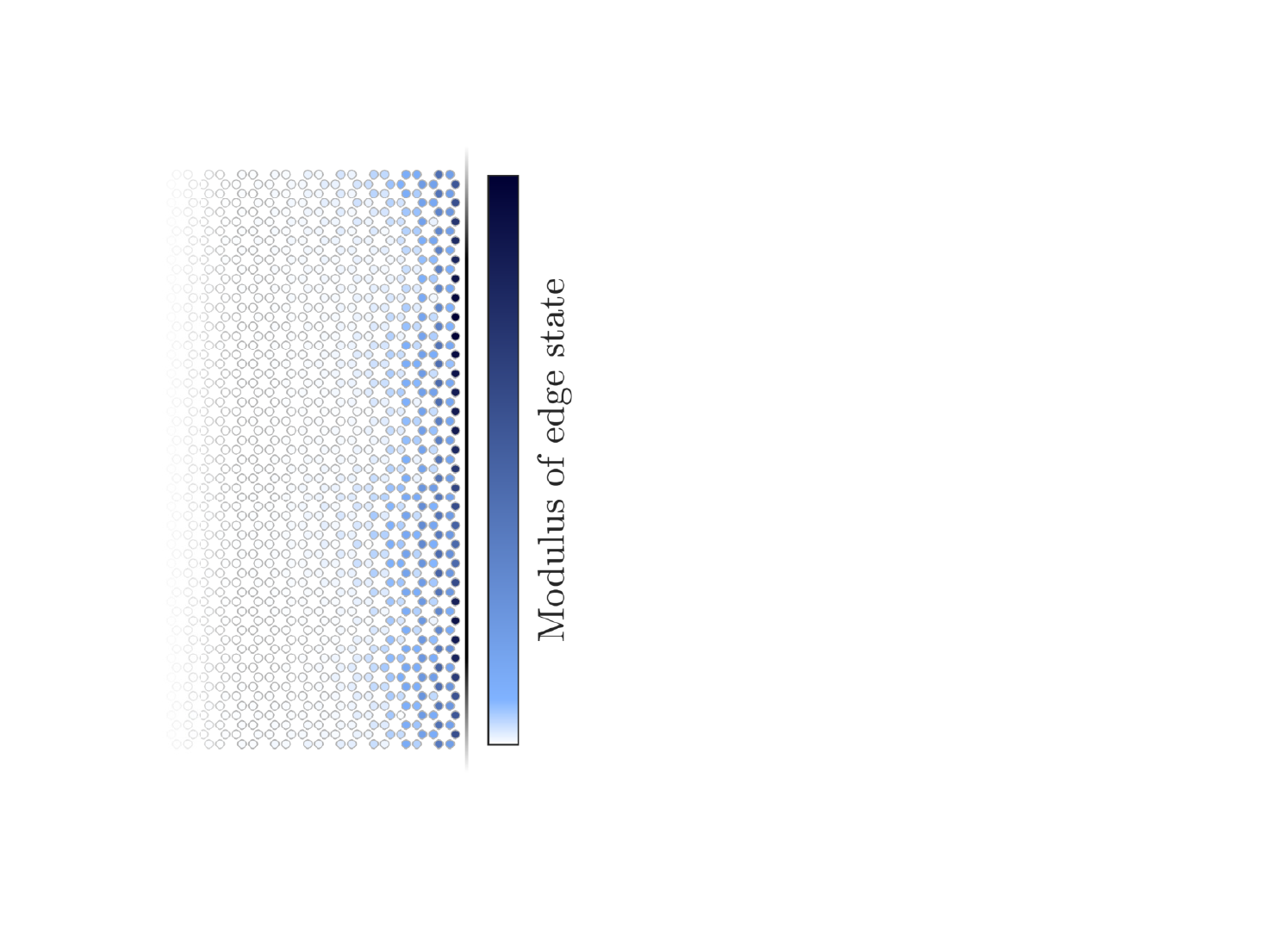}
\end{overpic}
\end{tabular}
    \caption{Left: Specific realisation of the random potential \eqref{eq:V_dis}, added at every site. Right: The approximate edge state of $H_{E,d}$ with the same energy as edge state shown in left panel of \cref{fig:haldane_remove_states1}.}
\label{fig:haldane_ext_pot}
\end{figure}
\subsection{Edge wave-packets and spectral measure} \label{sec:res:WPs}

We now compute edge wave-packets and investigate their spectral measures using the methods of \cref{Num_spec_meas}. We start by computing the spectral measure of the approximate edge state in the right panel of \cref{fig:haldane_remove_states1}. This is depicted in blue in \cref{fig:haldane_spectral_measure}, where we used a sixth-order kernel with smoothing parameter $\epsilon=0.01$. The truncation parameter is selected adaptively for each spectral parameter as outlined in \cref{Num_spec_meas}. We see that, as expected, the spectral measure is heavily weighted around the energy at which we computed the approximate edge state.  Next, we look at the spectral measure when we create a wave packet out of this approximate eigenstate. To create the wave-packet, we multiply the approximate eigenstate by a Gaussian centred away from the defect. This multiplication is equivalent to convolving the approximate edge state by a Gaussian in momentum space. Therefore, the spectral measure of the wave-packet will be spread out compared to the spectral measure of the approximate edge state. The orange part of the plot confirms this expected behaviour. We have repeated this process for the state shown in the right panel of \cref{fig:haldane_ext_pot}, but we omit these results since they are similar to those shown in \cref{fig:haldane_spectral_measure}. 

\begin{figure}
\centering
\begin{minipage}[b]{0.8\textwidth}
    \begin{overpic}[width=1\textwidth]{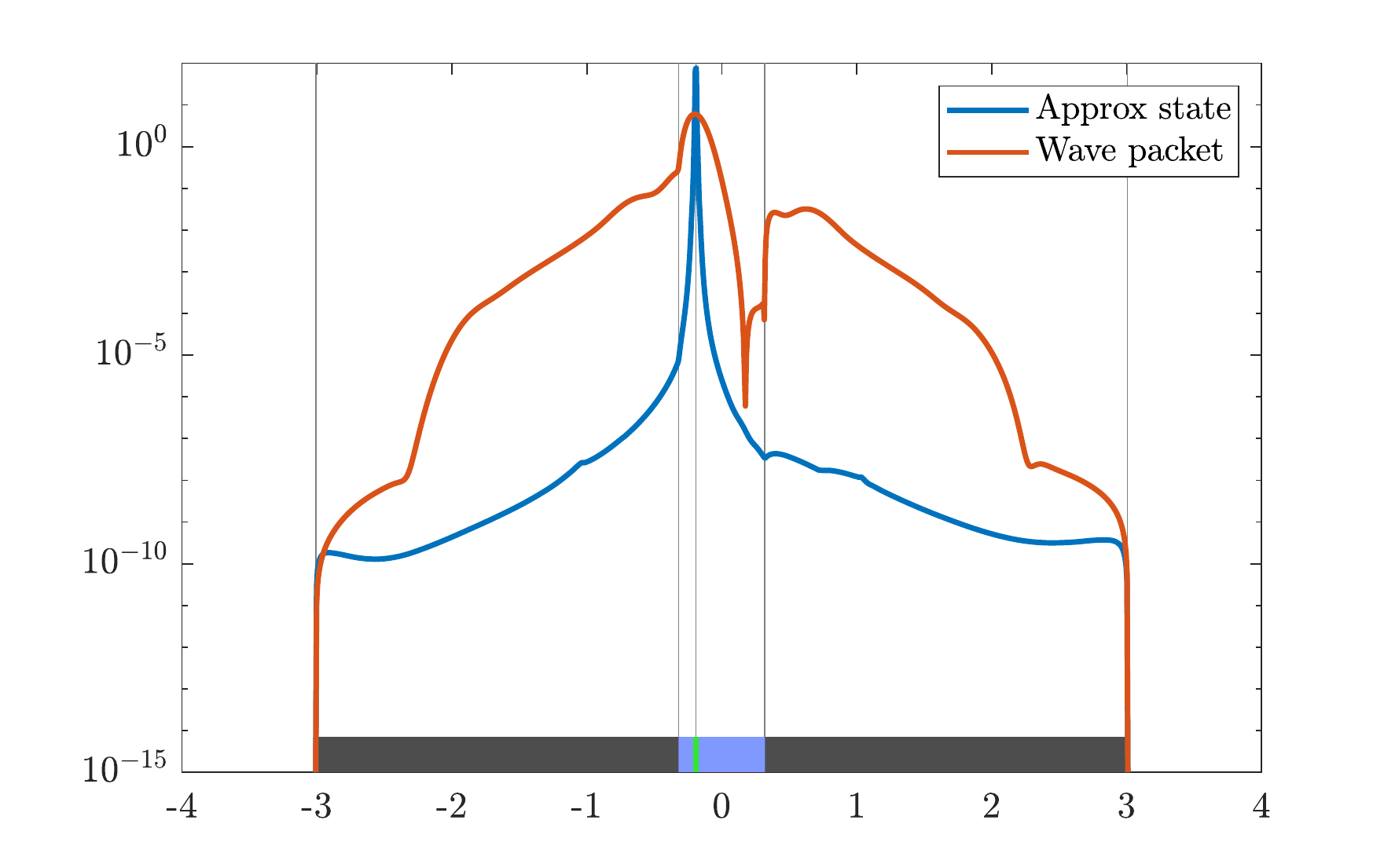}
		\put(51,-1){$E$}
     \end{overpic}
\end{minipage}
\caption{The (smoothed) spectral measure of an approximate eigenstate and a wave packet.  The spectrum of the Hamiltonian is plotted at the bottom with vertical grey lines added for clarity. The approximate state was computed at the green energy.  Notice the log scale on the vertical axis. The spectral measure of the edge wave-packet is much more spread out, as expected, with non-trivial support on the bulk spectrum.}
\label{fig:haldane_spectral_measure}
\end{figure}

We can probe how the spectral properties of the edge Hamiltonian change when the random potential \eqref{eq:V_dis} is added by computing the spectral measures of a delta function at a single site on the edge\footnote{The spectral measure for edge wave-packets with added random disorder behave similarly to \cref{fig:haldane_spectral_measure}.}, with and without the random potential. The results are shown in \cref{fig:haldane_spectral_measure2} for $t=1$, $t'=0.1$, $\phi=\pi/2$, $V=0.2$ and $w=1$. We observe that the potential causes the spectral measure to become significantly more singular in the bulk spectrum, while the edge spectrum changes shape but remains smooth. This is consistent with absolute continuity of the edge spectral measure \cite{Bols2021}, but we are not aware of work which would explain the behavior of the bulk spectral measure.

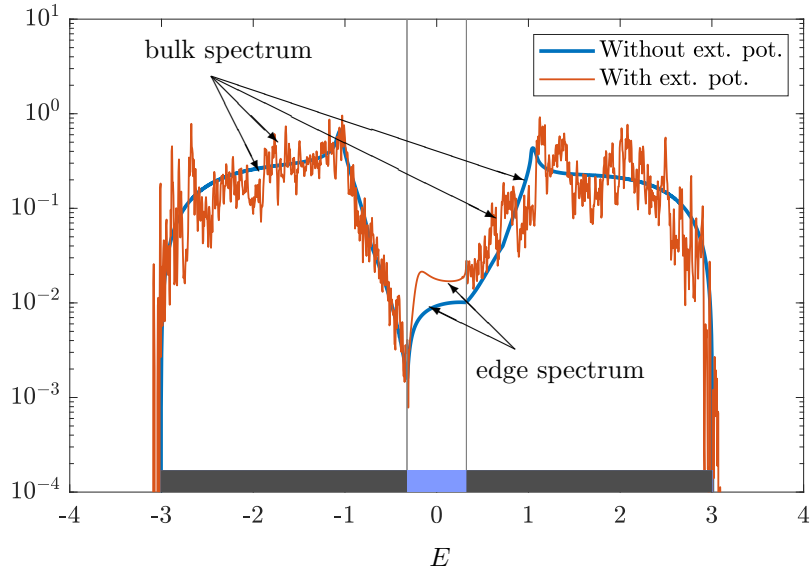
\begin{figure}
\centering
\begin{minipage}[b]{0.8\textwidth}
    \begin{overpic}[width=1\textwidth]{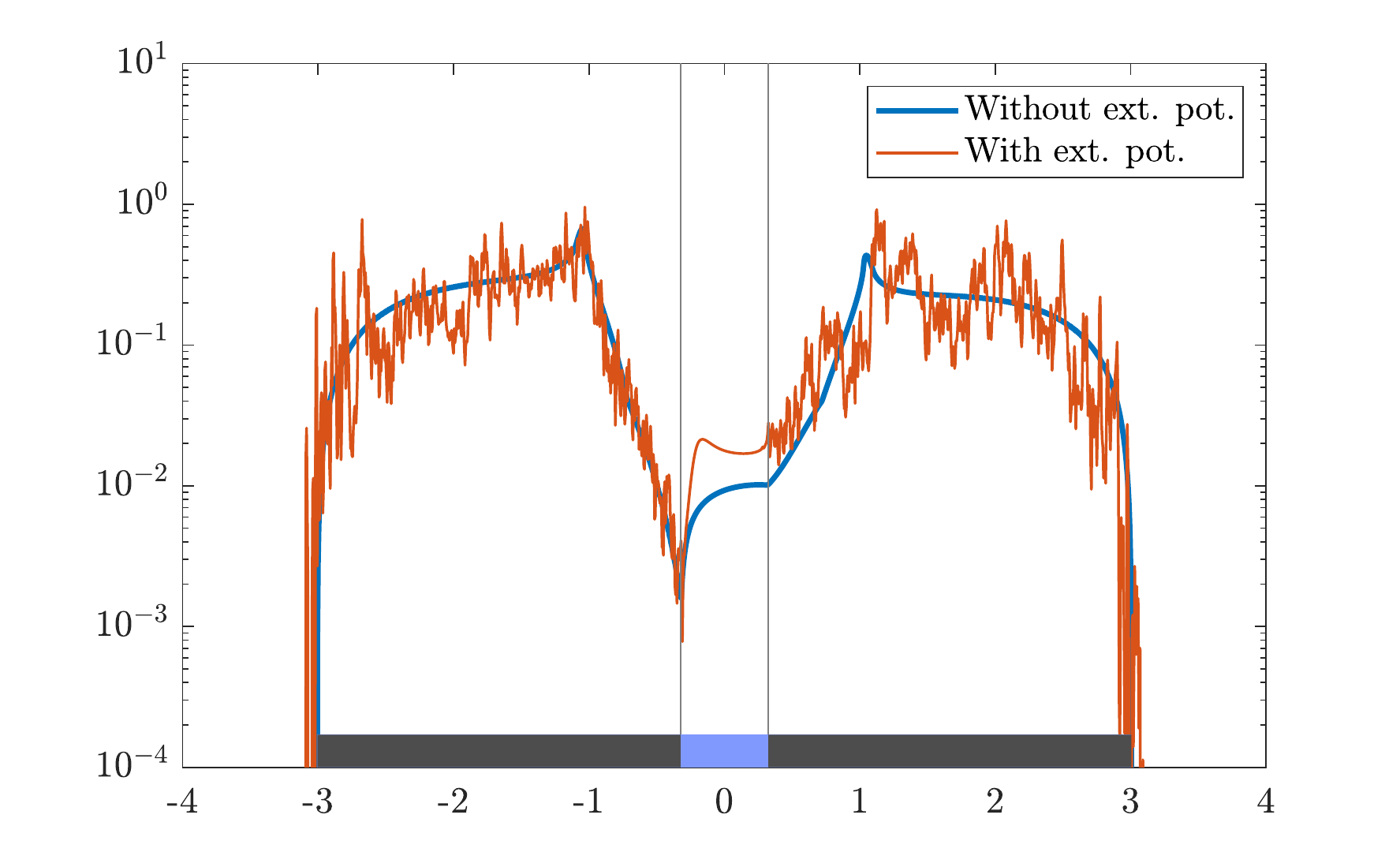}
		\put(51,-1){$E$}
		\put (20,53) { {bulk spectrum}}
	\put(28,51)  {\vector(1,-1){7}}
	\put(28,51)  {\vector(1,-2){5}}
	\put(28,51)  {\vector(2,-1){30}}
	\put(28,51)  {\vector(3,-1){33}}
	\put (55,19) { {edge spectrum}}
	\put(60,22)  {\vector(-1,1){7}}
	\put(60,22)  {\vector(-2,1){9}}
     \end{overpic}
\end{minipage}
\caption{The (smoothed) spectral measure of a site along the zig-zag edge. The boundaries of the edge spectrum of the Hamiltonian without the external potential are shown as vertical grey lines added for clarity. Notice that the addition of the random potential causes a more singular spectral measure on the bulk spectrum, but the spectral measure on the edge spectrum remains smooth.}
\label{fig:haldane_spectral_measure2}
\end{figure}

\subsection{Time propagation of edge wave-packets} \label{sec:res:time_prop}

\begin{figure}
\centering
\begin{minipage}[b]{1\textwidth}
    \begin{overpic}[width=1\textwidth,clip,trim={15mm 0mm 0mm 0mm}]{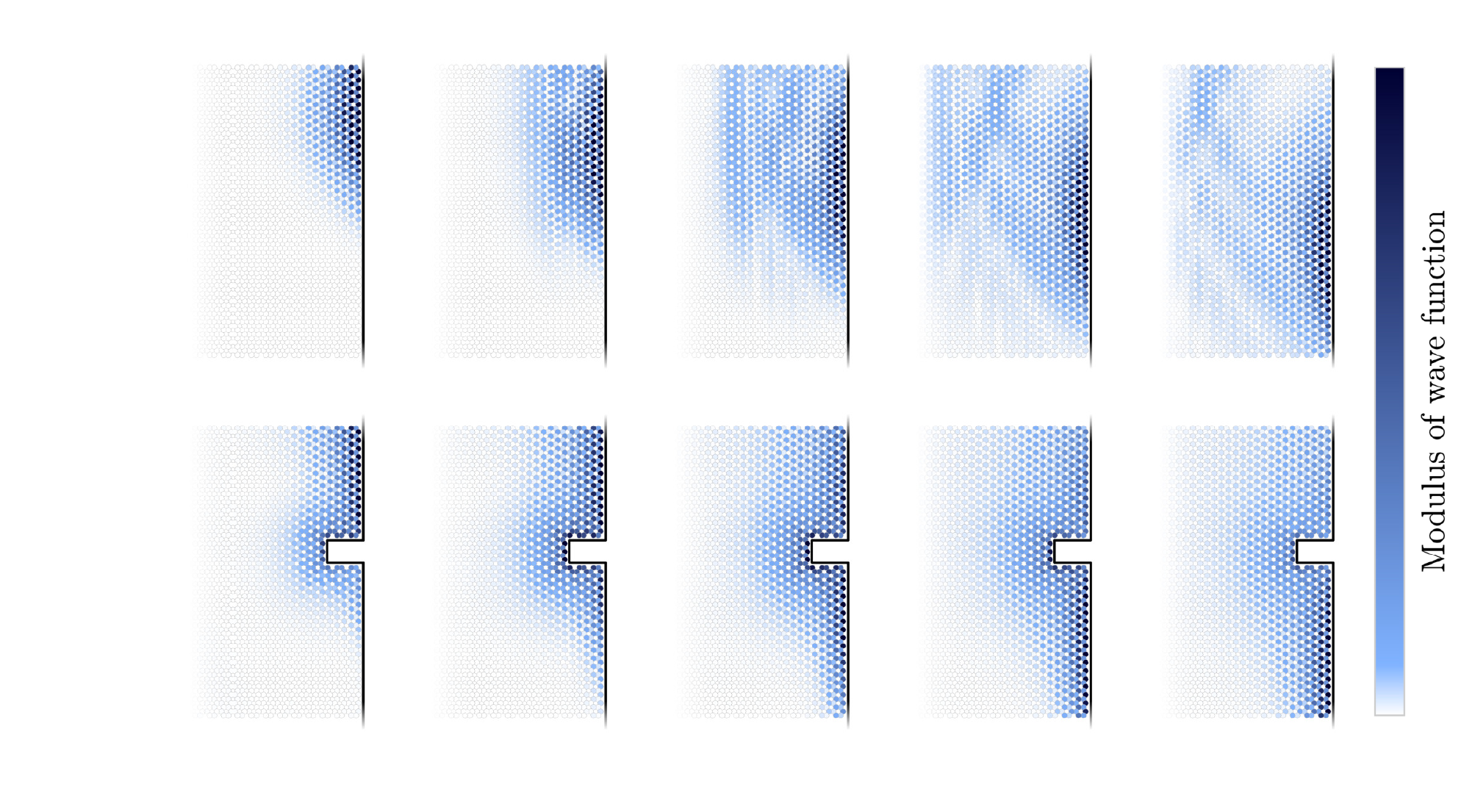}
		\put(10,55.5){$t=0$}
		\put(27,55.5){$t=10$}
		\put(45,55.5){$t=20$}
		\put(63,55.5){$t=30$}
		\put(81,55.5){$t=40$}
		\put(9,29){$t=160$}
		\put(26,29){$t=175$}
		\put(44,29){$t=190$}
		\put(62,29){$t=205$}
		\put(80,29){$t=220$}
     \end{overpic}
\end{minipage}
    \caption{Time propagation of a wave packet.  Top left is the initial state (a wave packet far away from a defect).  The images go forward in time as we move to the right. We observe the wave-packet clearly losing mass to bulk modes. Then we fast forward to the bottom left (and move our camera down) to see the remaining wave packet just starting to hit the defect. The wave packet then crawls around the defect as we go to the right. Note that (essentially) no further mass is lost to bulk modes as the wave-packet propagates around the defect.}
\label{fig:haldane_time_frames_remove_defect}
\end{figure}

\begin{figure}
\begin{minipage}[b]{1\textwidth}
    \begin{overpic}[width=1\textwidth,clip,trim={15mm 0mm 0mm 0mm}]{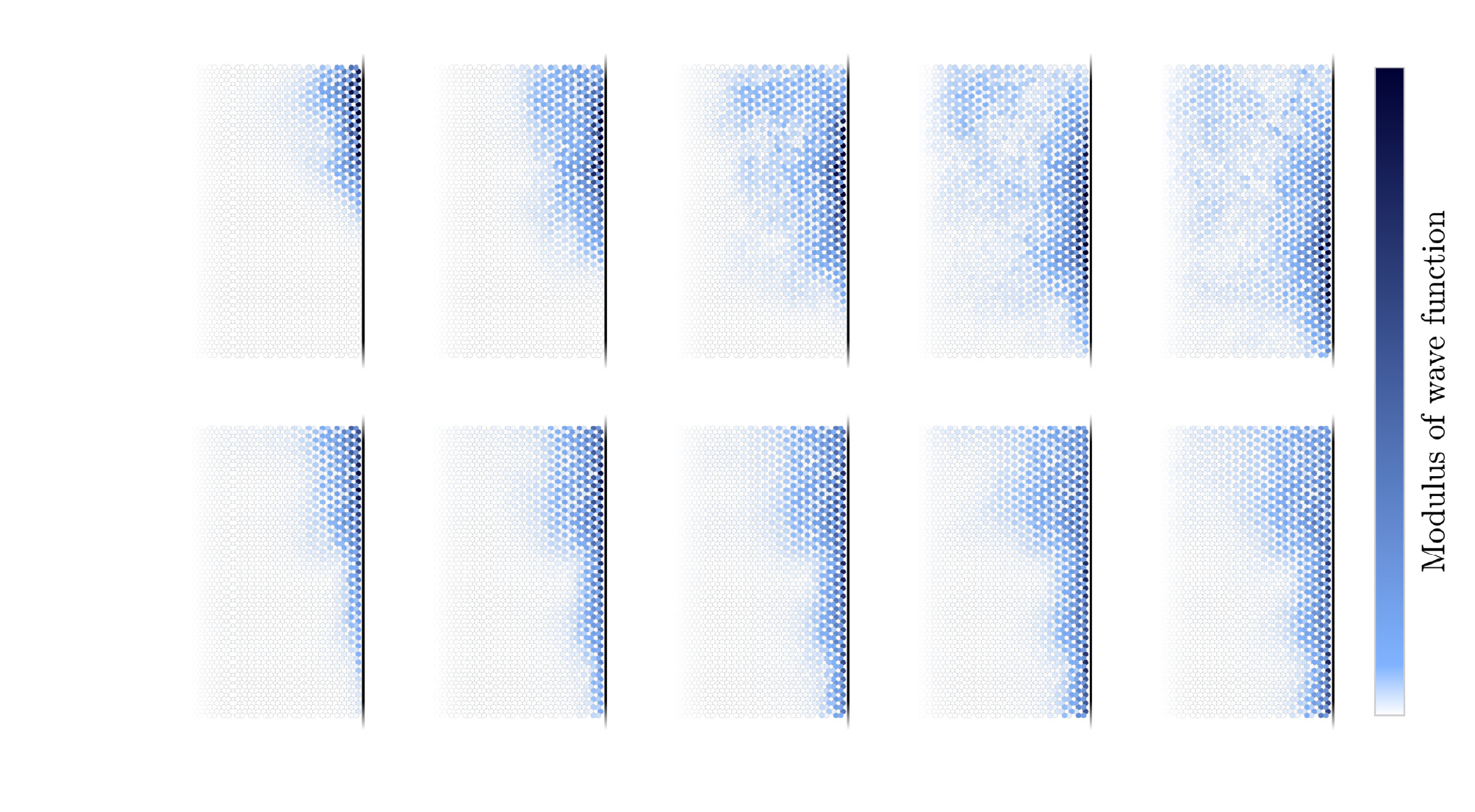}
		\put(10,55.5){$t=0$}
		\put(27,55.5){$t=10$}
		\put(45,55.5){$t=20$}
		\put(63,55.5){$t=30$}
		\put(81,55.5){$t=40$}
		\put(9,29){$t=160$}
		\put(26,29){$t=175$}
		\put(44,29){$t=190$}
		\put(62,29){$t=205$}
		\put(80,29){$t=220$}
     \end{overpic}
\end{minipage}
    \caption{The propagation of a wave packet in a system with a random perturbation to the external potential at every site.  Again, top left is the initial state and time goes forward as we move to the right. Then we fast forward (and move our camera down) to get to the bottom left picture. What is left of the wave packet then continues to propagate. Just as in \cref{fig:haldane_time_frames_remove_defect}, we observe loss of mass to bulk modes at the beginning of the simulation, and very little afterwards.}
\label{fig:haldane_time_frames_ext_pot}
\end{figure}

Finally, we look at the time propagation of wave packets along the edge for $t=1$, $t'=0.1$, $\phi=\pi/2$ and $V=0.2$. We begin with a wave packet produced as in the previous subsection -- by computing an approximate edge state and multiplying by a Gaussian far from (above) the defect. Due to the positive slope of the edge spectrum in \cref{fig:haldane_band_structure}, we expect that the edge states will move down the edge.  Additionally, since the edge states of the edge Hamiltonian with sites removed are similar to the unperturbed edge state, we expect them to behave similarly.  Lastly, by looking at the spectral measure of a wave packet in the previous subsection, we expect that some of the density of the wave packet will be lost to the bulk. However, most of it will continue along the edge. 

We compute the time evolution with a relative $\ell^2$ error bound of $10^{-5}$. \cref{fig:haldane_time_frames_remove_defect} shows the results. In the beginning, we observe part of the wave packet propagating into the bulk. This propagation is consistent with our computation of the spectral measure in \cref{fig:haldane_spectral_measure}, which showed that the wave packet has non-trivial spectral support among bulk modes. Interestingly, the remaining wave packet remains very coherent as it propagates around the defect, with little further coupling to the bulk. This is consistent with our computation of an approximate edge state of the model with defect, \cref{fig:haldane_remove_states1}, which showed the edge state snaking around the defect similarly to the edge state without a defect, and with our computation of the spectral measure in \cref{fig:haldane_spectral_measure}, which showed that the wave packet is primarily concentrated among such edge states.

We perform a similar test with a Hamiltonian that has disorder everywhere. Namely, we add a random potential described in \cref{sec:res:11} at each site with $w=1$. \cref{fig:haldane_time_frames_ext_pot} shows the results. Again, we observe part of the wave packet propagating to the bulk and then generally coherent propagation along the edge.

\appendix

\section{Pseudocode for the algorithms}
\label{append_psueod}

\begin{algorithm}[H]
\vspace{2mm}
\textbf{Input:} $\widetilde{H}$, $f$, $n$ and $G_n$ (with resolution $r_n\geq 1$).\\
\begin{algorithmic}[1]
\STATE For each $z\in G_n$, compute $\tilde F_n(z)=\ceil{2r_nF_n(z)}/(2r_n)$ and $v_n(z)$, the right-singular vector of $P_{f(n)}(\widetilde{H}-z)P_n$ corresponding to the smallest singular value.
\STATE For $z\in G_n$, if $\tilde F_n(z)\leq 1/2$, then set
\begin{align*}
I_z=\left\{w\in G_n:|w-z|\leq \tilde F_n(z)\right\},\quad M_z=\left\{w\in G_n:\tilde F_n(w)=\min_{x\in I_z}\tilde F_n(x)\right\}.
\end{align*}
Otherwise, set $M_z=\emptyset$.
\end{algorithmic} \vspace{2mm} \textbf{Output:} $\Gamma_n=\cup_{z\in G_n}M_z$ (approximation of spectrum), $E_n=\max_{z\in\Gamma_n}\tilde F_n(z)$ (error bound) and $V_n=\cup_{z\in\Gamma_n}\{v_n(z)\}$ (approximate states).
\caption{Computation of spectrum and the associated approximate states with error control via the method of \cite{Colbrook2019}. The computation of $\tilde F_n$ can be performed in parallel.
}\label{alg:spec_comp}
\end{algorithm}

\begin{algorithm}[H]
\textbf{Input:} $\widetilde{H}$, $\psi\in\mathcal{H}\cong l^2(\mathbb{N})$, $x_0\in\mathbb{R}$, $a_1,\dots,a_m\in\{z\in\mathbb{C}:{\rm Im}(z)>0\}$, and $\epsilon>0$. \\
\vspace{-4mm}
\begin{algorithmic}[1]
\STATE Solve the Vandermonde system \eqref{eqn:vandermonde_condition} for the residues $\alpha_1,\dots,\alpha_m\in\mathbb{C}$.
	\STATE Solve $(\widetilde{H}-(x_0-\epsilon a_j))u_{j}^\epsilon=\psi$ for $1\leq j\leq m$.
	\STATE Compute $\mu_{\psi}^\epsilon(x_0)=\frac{-1}{\pi}\mathrm{Im}\left(\sum_{j=1}^m\alpha_j\langle u_{j}^\epsilon,{\psi}\rangle\right)$.
\end{algorithmic} \textbf{Output:} $\mu_{\psi}^\epsilon(x_0)$.
\caption{A practical framework for evaluating an approximate spectral measure of an operator $\widetilde{H}$ at $x_0\in\mathbb{R}$ with respect to a vector $_{\psi}\in\mathcal{H}$ via the method of \cite{colbrook2019computing,colbrook2020}. For the examples of this paper, the resolvent is computed using the rectangular truncations in \eqref{rectangular_key}. However, the framework generalises to arbitrary self-adjoint operators given the ability to approximate solutions of the shifted linear systems and inner products.}\label{alg:spec_meas}
\end{algorithm}

\begin{algorithm}[H]
\textbf{Input:} $\widetilde{H}$, $\psi\in\mathcal{H}\cong l^2(\mathbb{N})$, $[a,b]\subset\mathbb{R}$, $\theta_1,\ldots,\theta_N\in[0,1]$, $w_1,\ldots,w_N\in\mathbb{R}_+$, $a_1,\dots,a_m\in\{z\in\mathbb{C}:{\rm Im}(z)>0\}$, and $\epsilon>0$. \\
\vspace{-4mm}
\begin{algorithmic}[1]
\STATE Solve the Vandermonde system \eqref{eqn:vandermonde_condition} for the residues $\alpha_1,\dots,\alpha_m\in\mathbb{C}$.
\STATE Set $\smash{z_\ell=a+\frac{b-a}{2}(1+\exp(\pi i\theta_\ell))}$ and $\smash{z_\ell'=-i\frac{b-a}{2}\exp(\pi i\theta_\ell)}$.
\STATE Solve $(\widetilde{H}-(z_\ell-\epsilon a_j))u_j^\epsilon=\psi$ and $(\widetilde{H}-(\bar z_\ell-\epsilon \bar a_j))v_j^\epsilon=\psi$ for $1\leq j\leq m$, and each $1\leq\ell\leq N$.
\STATE Solve $(\widetilde{H}-a\mp i\epsilon)u_{a,\pm}^\epsilon=\psi$ and $(\widetilde{H}-b\mp i\epsilon)u_{b,\pm}^\epsilon=\psi$.
\STATE Compute $\mathcal{E}^\epsilon_{(a,b)}\psi=\frac{-1}{2\pi i}\sum_{\ell=1}^N\sum_{j=1}^{m}w_\ell\left[\alpha_j z_\ell' u_j^\epsilon - \bar\alpha_j \bar z_\ell' v_j^\epsilon\right]-\frac{\epsilon}{2i}\left(c_l\left[u_{a,+}^\epsilon-u_{a,-}^\epsilon\right]+c_r\left[u_{b,+}^\epsilon-u_{b,-}^\epsilon\right]\right)$.
\end{algorithmic} \textbf{Output:} $\mathcal{E}^\epsilon_{(a,b)}\psi$.
\caption{An efficient method (see~\cref{Num_spec_proj}) to compute spectral projections associated with the projection-valued measure $\mathcal{E}$ of an operator $\widetilde{H}$.}\label{alg:spec_proj}
\end{algorithm}

\begin{table*}[h!]
\renewcommand{\arraystretch}{1.6}
\centering
\begin{tabular}{l|c|c}
$m$ & $\pi K(x)\prod_{j=1}^m(x-a_j)(x-\overline{a_j})$ & $\{\alpha_1,\ldots,\alpha_{\ceil{m/2}}\}$\\
\hline
$2$ & $\frac{20}{9}$ & $\left\{\frac{1+3i}{2}\right\}$\\
$3$ &$-\frac{5}{4}x^2+\frac{65}{16}$ & $\left\{-2+i,5\right\}$\\
$4$ & $-\frac{3536}{625}x^2+\frac{21216}{3125}$ & $\left\{\frac{-39-65i}{24},\frac{17+85i}{8}\right\}$\\
$5$ & $\frac{130}{81}x^4 - \frac{12350}{729}x^2 + \frac{70720}{6561}$ & $\left\{\frac{15-10i}{4},\frac{-39+13i}{2},\frac{65}{2}\right\}$\\
$6$ & $\frac{1287600}{117649}x^4 - \frac{34336000}{823543}x^2 + \frac{667835200}{40353607}$ & $\left\{\frac{725+1015i}{192},\frac{-2775-6475i}{192},\frac{1073+7511i}{96}\right\}$\\
\end{tabular}
\caption{The numerators and residues of the first six rational kernels with equispaced poles. We give the first ${\ceil{m/2}}$ residues because the others follow by the symmetry $\alpha_{m+1-j}=\overline{\alpha_j}$.\label{tab_kernel}}\renewcommand{\arraystretch}{1}
\end{table*}

\section{Convergence for smoothed projection-valued measures}\label{sec:PVM_converge}

In this appendix, we prove~\cref{thm:Stones_formula_generalized}. We begin by establishing convergence and calculating the endpoint contributions $c_l$ and $c_r$ for rational kernels with conjugate pole pairs.

\begin{proposition}\label{thm:GSF1}
Given a projection-valued measure $\mathcal{E}$ (see~\eqref{eqn:PVMdiag}) and $m$th order kernel $K$ with conjugate pole pairs (see~\eqref{eqn:high_order_kernel}), then for any $[a,b]\subset\mathbb{R}$, we have that
$$
\lim_{\epsilon\rightarrow 0^+}\inty{a}{b}{[K_{\epsilon}*\mathcal{E}](x)}{x} = \mathcal{E}((a,b))+c_l\mathcal{E}(\{a\})+c_r\mathcal{E}(\{b\}),
$$
where $c_l=\smash{\pi^{-1}\sum_{j=1}^m\beta_j(\pi-\arg(a_j))+i\gamma_j\log|a_j|}$ and $c_r=\smash{\pi^{-1}\sum_{j=1}^m\beta_j\arg(a_j)-i\gamma_j\log|a_j|}$.
\end{proposition}

\begin{proof}
First, integrate both sides of~\eqref{gen_stone_comp2} over the interval, substitute the resolvent identity in~\eqref{eqn:PVMres} on the right-hand side, and apply Fubini's theorem to obtain
$$
\inty{a}{b}{[K_{\epsilon}*\mathcal{E}](x)}{x} = \frac{-1}{2\pi i}\inty{\mathrm{Sp}(H)}{}{\inty{a}{b}{\sum_{j=1}^{m}\left[\frac{\alpha_j}{\lambda-(x-\epsilon a_j)}-\frac{\bar\alpha_j}{\lambda-(x-\epsilon \bar a_j)}\right]}{x}}{\mathcal{E}(\lambda)}.
$$
To establish the theorem, we take the limit $\epsilon\rightarrow 0$ and apply the dominated convergence theorem to interchange the limit and the outer integral. This is permissible due to the decay condition in part (iii) of \cref{def:mth_order_kernel}. We claim that, as $\epsilon\rightarrow 0$, the inner integral converges to $-2\pi i$ when $\lambda\in(a,b)$, $0$ when $\lambda\not\in[a,b]$, and $(-2\pi i)c_l$ or $-(2\pi i)c_r$ when $\lambda=a$ or $\lambda=b$, respectively.

We compute the inner integral directly by integrating the sum term by term, so that
$$
\inty{a}{b}{\sum_{j=1}^m\left[\frac{\alpha_j}{\lambda-(x-\epsilon a_j)}-\frac{\bar\alpha_j}{\lambda-(x-\epsilon \bar a_j)}\right]}{x} 
=\sum_{j=1}^m\left[\bar\alpha_j\log\left(\lambda-(x-\epsilon\bar a_j)\right) -\alpha_j\log\left(\lambda-(x-\epsilon a_j)\right)\right]\big|^b_a.
$$
Using the identity $\log(z)=\log|z|+i\arg(z)$ to simplify, we find that the right-hand side is equal to
\begin{equation}\label{eqn:inner_int}
2\sum_{j=1}^m{\rm Im}(\alpha_j)\left[\log|\lambda-b+\epsilon a_j| -\log|\lambda-a+\epsilon a_j|\right]-i{\rm Re}(\alpha_j)\left[\arg(\lambda-b+\epsilon a_j)-\arg(\lambda-a+\epsilon a_j)\right].
\end{equation}
To calculate the limit, note that the first row of~\eqref{eqn:vandermonde_condition} states that $\sum_{j=1}^m \alpha_j=1$. In particular, $\smash{\sum_{j=1}^m {\rm Re}(\alpha_j)=1}$ and $\smash{\sum_{j=1}^m{\rm Im}(\alpha_j)=0}$. Then, the right-hand terms involving $\arg$ evaluate to
\begin{equation}\label{eqn:arg_terms}
\lim_{\epsilon\rightarrow 0}\sum_{j=1}^m{\rm Re}(\alpha_j)\left[\arg(\lambda-b+\epsilon a_j)-\arg(\lambda-a+\epsilon a_j)\right]
=\begin{cases}
\pi, \quad a<\lambda<b, \\
\sum_{j=1}^m{\rm Re}(\alpha_j)(\pi-\arg(a_j)), \quad \lambda=a, \\
\sum_{j=1}^m{\rm Re}(\alpha_j)\arg(a_j), \quad \lambda=b, \\
0, \quad\text{otherwise}.
\end{cases}
\end{equation}
On the other hand, the left-hand terms involving logarithms vanish when $\lambda\neq a$ and $\lambda\neq b$, that is,
\begin{equation}\label{eqn:log_terms1}
\lim_{\epsilon\rightarrow 0}\sum_{j=1}^m{\rm Im}(\alpha_j)\left[\log|\lambda-b+\epsilon a_j| -\log|\lambda-a+\epsilon a_j|\right]=\left[\log|\lambda-b|-\log|\lambda-a|\right]\sum_{j=1}^m{\rm Im}(\alpha_j)=0.
\end{equation}
Finally, when $\lambda=b$ we expand $\log|\epsilon a_j|=\log|\epsilon|+\log|a_j|$ and perform a similar calculation to obtain
\begin{equation}\label{eqn:log_terms2}
\lim_{\epsilon\rightarrow 0}\sum_{j=1}^m{\rm Im}(\alpha_j)\left[\log|\epsilon a_j|-\log|b-a+\epsilon a_j|\right]=\sum_{j=1}^m{\rm Im}(\alpha_j)\log|a_j|.
\end{equation}
We omit the analogous calculation for $\lambda=a$, which only differs by a minus sign. Collecting the results in~\eqref{eqn:inner_int}, \eqref{eqn:arg_terms}, \eqref{eqn:log_terms1} and \eqref{eqn:log_terms2} establishes the claim and concludes the proof the proposition.
\end{proof}

In practice, we usually employ symmetric rational kernels whose poles have reflection symmetry over the imaginary axis. In this case, the residues are symmetric over the real axis, and the constants $c_l$ and $c_r$ simplify considerably. These statements are made precise in the next two lemmas.

\begin{lemma}\label{lem:GSF2}
If the poles satisfy $a_{m+1-j}=-\bar a_j$, then the residues satisfy $\alpha_{m+1-j}=\bar\alpha_j$.
\end{lemma}
\begin{proof}
We proceed by calculating the residues directly from the Vandermonde system in~\eqref{eqn:vandermonde_condition}. By Cramer's rule, $\alpha_j={\rm det}(V_j)/{\rm det}(V)$, where $V$ is the transposed Vandermonde matrix and $V_j$ is identical except that the $j$th column is replaced by the unit vector on the right-hand side of~\eqref{eqn:vandermonde_condition}. We claim that ${\rm det}(V)$ is real and ${\rm det}(V_{m+1-j})=\overline{{\rm det}(V_j)}$, which together imply that $\alpha_{m+1-j}=\bar\alpha_j$.

The determinant of $V$ is given by the well-known formula $\smash{{\rm det}(V)=\prod_{1\leq i<j\leq m} (a_j-a_i)}$. Pairing conjugate terms and noting that $(a_{m+1-i}-a_{m+1-j})=(\bar a_j-\bar a_i)$ by the reflection symmetry hypothesis, we find that the determinant is real because
$$
{\rm det}(V)=\prod_{1\leq i<j\leq m} (a_j-a_i) = \prod_{1\leq i<j\leq\lceil m/2\rceil}(a_j-a_i)(a_{m+1-i}-a_{m+1-j})
=\prod_{1\leq i<j\leq\lceil m/2\rceil}|a_j-a_i|^2.
$$
To calculate the determinant of $V_j$, note that a Laplace expansion down the $j$th column yields
\begin{equation}\label{eqn:det_vj_1}
{\rm det}(V_j)=(-1)^{j-1}\begin{vmatrix}
a_1 & \dots & a_{j-1} & a_{j+1} & \dots & a_{m} \\
\vdots & & \vdots & \vdots &  & \vdots \\
a_1^{m-1} & \dots & a_{j-1}^{m-1} & a_{j+1}^{m-1} & \dots & a_{m}^{m-1}
\end{vmatrix}
=(-1)^{j-1}\prod\limits_{\substack{1\leq i\leq m,\\ i\neq j}}a_i\prod\limits_{\substack{1\leq i<k\leq m,\\ i,k\neq j}}(a_k-a_i).
\end{equation}
The second equality follows by factoring the poles $a_1,\ldots,a_{j-1},a_{j+1},\ldots a_m$ out of their respective columns and applying the formula for the determinant of the resulting $(m-1)\times(m-1)$ Vandermonde system (note that the indices are $1,\ldots,j-1,j+1,\ldots,m$). Since the poles are distinct, we may write
\begin{equation}\label{eqn:det_vj_2}
\prod\limits_{\substack{1\leq i<k\leq m,\\ i,k\neq j}}(a_k-a_i)
={\rm det}(V)\left(\prod_{1\leq i<j}(a_j-a_i)\prod_{j<i\leq m}(a_i-a_j)\right)^{-1}.
\end{equation}
Applying the reflection symmetry hypothesis and re-indexing with $i'=m+1-i$, we calculate that
\begin{equation}\label{eqn:det_vj_3}
\prod_{1\leq i<j}(a_j-a_i)\prod_{j<i\leq m}(a_i-a_j)=\prod_{1\leq i'<m+1-j}(\bar a_{m+1-j}-\bar a_{i'})\prod_{m+1-j<i'\leq m}(\bar a_{i'}-\bar a_{m+1-j}).
\end{equation}
Since ${\rm det}(V)$ is real-valued, it follows that $\smash{\prod_{\substack{1\leq i<k\leq m,\\ i,k\neq j}}(a_k-a_i)= \prod_{\substack{1\leq i<k\leq m,\\ i,k\neq m+1-j}}\overline{(a_k-a_i)}}$. Similarly,
\begin{equation}\label{eqn:det_vj_4}
\prod\limits_{\substack{1\leq i\leq m,\\ i\neq j}}a_i = (a_j)^{-1}\prod\limits_{1\leq i\leq m}a_i = (-\bar a_{m+1-j})^{-1}\prod\limits_{1\leq i\leq m}(-\bar a_{m+1-i})^{-1}=(-1)^{m+1}\prod\limits_{\substack{1\leq i\leq m,\\ i\neq m+1-j}}\bar a_i.
\end{equation}
Compiling the calculations in~\eqref{eqn:det_vj_1}, \eqref{eqn:det_vj_2}, \eqref{eqn:det_vj_3}, and \eqref{eqn:det_vj_4} establishes the claim, as we conclude that
$$
{\rm det}(V_j)=(-1)^{j-1}\prod\limits_{\substack{1\leq i\leq m,\\ i\neq j}}\! \! \! a_i\! \! \! \prod\limits_{\substack{1\leq i<k\leq m,\\ i,k\neq j}}(a_k-a_i)=(-1)^{m-j}\prod\limits_{\substack{1\leq i\leq m,\\ i\neq m+1-j}}\bar \! \! \! a_i\! \! \! \prod\limits_{\substack{1\leq i<k\leq m,\\ i,k\neq m+1-j}}\overline{(a_k-a_i)}=\overline{{\rm det}(V_{m+1-j})}.
$$
Therefore, Cramer's rule implies that the residues satisfy $\alpha_{m+1-j}=\bar\alpha_j$, for each $j=1,\ldots,m$.
\end{proof}

With the conjugate symmetry of the residues in hand, we can now show that the constants in~\cref{thm:GSF1} simplify significantly for symmetric kernels. Recall that $\smash{\sum_{j=1}^m\beta_j=1}$ and $\smash{\sum_{j=1}^m\gamma_j=0}$.

\begin{lemma}\label{lem:GSF3}
If the poles satisfy $a_{m+1-j}=-\bar a_j$, then $c_l=c_r=1/2$ in~\cref{thm:GSF1}.
\end{lemma}
\begin{proof}
To begin, note that symmetry of the residues over the real axis in~\cref{lem:GSF2} implies that $\gamma_{m+1-j}=-\gamma_j$, while the symmetry of the poles over the imaginary axis implies that $\log|a_j|=\log|a_{m+1-j}|$. Therefore, the logarithmic terms in $c_l$ and $c_r$ vanish because\footnote{When $m$ is odd, the relation $\gamma_{m+1-j}=-\gamma_j$ holds for $j=\lceil m/2\rceil$, so that $\gamma_{\lceil m/2\rceil}=0$.}
$$
\sum_{j=1}^m\gamma_j\log|a_j|=\sum_{j=1}^{\lfloor m/2\rfloor}(\gamma_j\log|a_j|+\gamma_{m+1-j}\log|a_{m+1-j}|)=0.
$$
Furthermore, the pole symmetries $a_{m+1-j}=-\bar a_j$ imply that $\arg(a_{m+1-j})=\pi-\arg(a_j)$, while the residue symmetries also imply that $\beta_{m+1-j}=\beta_j$. Therefore, we find that
\begin{equation}\label{eqn:endpts1}
c_l=\pi^{-1}\sum_{j=1}^m\beta_j\left(\pi-\arg(a_j)\right)=\pi^{-1}\sum_{j=1}^m\beta_j\arg(a_j)=c_r.
\end{equation}
Now, observe that $\beta_{j}\arg(a_j)+\beta_{m+1-j}\arg(a_{m+1-j})=\pi\beta_j$. For even $m$, we calculate that
\begin{equation}\label{eqn:endpts2}
\sum_{j=1}^m\beta_j\arg(a_j)=\sum_{j=1}^{m/2}\left(\beta_j\arg(a_j)+\beta_{m+1-j}\arg(a_{m+1-j})\right)=\pi\sum_{j=1}^{m/2}\beta_j=\frac{\pi}{2},
\end{equation}
The last equality follows from the fact that $\smash{2\sum_{j=1}^{m/2}\beta_j=\sum_{j=1}^m\beta_j=1}$ when $m$ is even. Analogously for odd $m$, we obtain
\begin{equation}\label{eqn:endpts3}
\sum_{j=1}^m\beta_j\arg(a_j)=\frac{\pi}{2}\beta_{\lceil m/2\rceil}+\pi\sum_{j=1}^{\lfloor m/2\rfloor}\beta_j=\frac{\pi}{2}.
\end{equation}
Here, we have used that $\smash{\beta_{\lceil m/2\rceil}+2\sum_{j=1}^{\lfloor m/2\rfloor}\beta_j=\sum_{j=1}^m\beta_j=1}$ when $m$ is odd. Plugging~\eqref{eqn:endpts2} and \eqref{eqn:endpts3} into~\eqref{eqn:endpts1} demonstrates that $c_l=c_r=1/2$, which concludes the proof.
\end{proof}

\section{Haldane model details} \label{sec:Haldane_details}

In this section, we fill in some details omitted in the discussion in \cref{sec:Haldane}. The explicit formula for the bulk band functions of $H_B(\vec{k})$ is
\begin{equation} \label{eq:full_band_functions}
    E_\pm(\vec{k}) = f(\vec{k}) \pm \sqrt{ g_1(\vec{k}) + g_2(\vec{k}) },
\end{equation}
where
\begin{equation}
\begin{split}
    &f(\vec{k}) := 2 t' \cos(\phi) \left[ \cos(k_1) + \cos(k_2) + \cos(k_1 - k_2) \right] \\
    &g_1(\vec{k}) := |t|^2 \left( |1 + \cos(k_1) + \cos(k_2)|^2 + |\sin(k_1) + \sin(k_2)|^2 \right) \\
    &g_2(\vec{k}) := | V + 2 t' \sin(\phi)( \sin(k_1) - \sin(k_2) - \sin(k_1 - k_2) ) |^2.
\end{split}
\end{equation}
The explicit action of the Bloch-reduced edge Hamiltonian $H_E(k)$ in $\ell^2(\mathbb{N};\mathbb{C}^2)$ is
\begin{equation} \label{eq:edge_H}
\begin{split}
    &\left(\oldhat{H}_{E}(k) \tilde{\psi}(k)\right)_{m} := \\
    &t \begin{pmatrix} (1 + e^{- i k_2}) \tilde{\psi}_m^B(k) + \tilde{\psi}_{m-1}^B(k) \\ (1 + e^{i k_2}) \tilde{\psi}_m^A(k) + \tilde{\psi}_{m+1}^A(k) \end{pmatrix} + V \begin{pmatrix} \tilde{\psi}_m^A(k) \\ - \tilde{\psi}_m^B(k) \end{pmatrix} \\
        &+ t' \begin{pmatrix} e^{i \phi} \left( e^{i k_2} \tilde{\psi}_m^A(k) + \tilde{\psi}_{m-1}^A(k) + e^{- i k_2} \tilde{\psi}_{m+1}^A(k) \right) + e^{- i \phi} \left( e^{- i k_2} \tilde{\psi}_m^A(k) + \tilde{\psi}_{m+1}^A(k) + e^{i k_2} \tilde{\psi}^A_{m-1}(k) \right) \\ e^{i \phi} \left( e^{- i k_2} \tilde{\psi}^B_m(k) + \tilde{\psi}_{m+1}^B(k) + e^{i k_2} \tilde{\psi}^B_{m-1}(k) \right) + e^{- i \phi} \left( e^{i k_2} \tilde{\psi}^B_m(k) + \tilde{\psi}^B_{m-1}(k) + e^{- i k_2} \tilde{\psi}^B_{m+1}(k) \right) \end{pmatrix},
\end{split}
\end{equation}
subject to the boundary condition $\psi_{-1}(k) = 0$.

\section*{Acknowledgements}

This work was supported by a Research Fellowship at Trinity College Cambridge (MJC), a Fondation Sciences Mathematiques de Paris Postdoctoral Fellowship at École Normale Supérieure (MJC), ARO MURI Award No. W911NF-14-0247 (ABW), and NSF DMREF Award No. 1922165 (ABW). The authors are grateful to Jacob Shapiro for helpful discussions.

\small
\bibliographystyle{plainurl}
\bibliography{top_ins}

\end{document}